\documentclass{article}
\usepackage[utf8]{inputenc}
\usepackage[utf8]{inputenc}
\usepackage{hyperref}
\usepackage{amssymb}
\usepackage{amsfonts}
\usepackage{amsthm}
\usepackage{mathtools}
\usepackage{bm}
\usepackage[T1]{fontenc}
\usepackage[thinlines]{easytable}
\usepackage{booktabs}
\usepackage{float}
\theoremstyle{definition}
\newtheorem{theorem}{Theorem}[section]
\newtheorem{lemma}[theorem]{Lemma}
\newtheorem{proposition}[theorem]{Proposition}
\newtheorem{mydef}[theorem]{Definition}

\newtheorem{corollary}[theorem]{Corollary}
\newtheorem{remark}[theorem]{Remark}
\newtheorem{example}[theorem]{Example}
\usepackage[margin=1.00in]{geometry}
\usepackage{graphicx}
\usepackage{mathrsfs}
\usepackage{authblk}
\usepackage{tikz}
\usepackage{slashed}
\usetikzlibrary{arrows}

\usepackage{caption}
\usepackage{theoremref}
\usepackage{subcaption}
\usepackage{cleveref}
\usepackage{theoremref} 
\usepackage[title]{appendix}
\usetikzlibrary{matrix,arrows,chains,calc}
\tikzset{>=latex}

\usepackage{tikz-cd}
\usepackage{lipsum}
\newtheoremstyle{break}
  {\topsep}{\topsep}%
  {\itshape}{}%
  {\bfseries}{}%
  {\newline}{}%
\theoremstyle{break}
\makeatletter \g@addto@macro\@floatboxreset\centering \makeatother

\newcommand{\R}{\mathbb{R}}
\newcommand{\C}{\mathbb{C}}

\newcommand{\mf}{\mathfrak}
\newcommand{\Ad}{\text{Ad}}
\newcommand{\ad}{\text{ad}}

\newcommand{\vol}{\text{Vol}}

\newcommand\blfootnote[1]{%
  \begingroup
  \renewcommand\thefootnote{}\footnote{#1}%
  \addtocounter{footnote}{-1}%
  \endgroup
}

\newcommand{\Cl}{\text{C}l}
\DeclarePairedDelimiter{\norm}{\lVert}{\rVert}

\newcommand*\diff{\mathop{}\!\mathrm{d}}
\newcommand{\ra}[1]{\renewcommand{\arraystretch}{#1}}

\usetikzlibrary{decorations.pathmorphing}
\usetikzlibrary{decorations.markings}
\tikzset{->-/.style={decoration={
  markings,
  mark=at position #1 with {\arrow{>}}},postaction={decorate}}}
\newcommand{\extp}{\@ifnextchar^\@extp{\@extp^{\,}}}
\def\@extp^#1{\mathop{\bigwedge\nolimits^{\!#1}}}

\title{Deformations of Asymptotically Conical $G_2$-Instantons}
\author{Joe Driscoll*}
\affil{ \small Imperial College London, U.K. Email: \href{mailto:j.driscoll@imperial.ac.uk}{j.driscoll@imperial.ac.uk}}
\date{}

\begin{document}

\maketitle
\blfootnote{ * ORCID: 0000-0002-0703-2975}
\begin{abstract}We develop the deformation theory of instantons on asymptotically conical $G_2$-manifolds, where an asymptotic connection at infinity is fixed. A spinorial approach is adopted to relate the space of  deformations to the kernel of a twisted Dirac operator on the $G_2$-manifold and to the eigenvalues of a twisted Dirac operator on the nearly K\"ahler link. This framework is then used to calculate the virtual dimension  of the moduli spaces of $G_2$-instantons on which several known examples live. One such example considered is the $G_2$-instanton of G\"unaydin-Nicolai, which lives on $\R^7.$ As an application of the deformation theory, we show how knowledge of the virtual dimension of the moduli space allows  us to prove that unobstructed connections in the moduli space are $G_2$-invariant. By classifying such connections we prove a uniqueness result for unobstructed $G_2$-instantons on the principal $G_2$-bundle over $\mathbb{R}^7.$
\end{abstract}
\tableofcontents
\section{Introduction}
Instantons are connections whose curvature satisfies a certain algebraic equation. 
In this article we study the deformation theory of instanton connections on manifolds with holonomy $G_2$ which asymptote to a cone. The physical motivation for studying instantons on $G_2$-manifolds stems from the fact that such connections automatically satisfy the Yang-Mills equation. Furthermore on a compact manifold $G_2$-instantons are absolute minima of the Yang-Mills energy functional.  Interest in these connections has grown significantly in recent years since the suggestion of Donaldson-Thomas \cite{donaldson1998gauge} and Donaldson-Segal \cite{donaldson2009gauge} that it may be possible to define invariants from their moduli spaces.    \par
The first example of a $G_2$-instanton was constructed on the principal $G_2$ bundle over $\R^7$ by G\"unaydin-Nicolai in 1995 \cite{gunaydin1995seven} and we shall refer to this example as the standard $G_2$ instanton. Other examples of instantons on non-compact manifolds have been found in more recent years. Firstly, Clarke \cite{clarke2014instantons} found a family of instanton connections on the manifold $\slashed{S}(S^3)$ which was shown by Bryant-Salamon to carry a $G_2$-metric \cite{bryant1989construction}. The other $G_2$-manifolds constructed by Bryant-Salamon  are $\Lambda^2_-(S^4)$ and $\Lambda^2_-(\C P^2)$ and examples of instantons have been constructed on these spaces by Oliveira \cite{oliveira2014monopoles}. Recently Lotay-Oliveira \cite{lotay2018su2insantons} studied the moduli spaces of  instantons on non-compact $G_2$-manifolds where the connections are required to be invariant under a group action. In particular, they found a limiting connection of Clarke's family of instantons. The important observation here is that all of these examples are defined on asymptotically conical manifolds and the  $G_2$-instantons converge to instanton connections on the nearly K\"ahler 6-manifold at infinity.\par 
Since the deformation space is the kernel of an elliptic operator we can develop a deformation theory that relies on the analytic framework  for asymptotically conical manifolds that has been developed by Lockhart-McOwen  \cite{lockhart1985elliptic} and Marshall \cite{marshal2002deformations}. These tools enable  one to determine when elliptic operators on non-compact manifolds are Fredholm, and hence a Kuranishi model is applicable to study the moduli space of solutions.  In the holonomy $G_2$-setting, this framework has been used to study the moduli space of associative and coassociative submanifolds by Lotay \cite{lotay2009asymptotically}\cite{lotay2009deformation} and the moduli space of $G_2$ structures by Karigiannis-Lotay \cite{karigiannis2012deformation}. In the gauge theory setting  Nakajima \cite{nakajima1990moduli} has used a similar analysis to study the moduli space of ASD-instantons on asymptotically locally Euclidean manifolds where a flat connection at infinity is fixed.\par
Since the asymptotic connection in our setting is a nearly K\"ahler instanton it is interesting to compare deformations of the $G_2$ instanton with deformations of this asymptotic connection. In particular there is  a projection  between the moduli spaces of these two instantons and one can try to understand the properties of this map. The deformation theory of nearly K\"ahler instantons has been developed by Charbonneau-Harland \cite{charbonneau2016deformations} and we use many of the ideas and techniques they develop to analyse the asymptotic connection. \par

This paper studies the deformation theory of $G_2$-instantons on asymptotically conical manifolds by prescribing a fixed rate of decay at infinity. We take a spinorial approach and relate deformations of the instanton to the kernel of a Dirac operator on a Hilbert space of spinors with fixed decay rate. In contrast to the compact case the index of the Dirac operator controlling the deformation theory  is not expected to be 0.
This can be seen in Section~\ref{sec:gtonac} where we  show that the dimension of the space of solutions to the linearised $G_2$-instanton equation is determined by the spectrum of a twisted Dirac operator on a nearly K\"ahler 6-manifold and apply the implicit function theorem to show the moduli space is a smooth manifold when the deformation theory is unobstructed. \par
In order to calculate the virtual dimension of a given moduli space in practice, one must be able to calculate some eigenvalues of certain twisted Dirac operators on nearly K\"ahler 6-manifolds. In Section~\ref{sec:lichnerowicz} we develop methods suitable for studying eigenvalues of Dirac operators on homogeneous nearly K\"ahler manifolds where a twisting is provided by the canonical connection. The assumptions made here are appropriate, since this framework is applicable to study every known asymptotically conical $G_2$-instanton.  We develop a suitable Lichnerowicz formula which allows us to reduce the calculation of the virtual dimension to the calculation of eigenvalues from certain Hermitian matrices.\par
Once the tools from Section~\ref{sec:lichnerowicz} have been developed, we are ready to study some examples of $G_2$-instantons. We begin this work in Section~\ref{sec:clarke}, introducing Clarke's $G_2$-instantons \cite{clarke2014instantons} before performing the necessary calculation to find the virtual dimension of the moduli space. Similar work is then carried out for the $G_2$-instanton of Lotay-Oliveira \cite{lotay2018su2insantons}, which also lives on $\R^4\times S^3,$ and certain $G_2$-instantons of Oliveira \cite{oliveira2014monopoles} on $\Lambda^2_-(\mathbb{CP}^2)$ and $\Lambda^2_- (S^4).$  \par
The final example, considered in Section~\ref{sec:stdinstsection}, is the $G_2$-instanton of G\"unaydin Nicolai \cite{gunaydin1995seven}, which lives on $\R^7.$ To perform the necessary calculation of eigenvalues we first present a formula for the (twisted) Dirac operator on $S^6$ which is, to the author's knowledge, new. 
As an application of the work to this point, we show in Section~\ref{sec:invariance} how to prove a uniqueness theorem for unobstructed connections in the moduli space. We show that such a connection is necessarily invariant under an action of $G_2$ and then use the algebraic framework of Wang's theorem \cite{wang1958invariant} to classify such connections. 

\section*{Acknowledgement}
This research was carried out at the University of Leeds, under Engineering and Physical Sciences Research Council (EPSRC) grant 1801930. I would like to thank EPSRC for providing the opportunity to carry out such a stimulating project, and to thank the university for providing the platform from which to do so.
Special thanks must go to my supervisor Derek Harland for his guidance and expertise on this project.  Finally, I thank the anonymous referee for numerous comments and suggestions which greatly improved an earlier version of this paper.
\section{Preliminaries}
\subsection{\texorpdfstring{$G_2$}{G2}-Manifolds}

A $G_2$-structure on a  7-manifold is a choice of 3-form $\varphi$ that at each point can be identified with the standard positive 3-form $\varphi_0$ on $\R^7$ which we take to be
\begin{equation}
\varphi_0=\diff x^{127}+\diff x^{347} + \diff x^{567} + \diff x^{145} + \diff x^{136} + \diff x^{235} - \diff x^{246}
\end{equation}
where $\diff x^{ijk}=\diff x^i \wedge \diff x^j \wedge \diff x^k.$

Such a 3-form $\varphi$ determines  an orientation $\text{Vol}_7$ and hence a metric $g=g_\varphi$ via the formula $(\iota_u \varphi)\wedge (\iota_v\varphi)\wedge \varphi=6g(u,v)\text{Vol}_7 $ (the notation ``positive'' 3-form refers to this associated quadratic form being positive definite).
This allows us to  define a 4-form $\psi:=*_\varphi \varphi.$ 
\begin{mydef}
A $G_2$-manifold is a 7-manifold $M$ together with a $G_2$-structure $\varphi$ such that $\nabla \varphi=0,$ where $\nabla$ is the Levi-Civita connection of the metric determined by $\varphi.$ Such a $G_2$-structure is called \emph{torsion free}.
\end{mydef}
The following result of Fernandez-Gray \cite{fernandez1982riemannian} gives an alternate characterisation of $G_2$ manifolds:
\begin{theorem}
Let $(M,\varphi)$ be a $G_2$-structure manifold. Then the following are equivalent:
\begin{enumerate}
    \item $\nabla \varphi=0$ 
    \item $\text{\normalfont Hol}(g_\varphi) \subseteq G_2$
    \item $ \diff \varphi=\diff \psi=0.$
\end{enumerate}
\end{theorem}

On any manifold with a $G_2$-structure there is a decomposition of the exterior bundles   determined by the irreducible representations of the group $G_2$ (a more detailed coverage can be found in \cite{bryant2003some} for example). We denote this splitting $\Lambda^k(T^*M)=\bigoplus_d\Lambda^k_d$ where $\Lambda^k_d(T^*M)$ is a rank $d$ vector bundle, fiberwise isomorphic to an irreducible representation of $G_2$ of dimension $d.$ The splitting is
\begin{align*}
\Lambda^1(T^*M)&=\Lambda^1_7 \\
\Lambda^2(T^*M)&=\Lambda^2_7\oplus \Lambda^2_{14} \\
\Lambda^3(T^*M)&=\langle \varphi \rangle_\R\oplus \Lambda^3_{7}\oplus \Lambda^3_{27}.
\end{align*}
Furthermore, the Hodge star operator yields isomorphic splittings $\Lambda^{7-k}(T^*M)=\bigoplus_d *\left( \Lambda^{k}_d(T^*M) \right).$
We have explicit models for these spaces as follows:
\begin{align}
\Lambda^2_7&=\{ u\lrcorner \varphi \, ; \, u \in \Lambda^1(T^*M) \} \\
\Lambda^2_{14}&=\{ \alpha \in \Lambda^2(T^*M)\, ; \, \alpha \wedge \psi=0\}=\{ \alpha \in \Lambda^2(T^*M) \, ; *(\alpha \wedge \varphi)=-\alpha\} \label{eq:g2subspace} \\
\Lambda^3_{7}&=\{ u\lrcorner \psi \, ; \, u \in \Lambda^1(T^*M) \} \\
\Lambda^3_{27}&=\{ \eta \in \Lambda^3(T^*M) \, ; \, \eta \wedge \varphi = \eta \wedge \psi =0\}.
\end{align}
We call the seven dimensional irreducible representation the standard representation and denote it $V.$ The 14-dimensional irreducible representation  is isomorphic to the adjoint representation $\mf{g}_2$ and the 27-dimensional representation is isomorphic to $\text{Sym}^2_0(V)$ \cite{bryant2003some}.\par
Since $G_2$ is simply connected, manifolds with a  $G_2$-structure are spin. The spin bundle is constructed from an irreducible representation of Spin(7) arising by restricting  a representation of the Clifford algebra $\Cl(\R^7)\cong\text{Mat}_{\R}(8)\oplus\text{Mat}_{\R}(8)$. There are two  choices of representation  $W^+$ and $W^-,$they are both $8$ dimensional and  distinguished by the fact that the volume form $\text{Vol}_7$ acts as $\pm1$ on $W^\pm.$ The resulting spin bundle is independent of this choice \cite{lawson2016spin}. We make that choice that the volume form acts as $+1$ since this will ensure our formula for Clifford multiplication is the standard one in the literature. Since $G_2$-manifolds are Ricci flat, they come with a unique non-zero parallel spinor which we denote $s_7.$ The stabiliser of $s_7$ is $G_2$ and this leads to another description 
$$\Lambda^2_{14}=\{\alpha \in \Lambda^2(T^*M) \, ; \, \alpha \cdot s_7=0\}$$
of the bundle $\Lambda^2_{14}$ corresponding to the adjoint representation of $\mf{g}_2.$
The map $\Lambda^0\oplus \Lambda^1 \to \slashed{S}(M), (f+v)\mapsto (f+v)\cdot s_7$ is an isomorphism \cite{agricola2015spinorial} so that $\slashed{S}(M)\cong \Lambda^0\oplus \Lambda^1.$
\begin{lemma}\thlabel{evalsg2}
The 3-form $\varphi$ and 4-form $\psi$ act with the following eigenvalues on the subspaces of $\slashed{S}(M)$:
\begin{center}
    \begin{tabular}{c|cc}
    & $\Lambda^0$ & $\Lambda^1$  \\ \hline
    $\varphi $& -7 & 1 \\
    $\psi$ & -7 & 1
    
    \end{tabular}.
\end{center}
\end{lemma}
\begin{proof}
Since $\varphi$ is $G_2$-invariant Schur's lemma says that it preserves this decomposition since $\Lambda^0$ and $\Lambda^1$ are irreducible representations of $G_2.$ Furthermore it must act as a constant on each space and the action is traceless.
We first consider the case $\Lambda^0$: We have that $\psi=*\varphi=\varphi\cdot \text{Vol}_7$ and a direct calculation shows that $\varphi \cdot \psi=7\text{Vol}_7-6\varphi.$ Let $\varphi\cdot s_7=\lambda s_7$ then we find 
\begin{align*}
    \varphi\cdot \psi \cdot s_7&=\varphi^2\cdot \text{Vol}_7 \cdot s_7=\lambda^2s_7 \\
    &= (7\text{Vol}_7 - 6\varphi)\cdot s_7 =(7-6\lambda)s_7.
\end{align*}
Therefore $\lambda=-7$ or $\lambda=1.$ The eigenvalues of $\varphi$ acting on $\Lambda^1$ satisfy the same equation and since $\varphi$ is traceless we must have  that $\varphi$ acts as $-7$ on $\Lambda^0$ and as $1$ on $\Lambda^1.$
\end{proof}
\begin{remark}
Had we instead chosen $\text{Vol}_7$ to act as -1, the eigenvalues of $\varphi$ would differ from those above by a minus sign, while the eigenvalues of $\psi$ are independent of this choice.
\end{remark}
An argument similar to those of \cite{lotay2009deformation} and \cite{charbonneau2016deformations} yields the following corollary:
\begin{corollary}
Let $\alpha \in \Omega^2(M),$ then 
$$ \alpha \cdot s_7=*(\alpha \wedge \psi)\cdot s_7.$$

\end{corollary}
\begin{proof}
Since the $\Lambda^2_{14}$ component of $\alpha$ annihilates $s_7$ we have that 
$$ \alpha \cdot s_7=\pi_7(\alpha) \cdot s_7.$$
Now $\pi_7(\alpha)=v\lrcorner \varphi$ for some $v\in \Omega^1(M)$ so \thref{evalsg2} says
\begin{align*}
\alpha \cdot s_7&= (v\lrcorner \varphi )\cdot s_7 = -\frac{1}{2}\{v,\varphi\}\cdot s_7 \\
&= 3v\cdot s_7.
\end{align*}
To find $v,$ note that for $v\in\Lambda^1$ we have $*(*(v\wedge \psi)\wedge \psi)=3v$ (see \cite{bryant2003some} for details)  and thus we can calculate
\begin{align*}
    *(\alpha \wedge \psi)&=*(\pi_7(\alpha)\wedge \psi)=*((v\lrcorner \varphi)\wedge \psi) \\
    &=3v
\end{align*}
so that $v=\frac{1}{3}*(\alpha \wedge \psi)$ and the result follows. 
\end{proof}
\begin{corollary}
Let $f\in\Omega^0(M)$ and $u,v \in \Omega^1(M).$ Then Clifford multiplication of the spinor $(f+v)\cdot s_7$ by $u$ is 
$$ \text{\normalfont cl} (u)  (f+v) \cdot s_7=(- u\lrcorner v  +fu +*(u\wedge v \wedge \psi))\cdot s_7.$$
\end{corollary}
Recall the Dirac operator $D\colon \Gamma(\slashed{S}(M))\to \Gamma(\slashed{S}(M))$ is given in a local orthonormal frame $e^i$ of $T^*M$ by the formula $D(s)=e^i\cdot \nabla_is.$ It is easily verified that 
$$D((f+v)\cdot s)= (\diff^*v+\diff f+*(\diff v\wedge \psi))\cdot s_7$$
so we can write the Dirac operator as the $2\times 2$ matrix 
\begin{equation}\label{eq:diracop7}
   D=\begin{pmatrix}
    0&\diff^* \\
    \diff &*(\psi\wedge \diff \, \cdot )
    \end{pmatrix}.
\end{equation}
\subsection{Nearly K\"ahler Manifolds}

Let $(\Sigma,g)$ be a Riemannian spin manifold and let $\slashed{S}(\Sigma)$ denote the real spinor bundle associated to $\Sigma.$ A spinor $s \in \Gamma(\slashed{S}(\Sigma))$ is called a \emph{real Killing spinor} if there exists a non-zero real constant $\lambda$ such that 
\begin{equation}\label{eq:killingeqn}
\nabla_Xs=\lambda X\cdot s
\end{equation}
for all $X\in\Gamma(T\Sigma)$ and where $\nabla $ is the Levi-Civita connection acting on the spin bundle. From here on we fix $\Sigma$ to be  6-dimensional.\\
\begin{mydef}
A $6$-manifold $(\Sigma,g)$ together with a real Killing spinor $s_6\in \Gamma(\slashed{S}(\Sigma))$ is called a \emph{nearly K\"ahler $6$-manifold}.
\end{mydef}
By scaling the metric, we can also scale the constant $\lambda.$ In order that the Killing spinor lifts to a parallel spinor on the cone we shall fix 
$$\lambda =\frac{1}{2}.$$
Such a manifold is Einstein with $\text{Ric}=5g$ \cite{baum1990twistor}. The group fixing a Killing spinor on a 6-manifold is $ \text{SU}(3) $ and the spinor defines an $ \text{SU}(3) $-structure on $\Sigma$ as described in \cite{charbonneau2016deformations}.
Thus we have a holomorphic $(3,0)$-form $\Omega$ and an almost complex structure $J$ which allows us to define a fundamental 2-form $\omega=g(J\cdot, \cdot). $  In a suitable local orthonormal frame these take the form
\begin{align*}
    \Omega&=(e^1+ie^2)\wedge (e^3+ie^4)\wedge (e^5+ie^6) \\
    \omega &= e^{12} + e^{34} + e^{56}. 
\end{align*}
On a nearly K\"ahler manifold $\omega$ is not closed but the following equations are satisfied:
\begin{equation}\label{eq:nkforms}
\diff \omega =3\text{Im}\Omega, \qquad \diff \text{Re}\Omega=2\omega^2.
\end{equation}
In fact, a 6-manifold is nearly K\"ahler if and only if it is an $ \text{SU}(3) $-structure manifold such that \eqref{eq:nkforms} holds \cite{friedrich1985first}. The $ \text{SU}(3) $-structure has non-vanishing torsion and therefore the holonomy group of the Levi-Civita connection $\nabla$ need not be an $ \text{SU}(3) $ subgroup. There is however a distinguished connection on the tangent bundle with skew parallel torsion and holonomy $ \text{SU}(3) .$
This connection is known as the \emph{canonical connection} and is defined via the formula 
\begin{equation}\label{eq:canonicalconnectiontm}
    g(\nabla^{\text{can}}_X,Y,Z)=g(\nabla_XY,Z) + \frac{1}{2}\text{Re}\Omega(X,Y,Z)
\end{equation}
for all $X,Y,Z\in\Gamma(T\Sigma).$
It proves useful to define a one parameter family of connections interpolating between the Levi-Civita connection and the canonical connection by setting
\begin{equation}
    g(\nabla^{t}_XY,Z)=g(\nabla_XY,Z) + \frac{t}{2}\text{Re}\Omega(X,Y,Z)
\end{equation}
for $t\in\R.$
The torsion tensor $T^t$ of the connection $\nabla^t$ is 
$$g(X,T^t(Y,Z))=t\text{Re}\Omega(X,Y,Z).$$
The fact that the holonomy group of the canonical connection $\nabla^{\text{can}}=\nabla^1$ is a subgroup of $ \text{SU}(3) $ follows from the equation
$$\nabla^t_Xs_6=\frac{1-t}{2}X\cdot s_6.$$\par
The first examples of nearly K\"ahler manifold were given in \cite{wolf1968homogeneous1} whilst in recent years the first compact, non-homogeneous examlpes were gievn in  \cite{foscolo2017exoticnk}.  For the purpose of this article, the most important examples of nearly K\"ahler manifolds will be the homogeneous ones. There are precisely four such manifolds 
\begin{align*}
    S^6&=G_2/ \text{SU}(3) , \qquad &S^3\times S^3=\text{SU}(2)^3/\text{SU}(2) \\
    \C \mathbb{P}^3&=\text{Sp}(2)/\text{Sp}(1)\times \text{U}(1), \qquad &\mathbb{F}_{1,2,3}= \text{SU}(3) /\text{U}(1)^2.
\end{align*}
In each case the homogeneous space $G/H$ is \emph{reductive}, meaning there is a splitting $\mf{g}=\mf{h}\oplus \mf{m}$ with $\mf{m}$ closed under the adjoint action of $H.$ These coset spaces are called \emph{3-symmetric} since in each case the subgroup $H$ is fixed by an automorphism $s$ of $G$ satisfying $s^3=\text{Id}.$ The induced Lie algebra automorphism $S$ also satisfies $S^3=\text{Id}.$ This acts trivially on $\mf{h}$ and non-trivially on $\mf{m}$; one defines a almost complex structure $J\colon \mf{m}\to \mf{m}$ via 
\begin{equation}\label{eq:3sym}
    S|_{\mf{m}}=-\frac{1}{2} + \frac{\sqrt{3}}{2}J.
\end{equation}
The Riemannian metric on each space is determined by the Killing form on $\mf{g}.$ In \cite{moroianu2010hermitian} it is shown that $-\frac{1}{12}$ is the normalisation that yields $\lambda=\frac{1}{2}$ in the Killing spinor equation \eqref{eq:killingeqn}. So the metric is induced from the bilinear form 
\begin{equation}\label{eq:killingform}
B(X,Y)=-\frac{1}{12}\text{Tr}_{\mf{g}_2}(\ad(X)\ad(Y)) \qquad \forall X, Y \in \mf{g}_2.
\end{equation}
We will refer to \eqref{eq:killingform} as the \emph{ nearly K\"ahler metric  } on $\mf{g}$. Extending this metric by left translation furnishes $G/H$ with a $G$-invariant metric. \par 
The tensors defining the $\text{SU}(3)$-structure admit local expressions in terms of the Lie algebra data. We record here only the identities required for this article and refer the reader to \cite{harland2010yang} for further details:  Let $\{I_A\}$ be a basis for $\mf{g}$ such that $I_a$ for $1\leq a \leq 6$ forms a basis for $\mf{m}$ and $I_i$ for $7\leq i \leq \text{dim}(G)$ forms a basis of $\mf{h}.$ The choice of orthonormal basis for $\mf{m}$ induces a local orthonormal frame $e^a$ which trivialises $T^*(G/H)|_U$ and we can write $e^i=e^i_ae^a$ with locally defined smooth functions $e^i_a.$  Let $f_{AB}^C$ be the structure constants defined by $[I_A,I_B]=f_{AB}^CI_C$ and use the nearly K\"ahler metric to lower an index $f_{ABC}\coloneqq f_{AB}^D\delta_{DC}$. The  local coordinate expression we shall require is that of the $3$-form $\text{Re}\Omega$ and this is
\begin{equation}\label{eq:holformincoords}
    \text{Re}\Omega = -\frac{1}{6}f_{abc}e^a \wedge e^b \wedge e^c.
\end{equation}

Reductive homogeneous spaces come with a distinguished connection, also called the canonical connection, on the principal $H$ bundle $G\to G/H$ whose horizontal distribution is given by left translation of $\mf{m}.$  The tangent bundle $T(G/H)$ is associated to   $G\to G/H$  via the representation $\mf{m}$ of $H.$ The canonical connection coming from the reductive homogeneous structure therefore defines a connection on $T(G/H)$ and this agrees with the connection  \eqref{eq:canonicalconnectiontm}, justifying the nomenclature.

Analogously to the $G_2$ case the exterior bundles split according to how the fibres split as representations of $ \text{SU}(3) .$ The splitting $\Lambda^k(T^*M)=\bigoplus_d\Lambda^k_d$, where $\Lambda^k_d$ has fibre dimension $d$, is as follows:
\begin{align}
    \Lambda^1(T^*\Sigma)&=\Lambda^1_6 \\
    \Lambda^2(T^*\Sigma)&=\langle \omega\rangle_\R \oplus \Lambda^2_6\oplus \Lambda^2_8 \\
    \Lambda^3(T^*\Sigma)&= \langle\text{Re}\Omega \rangle_\R\oplus \langle \text{Im}\Omega\rangle_\R \oplus \Lambda^3_6\oplus \Lambda^3_{12}
\end{align}
and there are isometric splittings $\Lambda^{6-k}=\bigoplus_d *(\Lambda^k_d).$ These spaces are modelled as follows:
\begin{align}
 \\
    \Lambda^2_6&=\text{Re}\left(\Lambda^{2,0}\oplus \Lambda^{0,2}\right)=\{v\lrcorner \text{Re}(\Omega)\, ; \, v\in\Lambda^1\} \label{eq:lambda^2_6} \\
    \Lambda^2_8&=\{\alpha \in \Lambda^2 \, ; \, *(\alpha \wedge \omega )=-\alpha\} \label{eq:su3space} \\
    \Lambda^3_6&=\{ v\wedge \omega \, ;\, v \in \Lambda^1 \} \\
    \Lambda^3_{12}&=\text{Re}\{\gamma \in \Lambda^{2,1}\oplus \Lambda^{1,2} \, ; \, \gamma \wedge \omega =0 \}.
\end{align}
The bundle $\Lambda^2_8$ has fibres isomorphic to $\mf{su}(3).$ Since the action of the group $ \text{SU}(3) $ fixes the Killing spinor it is clear that the Lie algebra action annihilates it, hence we have $\alpha \cdot s_6=0$ for any $\alpha \in \Lambda^2_8.$
The Killing spinor $s_6$ determines  a bundle map from $\Lambda^0(T^*\Sigma)\oplus \Lambda^1(T^*\Sigma)\oplus \Lambda^6(T^*M)$ to $\slashed{S}(\Sigma),$ defined by $\eta \mapsto \eta \cdot s_6$ and Charbonneau-Harland \cite{charbonneau2016deformations} note this is  an isomorphism:
\begin{equation}\label{eq:nkspinbundle}
    \slashed{S}(\Sigma)\cong \Lambda^0(T^*\Sigma)\oplus \Lambda^1(T^*\Sigma)\oplus \Lambda^6(T^*M).
\end{equation}
The almost complex structure can be constructed from this splitting; let $\text{Vol}_6$ be the volume element of the Clifford algebra, then one defines $\text{Vol}_6\cdot v\cdot s_6= Jv\cdot s_6$ for any $v\in\Lambda^1.$
The forms $\text{Re}\Omega $ and $*\omega$ act as scalar multiples on the summands of this splitting:
\begin{lemma}[{\cite[Lemma 2]{charbonneau2016deformations}}]\thlabel{eigenvals}
The subspaces of $\slashed{S}(\Sigma)$ isomorphic to $\Lambda^0, \Lambda^1$ and $\Lambda^6$ are eigenspaces of the operations of Clifford multiplication by $\text{\normalfont Re}\Omega$ and $*\omega$ with the following eigenvalues 
\begin{center}
\begin{tabular}{c|ccc}
 & $\Lambda^0$ & $\Lambda^1 $&$\Lambda^6 $ \\ \hline
 $\text{\normalfont Re}\Omega$ & 4&0&-4 \\
 $*\omega$ &-3&1&-3.
 \end{tabular}
 \end{center}
\end{lemma}

We would like to understand Clifford multiplication by 1-forms under this splitting of the spin bundle. We begin with a lemma:
\begin{lemma}
For any $\alpha \in\Omega^2(\Sigma)$ we have that 
\begin{equation}
    \alpha \cdot s_6= (-(\alpha\lrcorner\omega)\text{\normalfont Vol}_6  - \alpha \lrcorner \text{\normalfont Re}\Omega )\cdot s_6.
\end{equation}
\end{lemma}
\begin{proof}
Let $\pi_d$ denote projection from $\Lambda^2$ to the subspace $\Lambda^2_d.$ The description \eqref{eq:su3space} shows that $\pi_8(\alpha)\cdot s_6=0,$ so that 
$$ \alpha \cdot s_6=(\pi_1(\alpha)+\pi_6(\alpha))\cdot s_6 .$$
The $\pi_1$ component of $\alpha$ is a multiple of $\omega$ so we calculate
$$\omega \cdot s_6 = \text{Vol}_6\cdot *\omega \cdot s = -3\text{Vol}_6\cdot s_6$$
by \thref{eigenvals}. Since $|\omega|^2=3$ we see that $\pi_1(\alpha)=\frac{1}{3}(\alpha \lrcorner \omega)\omega$ and thus 
$$\pi_1(\alpha)\cdot s_6 = (\alpha \lrcorner \omega) \text{Vol}_6 \cdot s_6. $$
We have that$\pi_6(\alpha)=v\lrcorner \text{Re} \Omega $ for a unique $v\in\Lambda^1$ and we find
$$v\lrcorner \text{Re}\Omega \cdot s_6 = -\frac{1}{2}\{v,\text{Re}\Omega\}\cdot s_6=-2v\cdot s_6$$
again by \thref{eigenvals}. An application of Schur's lemma shows that $v=\frac{1}{2}\alpha \lrcorner \text{Re}\Omega.$
Therefore 
$$ \pi_6(\alpha)\cdot s_6=-(\alpha \lrcorner \text{Re}\Omega ) \cdot s_6.$$
Overall we see that 
$$\alpha \cdot s_6=(-(\alpha \lrcorner  \omega )\text{Vol}_6 - (\alpha \lrcorner \text{Re}\Omega))\cdot s_6.$$
\end{proof}
\begin{corollary}\thlabel{nkclif}
Clifford multiplication of a spinor  $(f+v+h\text{\normalfont Vol}_6)\cdot s_6$  by  a 1-form $u\in\Omega^1(\Sigma)$ is given by 
\begin{equation}\label{eq:nkclifford}
\text{\normalfont} (u)(f+v+h\text{\normalfont Vol}_6)\cdot s_6=(- u\lrcorner v  + fu -hJu -(u \wedge v)\lrcorner \text{\normalfont Re}\Omega-    (u \lrcorner Jv)\text{\normalfont Vol}_6 )\cdot s_6. 
\end{equation}
\end{corollary}
\begin{proof}
Since the volume form anti-commutes with 1-forms we have that 
$$\text{cl} (u) (f+v+h\text{Vol})\cdot s_6= (fu+u\wedge v - u \lrcorner  v  -hJu)\cdot s_6$$
and the term $(u\wedge v)\cdot s_6$ is handled using the previous lemma and the identity 
$ u\lrcorner  Jv =( u \wedge v)\lrcorner \omega.  $
\end{proof}
We can apply these results to understand the Dirac operator on a nearly K\"ahler manifold. Let $D^t=\text{cl}\circ\nabla^t\colon \Gamma (\slashed{S}(\Sigma))\to \Gamma(\slashed{S}(\Sigma))$ be the Dirac operator constructed from the connection $\nabla^t$ acting on the spin bundle.
Charbonneau-Harland \cite{charbonneau2016deformations} show that  these operators differ by a multiple of the action of $\text{Re}\Omega$ on the spin bundle:
\begin{equation}\label{eq:diracdifference}
    D^t=D^0+\frac{3t}{4}\text{Re}\Omega
\end{equation}
where $D^0$ is the Levi-Civita Dirac operator. Note that, since $\text{Re}\Omega\cdot v\cdot s_6=0$ for any 1-form $v$, this family of operators have the same action on 1-forms, namely
$$D^t(v\cdot s_6)=(\diff v+\diff^*v+2v)\cdot s_6.$$

\begin{corollary}
Under the splitting $\slashed{S}(\Sigma)\cong \Lambda^0\oplus \Lambda^1 \oplus \Lambda^6$ the Dirac operator is 
\begin{equation}\label{eq:diracop6}
  D^t=\begin{pmatrix}
    -3+3t&\diff^*&0 \\
    \diff & 2-(\diff \, \cdot)\lrcorner \text{\normalfont Re}\Omega & J\diff* \\
    0&*\langle \diff \, \cdot,  \omega \rangle &3-3t
    \end{pmatrix}.
\end{equation}
\end{corollary}

\section{Asymptotically Conical \texorpdfstring{$G_2$}{G2}-Manifolds}

The link between Nearly K\"ahler and  $G_2$-geometry is via the cone construction, which we now describe:
\begin{mydef}
 Let $(\Sigma^6,g_6)$ be nearly K\"ahler. The $G_2$ cone of $\Sigma$ is   $C(\Sigma)=(0,\infty)\times \Sigma$ together with the torsion-free $G_2$-structure $(C(\Sigma), \varphi_C)$ defined by
\begin{align*}
    \varphi_C&=r^2\omega \wedge \diff r + r^3\text{\normalfont Im}\Omega 
\end{align*}
where $r$ is the coordinate on $(0,\infty).$ 
The metric determined by  $\varphi_C$ is the cone metric $g_C=\diff r^2 + r^2g_6$.
We choose the orientation such that $\diff r\wedge r^6V_6$ is the volume form on $C.$ We call $C$ a $G_2$ cone and $\Sigma$ the link of the cone.  Finally, we denote the natural projection map $\pi\colon C(\Sigma)\to \Sigma.$
\end{mydef}
A $G_2$ cone is of course not complete, but we will consider complete $G_2$-manifolds whose geometry is \emph{asymptotically} that of a $G_2$ cone. The next definition makes this notion precise.
\begin{mydef}\label{defac}
Let $(M, g,\varphi)$ be a non-compact $G_2$-manifold.
We call $M$ an asymptotically conical (AC) $G_2$-manifold  with rate $\mu<0$ if there exists a compact subset $K \subset M,$ a compact, connected nearly K\"ahler $6$-manifold $\Sigma$, a constant $R>1$ and a diffeomorphism 
$$ h \colon (R, \infty ) \times \Sigma \to  M \setminus K $$ such that 
\begin{equation}\label{eq:convergence}
\vert \nabla_C^j\left( h^*(\varphi|_{M\setminus K}) - \varphi_C \right) \vert(r,\sigma) = O(r^{\mu -j}) \qquad \text{ as $r \to \infty$}
\end{equation}
for each $\sigma\in\Sigma,$ for $j=0,1, 2 ,\ldots,$  where $\nabla_C$ is the Levi-Civita connection for the cone metric $g_C$ on $C(\Sigma)$, $\varphi_C$ is the $G_2$-structure on the cone and $\vert \cdot \vert$ is calculated using $g_C.$  We call $M\setminus K$ the end of $M$ and $\Sigma$ the asymptotic link of $M$.

\end{mydef}

\begin{remark}
We shall often drop the notation showing the dependence of a norm of the form in \eqref{eq:convergence} on a point in $\Sigma$. Such a norm is always to be understood pointwise in $\Sigma.$ 
\end{remark}
It follows from Taylor's theorem that the metric $g=g_\varphi$ and $\psi=*_\varphi \varphi$ satisfy the same asymptotic condition
\begin{align}
    \vert \nabla_C^j\left( h^*(g|_{M\setminus K}) - g_C \right) \vert  &= O(r^{\mu -j}) \qquad \text{ as $r \to \infty$} \\
    \vert \nabla_C^j\left( h^*(\psi|_{M\setminus K}) - \psi_C \right) \vert &= O(r^{\mu -j}) \qquad \text{ as $r \to \infty$}.
\end{align}

\begin{example}
\normalfont
\begin{enumerate}
    \item $(\R^7,\varphi_0)$ is clearly an AC $G_2$-manifold with any rate $\mu<0$ since $C(S^6)=\R^7\setminus \{0\}.$ In our convention $\varphi_C=\varphi_0$ and the induced metric is the Euclidean metric.
    \item The total space of $\slashed{S}(S^3),$ the spin bundle of $S^3,$ carries a metric of $G_2$ holonomy. This was the first non-trivial complete example of a $G_2$-manifold to be found and is due to the work of Bryant and Salamon in 1989 \cite{bryant1985metrics}. This example has cohomegeneity one, i.e there's a group acting isometrically with generic orbits of codimension one. This symmetry allows one to reduce the problem of finding a torsion free $G_2$-structure to solving an ODE. Furthermore the $G_2$-structure can be shown to be AC with asymptotic link $\Sigma=S^3\times S^3$ and rate -3. Bryant-Salamon also find $G_2$-holonomy metrics on $\Lambda^2_{-}S^4$  and $\Lambda^2_{-}\mathbb{CP}^2$ using this method. Here the asymptotic links are $\mathbb{CP}^3$ and $F_{1,2,3}$ respectively   and the rate of converge is -4 in both cases.  \item Foscolo, Haskins and Nordstr\"om find in \cite{foscolo2018infinitely}  infinitely many new diffeomorphism types of AC $G_2$-manifolds with asymptotic link (a quotient of) $S^3 \times S^3.$
  
\end{enumerate}
\end{example}

\subsection{Analysis on AC Manifolds}
For the remainder of this article, unless stated otherwise, $M$ will be an AC $G_2$-manifold with asymptotic link  $\Sigma$ and $h$ the diffeomorphism identifying the end of $M$ with the cone on $\Sigma.$ 
Weighted Sobolev spaces provide a natural setting for analysis on AC manifolds. Let us denote by $M_t$ for any $t>R$ the subspace of $M$ given by $h((t,\infty)\times \Sigma).$
\begin{mydef}
A radius function $\rho$ on $M$ is a smooth function on $M$ that satisfies the following
conditions. On the compact subset $K$ of $M$, we define $\rho = 1$. Let $x$ be a point in $M_{2R},$
then $h^{-1}(x)=(r,y)$ for some $r \in (2R, \infty)$ and we define $\rho(x) = r$ for such a point. Finally, in the
region $h((R, 2R) \times \Sigma)$, the function $\rho$ is defined by interpolating smoothly between its definition
near infinity and its definition in the compact subset $K$, in a decreasing fashion.
\end{mydef}
Given a principal $G$ bundle $P\to M $ we denote by $\Ad P$ the adjoint bundle of $P$, this is the vector bundle associated to $P$ via the adjoint representation of $G$ (hence the fibers are copies of the Lie algebra which we denote $\mf{g}$). We assume that $G$ is compact and  semi-simple and define a metric $-\langle ,\rangle_\mf{g}$ on the fibers using the negative of the Killing form on $\mf{g}.$ Furthermore a choice of connection $A$ on $P$ gives rise to a linear connection acting on sections of $\Ad P$ which we shall denote $\nabla^A.$ This is extended in the usual way to $\Omega^*(M, \Ad P)$ and for a section $f\in \Gamma(\Ad P)$ we use the notation $\nabla^A f =\diff_A f$ interchangeably.\par  
 We need to construct spaces of sections suitable for our purposes. For this let $E\to M$ be a vector bundle with a bundle metric  and denote by the subscript loc the space of sections $\eta$ such that $f\eta $ lie in the desired space  for all compactly supported functions $f$ on $M.$ For example, $C^k_\text{loc}(E)$ denotes the space of sections that in this sense lie in $C^k(E).$
\begin{mydef}
Let $p\geq1, k \in \mathbb{N}\cup \{0\}$ and $\mu\in \R.$ Let $(M,\varphi)$ be an AC $G_2$-manifold and fix a radius function $\rho$. Let $P\to M$ be a principal bundle with connection $A$ and let $T$ be either a tensor or spinor bundle, so that $T$ inherits a metric and Levi-Civita connection. Define a norm $\norm{\cdot}_{L^p_{k,\mu}}$ on $L^p_{k,\text{loc}}$ sections $\eta $ of $T\otimes \text{\normalfont Ad}P$ by defining
$$\norm{\eta}_{L^p_{k,\mu}} = \left( \sum_{j=0}^k \int_M \left| \rho^{j-\mu}\nabla^j\eta \right| ^p \rho^{-7}\diff V_g\right) ^{\frac{1}{p}}$$
where $|\cdot |$ is calculated using a combination of $g$ and $\langle,\rangle_\mf{g}$ and $\nabla=\nabla^{LC}\otimes\nabla^A$ is the tensor product of the Levi-Civita connection on $T$ and the connection $\nabla^A$ on $\text{\normalfont Ad}P.$ 
We  let $L^p_{k,\mu}(T\otimes \text{\normalfont Ad}P)$ denote the completion under this norm.
\end{mydef}
The weighted Sobolev spaces are Banach spaces and, when $p=2,$ Hilbert spaces.
An element $\eta\in L^p_{k,\mu}(E)$ can be thought of as a section that is $k$ times weakly differentiable such that the derivative $\nabla^j\eta$ is growing at most like $r^{\mu-j}$ on the end of $M.$ Indeed if $|\eta|=O(\rho^\mu)$ on the end of $M$ then $\eta \in L^p_{0,\mu+\epsilon}(E)$ for any $\epsilon >0.$

Denote $L^p_\mu=L^p_{0,\mu},$ then we have a weighted H\"older inequality \cite{bartnik1986mass}
$$ \norm{\xi\otimes\eta}_{L^p_{\mu+\nu}} \leq \norm{\xi}_{L^q_\mu}\norm{\eta}_{L^{q^\prime}_\nu}$$
where $\frac{1}{p}=\frac{1}{q}+\frac{1}{q^\prime}.$ 
As in the familiar $L^p$ case, this gives a duality pairing provided we use the correct weight $(L^p_\mu(E))^*\cong L^q_{-7-\mu}(E)$ where $\frac{1}{p}+\frac{1}{q}=1.$
We will mostly work with the Hilbert spaces $L^2_{k,\mu}(E).$ For such spaces one has that $\left( L^2_{k,\mu}(E)\right)^*\cong L^2_{-k,-7-\mu}(E),$ but in practice we will only ever be interested in the kernel of an operator with regularity properties that ensure the kernel is in fact independent of $k.$ In this way we can always work with Sobolev spaces with a positive degree of differentiability. \par
 We will also require weighted $C^k$ and $C^\infty$ spaces. 
 \begin{mydef}
Let  $\mu\in \R $ and let $k\in \mathbb{N} \cup \{0\}.$ The weighted $C^k$ space $C^k_\mu(E)$ is the subspace of $C^k_\text{loc}(E)$ such that the norm 
 $$\norm{\eta}_{C^k_\mu}= \sum_{j=0}^k\sup_M | \rho^{j-\mu}\nabla^j \eta | $$
 is finite. We also define $C^\infty_\mu(E)=\cap_{k\geq 0}C^k_\mu(E).$ The spaces $C^k_\mu$ are Banach spaces but $C^\infty_\mu$ need not be.
 \end{mydef}

The usual embedding theorems for Sobolev spaces can be adapted to the weighted case, we state here only the results needed for the purposes of this article. The  theorem is stated  using $\mu$-weighted H\"older spaces $C^{l,\alpha}_\mu(E)$, where $\alpha \in (0,1)$ is the H\"older exponent. These are spaces of sections with $l$ continuous derivatives and controlled growth on the end of $M$. The definition of a weighted Hölder space is stated here for completeness but we shall mostly choose  to work with Sobolev spaces.
\begin{mydef}
Let $\alpha\in (0,1),$ let $k\in 
\mathbb{N}$ and let $\mu \in \R.$ Let $d(x,y)$ be the geodesic distance between points $x,y \in M,$ let $0<c_1<1<c_2$ be constants and let 

$$H=\{(x,y)\in M \times M\, ; \, x\neq y, c_1\rho(x)\leq \rho(y)\leq c_2\rho(x) \text{ and the exists a geodesic in $M$ of length $d(x,y)$ from $x$ to $y$}\}.$$
A section $\eta$ of a vector bundle $E$, endowed with a fibre metric and a metric connection $\nabla$ is called \emph{Hölder continuous with exponent }$\alpha $ if 

$$[\eta]^\alpha = \sup_{(x,y)\in H}\frac{|\eta(x)-\eta(y)|_E}{d(x,y)^\alpha }<\infty.$$
Here $\nabla$ is used to identify the fibers $E_x$ and $E_y$ via parallel transport along a geodesic $\gamma$ connecting $x$ and $y$ (note we can find such a geodesic since $(x,y)\in H$).\par
The \emph{weighted Hölder space} $C^{k,\alpha }_\mu(E)$ of sections of $E$ is the subspace of $C^{k,a}_\text{loc}(E)$ such that the norm 

$$\norm{\eta}_{C^{k,\alpha}_\mu}=\norm{\eta}_{C^k_\mu}+[\eta]^{k,\alpha }_\mu $$
is finite, where 
$$[\eta]^{k,\alpha }_\mu= [\rho^{k+\alpha -\mu}\nabla^k\eta]^\alpha.$$

\end{mydef}
The spaces $C^{k,\alpha}_\mu(E)$ is a Banach space and there are embeddings $C^{k,\alpha }_\mu(E) \to C^l_\mu(E)$ whenever $l\leq k.$ With this definition in hand we can state the Sobolev embedding theorem we require:
\begin{theorem}[Weighted Sobolev Embedding Theorem]\thlabel{embed}
Let $k,l\geq 0$ and let $\alpha \in (0,1).$
\begin{enumerate}
    \item 
If $k - \frac{7}{p}\geq l+\alpha $ then there is a continuous embedding $L^p_{k,\mu}(E) \hookrightarrow C^{l,\alpha}_\mu(E).$ \\
\item 
If $k \geq l \geq 0 $ and $k-\frac{7}{p} \geq l-\frac{7}{q},$  $p\leq q $ and $\mu \leq \nu$ then there is a continuous embedding $L^p_{k, \mu}(E) \hookrightarrow L^q_{l,\nu}(E).$
\end{enumerate}
\end{theorem}
Throughout this article we choose to work with the spaces $L^2_{k,\mu}$ for $k\geq 4$ so that the sections we consider are continuous. We can use the weighted embedding theorem to form a multiplication theorem, this is done by adapting the argument of \cite{choquet1982analysis} to the weighted setting. 
\begin{theorem}[Weighted Sobolev Multiplication Theorem]\thlabel{weightedmult}
Let $\xi \in L^2_{k, \mu}(E)$, $\eta \in L^2_{l,\nu}(F)$ and suppose $l\geq k.$ If $k>\frac{7}{2}$ then multiplication $L^2_{k,\mu}(E)\times L^2_{l,\nu}(F)\hookrightarrow L^2_{k,\mu+\nu}(E\otimes F) $ is bounded, i.e there exists a constant $C>0$ such that
$$\norm{\xi \otimes \eta}_{k,\mu+\nu}\leq C \norm{\xi}_{k,\mu}\norm{\eta}_{l,\nu}.$$

\end{theorem}

For convenience we shall adopt the following notation:
$$\begin{array}{ccc}
     \Omega^m_{\mu}(M)\coloneqq C^\infty_\mu(\Lambda^m T^*M),& &\Omega^m_{k,\mu}(M)\coloneqq L^2_{k,\mu}(\Lambda^mT^*M)  \\
     \Omega^m_{\mu}(M,\Ad P)\coloneqq C^\infty_\mu(T^*M\otimes \Ad P),& & \Omega^m_{k,\mu}(M,\Ad P)\coloneqq L^2_{k,\mu}(\Lambda^mT^*M\otimes \Ad P).
\end{array}$$

\section{Gauge Theory On AC \texorpdfstring{$G_2$}{G2}-Manifolds}\label{sec:gtonac}

In this section we show that the basic setup of gauge theory familiar from the compact case follows through to the case of weighted spaces. \par 

Let $(X^n,g)$ be an $n$ dimensional Riemannaian manifold and let $P\to X$ be a principal bundle. The bundle $P$ comes with a \emph{gauge group} $\mathscr{G}=\{\text{principal bundle isomorphisms}\colon P \to P\text { covering the identity}\}$. The gauge group acts on the space of connections $\mathscr{A}$ and one may form the space of \emph{connections modulo gauge}, $\mathscr{A}/\mathscr{G}.$ The space of connections is an affine space identified with $\Omega^1(X,\Ad P)$ since any two connections differ by a section in this space. \par
When $X$ is equipped with an $(n-4)$ form $\alpha$ we may define an operator on $2$-forms by
\begin{align*}
    T_\alpha &\colon \Omega^2(M) \to \Omega^2(M)\\
    T_\alpha&(\eta)=*(\alpha \wedge \eta). 
\end{align*}
Given such a set-up we seek connections $A$ whose curvature $F_A$ satisfies $T_\alpha(F_A)=\lambda F_A$ for a specified  $\lambda \in \R.$ This mimics the ASD instanton equation on a 4 manifold, where $\alpha=1$ and $\lambda=-1.$ We have two cases of interest:  
\begin{enumerate}
\item When $X=\Sigma$ is a nearly-K\"ahler $6$-manifold we take $\alpha=\omega$ and $\lambda=-1.$ The equation to be solved is
\begin{equation}
*(\omega \wedge F_A)=-F_A
\end{equation}
which is called the nearly K\"ahler instanton equation. A connection whose curvature solves this equation is called a \emph{nearly K\"ahler instanton} (or sometimes a pseudo Hermitian-Yang-Mills connection).  It is a non-trivial fact that solutions are Yang-Mills, as noted in \cite{xu2009instantons}. By \eqref{eq:su3space} a nearly K\"ahler instanton is a connection whose curvature has 2-form component lying in $\Lambda^2_8.$ Furthermore the nearly K\"ahler instanton equation is equivalent to 
\begin{equation}
    F_A\cdot s_6=0.
\end{equation}
In \cite{harland2010yang} it is shown that the canonical connection defines a nearly K\"ahler instanton. For our purposes this is the most important example of a nearly K\"ahler instanton and many of the tools we develop will be specifically for this connection.
\item In the case when $X=M$ is a $G_2$-manifold we take $\alpha=\varphi$ and $\lambda=-1,$ thus we look to solve
\begin{equation}
*(\varphi \wedge F_A)=-F_A
\end{equation}
which is called the $G_2$-instanton equation.
A connection $A$ whose curvature solves this equation is called a $G_2$-\emph{instanton}. Since $\varphi$ is closed it is not hard to see that such connections are automatically Yang-Mills, $\diff_A*F_A=0$. By  \eqref{eq:g2subspace} a $G_2$-instanton is a connection whose curvature has 2-form component lying in $\Lambda^2_{14}.$ It follows that the $G_2$-instanton equation is equivalent to either of the equations 
\begin{align}
    F_A\wedge \psi&=0 \label{eq:g2instantoneqn}\\
    F_A\cdot s_7&=0.
\end{align}

\end{enumerate}
For the remainder of this article we denote by $A_\infty$ a connection on a bundle $Q\to \Sigma$ and $A_C\coloneqq \pi^*(A_\infty)$ denotes the pullback connection on $\pi^*Q\to C(\Sigma).$
It is easily verified that $ A_\infty $ is a nearly K\"ahler instanton if and only if $A_C$ is a $G_2$-instanton. Furthermore examples of $G_2$-instantons on AC $G_2$-manifolds have been observed to decay to nearly K\"ahler instantons on the asymptotic link:
\begin{example}
\begin{itemize}
 
    \item In \cite{gunaydin1995seven} an instanton $A_\text{std}$ on the trivial $G_2$-bundle over $\R^7$ in constructed. This example decays to the canonical connection living on the trivial $G_2$-bundle over $S^6.$ We refer to this as the ``standard'' instanton. This example will be covered in more detail in Section~\ref{sec:stdinstsection}.
 \item Nearly K\"ahler instantons have been constructed on the Bryant-Salamon manifolds by Clarke  \cite{clarke2014instantons}, Oliveira \cite{oliveira2014monopoles} and Lotay-Oliveira \cite{lotay2018su2insantons}. All such examples decay  to the canonical connection living on an appropriate bundle. In particular, a limiting connection of the family of Clarke's instantons was constructed in \cite{lotay2018su2insantons} and it is interesting to note that this connection actually decays faster than Clarke's examples.
 
\end{itemize}
\end{example}

We now make precise the concept of a $G_2$-instanton decaying to a nearly K\"ahler instanton. To do so we first specify the class of principal bundles we work with:
\begin{mydef}
Let $M$ be an asymptotically conical $G_2$-manifold, with asymptotic cone $\Sigma$, and let $P\to M$ be a principal bundle over $M$. We call $P$ asymptotically framed  if there exists a principal bundle $Q\to \Sigma $ such that
$$ h^*P\cong \pi^*Q$$
where $\pi \colon C(\Sigma)\to \Sigma$ is the natural projection map.
\end{mydef}
\begin{remark}
Such a framing always exists \cite{lotay2018su2insantons}, so no generality is lost by fixing a framing. This condition is the slightly more general than the setup of Taubes in \cite{taubes1987gauge} where it is assumed that $Q$ is trivial. When this is the case Taubes notes that if $G$ is simple and simply connected then $P$ must also be trivial. 
\end{remark}
It will now be assumed that the principal $G$-bundle $P$ has $G$ semisimple and that an invariant metric is fixed on the fibers of $\Ad P.$
\begin{mydef}\thlabel{acconnection}
Let $M$ be an asymptotically conical manifold.
A connection $A$ on an asymptotically framed bundle $P \to M$ is called asymptotically conical with rate $\mu$ if there exists a connection $A_\infty$ on $Q \to \Sigma$ such that, denoting $A_C=\pi^*(A_\infty)$ we have
\begin{equation}\label{eq:asympconnection}
    |\nabla_{C}^j(h^*(A|_{M_\infty}) - A_C)|_{C}=O(r^{\mu-1-j})
\end{equation}
for all non-negative integers $j$, for some $\mu<0$ and where $\nabla_{C}$ is a combination of the Levi-Civita connection on the cone metric and $A_C$. Here $|\cdot |_C$ is the norm induced by the cone metric and the metric on $\mf{g}.$ Furthermore, we call the quantity $\mu_0 = \inf \{ \mu ; \, A \text{ is asymptotically conical with rate }\mu\}$ the fastest rate of converge of $A.$ 
\end{mydef}

It is natural to introduce Sobolev spaces of connections. Recall the space of connections is an affine space, a choice of reference connection $A$  identifies the space of connections $\mathcal{A}$ with $\Omega^1(M, \Ad P)$, since other connection $B$ is $B=A+a$ where $a$ is a uniquely determined Lie algebra valued 1-form. We let 
$$\mathcal{A}_{k,\mu-1}=\left\{ A+a \, ; \, a \in \Omega^1_{k,\mu -1}(M, \Ad P)\right\}$$
be the space of $L^2_{k,\mu-1}$ connections and $\mathcal{A}_{\mu-1}=\cap_{k\geq 0}\mathcal{A}_{k,\mu-1}$ the space of $C^\infty_{\mu-1}$ connections.  \par 

We also need to introduce gauge transformations with specified decay properties. Recall a gauge transformation $g$ is an automorphism of the principal bundle that covers the identity and that $g$ acts on a connection $A$ via the formula $g\cdot A=gAg^{-1}-\diff gg^{-1}.$ Following the setup of Nakajima in \cite{nakajima1990moduli} suppose $P\to M$ is a principal $G$ bundle, pick a faithful representation $G\to GL(V)$ and form the associated bundle $\text{End}(V).$ We define 
\begin{equation}\label{eq:weightedgaugegroup}
\mathcal{G}_{k+1,\mu}\coloneqq \{ g \in C^0(\text{End}(V)) \, ; \, \norm{\text{Id}-g}_{k+1,\mu}<\infty, \, g\in G \text{ a.e}\}.
\end{equation}
Furthermore we define $\mathcal{G}_\mu\coloneqq \bigcap_{l\geq 0}\mathcal{G}_{l,\mu}.$ The main difference between our setup and that of \cite{nakajima1990moduli} is that we shall  consider varying the weight $\mu.$ As in the compact case one can check directly that the exponential map provides the (weighted) gauge group with charts making them into manifolds and that the action on the space of connections is smooth. See \cite{freed2012instantons} for details on this construction.
\begin{lemma}
The sets $\mathcal{G}_{k+1,\mu}$ are Hilbert Lie groups with Lie algebra $\Omega^0_{k+1,\mu}(M, \Ad P)$ for $k\geq 3$. The group $\mathcal{G}_{k+1,\mu}$ acts smoothly on $\mathcal{A}_{k,\mu-1}$ via gauge transformations when $k\geq 4$.
\end{lemma}
With this in hand we can define our main object of study:
\begin{mydef}
Let $M$ be an AC $G_2$-manifold with asymptotic link $\Sigma.$ Let $P\to M$ be a bundle asymptotically framed by $Q\to \Sigma$ and let $A_\infty $ be a nearly K\"ahler instanton on $Q.$ The moduli space of $G_2$-instantons asymptotic to $A_\infty$ with rate $\mu$ is 
\begin{equation}\label{eq:modulispacedefinition}
\mathcal{M}(A_\infty,\mu)=\{G_2 \text{ instantons $A$ on $P$ satisfying \eqref{eq:asympconnection}}\}/\mathcal{G}_\mu.\end{equation}
\end{mydef}
\begin{remark}
It follows from \cite[Proposition 3]{oliveira2014monopoles} that an asymptotically conical $G_2$-instanton $A$ must have connection at infinity $A_\infty$ a nearly-K\"ahler instanton.
\end{remark}
To begin studying this moduli space we first try to understand it as the zero set of a (non-linear) elliptic operator.
Pick a reference connection $A$ and write any other connection $B$ as $B-A=a$ where $a\in \Omega^1(M, \Ad P).$ 
The relationship between the curvatures is $F_B-F_{A}=\diff_Aa + a \wedge a$ and hence the $G_2$-instanton equation for $B$ becomes the non-linear equation $\psi \wedge (\diff_Aa + a \wedge a)=0.$ From an analytic perspective it is advantageous to work instead with the $G_2$-monopole equation
\begin{equation}\label{eq:g2monopole}
    \diff_A f + *\left( \psi \wedge (\diff_Aa + a\wedge a)\right)=0 
\end{equation} 
for some $f \in \Omega^0(M,\Ad P).$ This is because  adding the gauge fixing condition $\diff_A^*a=0$ to \eqref{eq:g2monopole} yields an elliptic equation. Solutions of this equation are precisely elements of the zero set of the non-linear operator $\mathcal{D}_A\colon \Gamma((\Lambda^0\oplus \Lambda^1)\otimes \Ad(P))\to \Gamma((\Lambda^0\oplus \Lambda^1)\otimes \Ad(P))$ given by 
\begin{equation}\label{eq:nonlinearop}
\mathcal{D}_A(f,a)= \diff_A^*a + \diff_Af + [a,f] + *(\psi \wedge (\diff_Aa + a \wedge a)).
\end{equation}
To see that the linearisation of $\mathcal{D}_A$ is elliptic we compare the  expression for the Dirac operator \eqref{eq:diracop7} and conclude the linearisation is the \emph{twisted Dirac operator} $D_A$ where 
\begin{equation}\label{eq:linearop}
D_A =\begin{pmatrix} 0 & \diff_A^* \\
\diff_A & *(\psi \wedge \diff_A \, \cdot \, \,   )
\end{pmatrix}.
\end{equation}
Nothing is lost in moving to this setup for  if $(f,A)$ satisfies the $G_2$-monopole equation (with a decay condition) then $f=0$. Thus the zero set of $\mathcal{D}_A$ consists of solutions $B$ to the $G_2$-instanton equation together with the gauge fixing condition $\diff_A^* (B-A)=0.$ Similarly if $(f,a)$ satisfies the linearised $G_2$-monopole equation then $f=0$. We delay the proof of these facts until later in this section.\par 
Being a twisted Dirac operator, $D_A$ is first order elliptic and formally self adjoint \cite{lawson2016spin}. The kernel of $D_A$  consists of gauge fixed solutions to the linearised $G_2$-monopole equation. In studying the behaviour of this operator we are lead to consider various other Dirac operators, so we briefly list those operators we will require. If $A$ is an asymptotically conical connection then there is a connection $A_\infty$ on $Q\to \Sigma$ and we define $A_C=\pi^*A_\infty$ as in \eqref{eq:asympconnection}. Hence we have operators\\

\begin{center}
\ra{1.3}
\begin{tabular}{|c|c|c|}\hline
  Operator& Bundle & Formula \\ \hline
     $D_A$&$\slashed{S}(M)\otimes \Ad P$ &$\text{cl}\circ \nabla^\text{LC}\otimes \nabla^A$ \\
     $D_{A_C}$& $\slashed{S}(C)\otimes \Ad (\pi^*Q)$ &$\text{cl}\circ \nabla^\text{LC}\otimes \nabla^{\pi^*A_\infty}$ \\
     $D^t_{A_\infty}$&$\slashed{S}(\Sigma)\otimes \Ad Q$ & $\text{cl}\circ \nabla^t\otimes \nabla^{A_\infty}$  \\
\hline
\end{tabular}
\end{center}

The operator $D_A$ fits into the analytic framework for operators on AC manifolds  which has been developed by Lockhart-McOwen in \cite{lockhart1985elliptic} and by Marshall in \cite{marshal2002deformations}, whose work we now adapt to our setting.
Suppose $(M,g)$ is asymptotically conical and $\rho$ is a radius function, then $g_\text{Acyl}\coloneqq \rho^{-2}g$ is asymptotically cylindrical \cite{marshal2002deformations}. Let $E$ any vector bundle with fibre metric induced from the Riemannian metric (the cases we shall consider all fall into this category) and let $T_E\colon (E,g) \to (E,g_\text{Acyl})$  be the natural isometry from the conformal change of AC to Acyl metrics.
Let $E$ and $F$ be two such bundles over $M$. We will call an operator $P\colon \Gamma(E) \to \Gamma(F) $ \emph{ asymptotically conical} with rate $\nu$ if 
$$\rho^\nu (T_F)^{-1}P\,T_E$$
is asymptotic to an operator $P_\infty$ that is invariant under the $\R^+$ action on $M_\infty=h((R,\infty)\times \Sigma).$ More details of this construction can be found in \cite{marshal2002deformations}.  An important fact is that an AC rate $\nu$ order $l$ operator $P\colon \Gamma(E)\to \Gamma(F)$ admits a bounded extension $P\colon L^2_{k+l,\mu}(E)\to L^2_{k,\mu-\nu}(F).$  Our main  cases of interest are:
\begin{itemize}
    \item $E=F=\slashed{S}(M) \otimes \Ad(P) \cong (\Lambda^0\oplus \Lambda^1)\otimes \Ad(P)$ and $T_E=\begin{pmatrix} 1 & 0\\
    0&\rho\end{pmatrix}$. The asymptotic conditions on the $G_2$-structure and the connection $A$ ensure the operator $D_A$ is asymptotically conical with rate 1.
    \item $E=F=\Ad(P)$ and $T_E=\text{Id}.$ Here the asymptotic conditions on the metric and the connection ensure the coupled Laplace operator $\diff_A^*\diff_A$  is asymptotically conical with rate 2.
\end{itemize}

Furthermore the above operators are \emph{uniformly elliptic}, which is to say  they are elliptic operators that converge to elliptic operators on the cone. The operators in consideration converge to $D_{A_C}$ and $(\diff_{A_C})^*\diff_{A_C}$ which are built from the connection $A_C$ living on the cone. Such operators come with estimates that ensure desirable regularity properties analogous to the situation on a compact manifold.
\begin{theorem}\thlabel{ellipticestimate}
Suppose $P\colon C^\infty_c(E) \to C^\infty_c(F)$ is a smooth uniformly elliptic, asymptotically conical operator of rate $\gamma$ and order $l\geq 1.$ Suppose that $\eta \in L^1_\text{loc}(F)$ and $\xi \in L^1_\text{loc}(E)$ is a weak solution of $P \xi = \eta.$\\
If $\xi \in L^p_{0,\beta + \gamma}(E)$ and $\eta \in L^p_{k, \beta}(F)$ then $\xi \in L^p_{k+l, \beta +\gamma}(E)$ with  
$$\norm{\xi}_{L^p_{k+l,\beta+\gamma}(E)} \leq C\left( \norm{\eta}_{L^p_{k,\beta}(F)} + \norm{\xi}_{L^p_{0,\beta + \gamma}(E)}\right)$$
where the constant $C>0$ is independent of $\xi.$ 
\end{theorem}
Thus the kernel of an order $l$ AC uniformly elliptic rate $\nu$ operator $P\colon L^2_{k+l,\mu} \to L^2_{k,\mu-\nu}$ is independent of $k$, we therefore denote the kernel simply by $\text{Ker}P_\mu\coloneqq \text{Ker}P\colon L^2_{k+l,\mu} \to L^2_{k,\mu-\nu}.$
Using this Sobolev estimate together with the weighted Sobolev embedding theorem we find that the kernel of an AC uniformly elliptic operator has the desirable property of consisting of smooth sections. 
To study the kernel of our operators we will need to determine the set of \emph{critical weights} which are determined by the asymptotic operators.
\begin{mydef}
Let $C$ be a $G_2$ cone with asymptotic link $\Sigma$ and $A_C=\pi^*(A_\infty).$ Let $P_C$ be either the twisted Dirac operator $D_{A_C}$ or the coupled Laplace operator $\diff_{A_C}^*\diff_{A_C}$ acting on sections of $E=\slashed{S}(C)\otimes\Ad (\pi^*(Q))$ and $E=\Ad(\pi^*(Q))$ respectively. The set $W_\text{crit}(P_C)$ of critical weights of the operator $P_C$ is 
\begin{equation}
    W_\text{\normalfont crit}(P_C)= \left\{ \lambda \in \R \, ; \, \exists \text{ a non-zero homogeneous order } \lambda \text{ section $\eta$ of $E$ with } P_C(\eta) =0\right\}.
\end{equation}
\end{mydef}
In the case of a twisted spinor, a section $\eta$ is homogeneous of order $\lambda$ if $\eta=(f+v)\cdot s_C$ (here $S_C$ denotes the parallel spinor on the cone) with $f=r^\lambda \pi^*(f_\infty)$ and $v=r^{\lambda+1}\pi^*(v_\infty),$ where $f_\infty,v_\infty$ are sections of the appropriate bundles over $\Sigma$ (equivalently $\eta=r^\mu s_\infty $ where $s_\infty$ is a spinor on $\Sigma$ lifted to the cone). In general one has to allow for complex critical weights but the  formal self-adjoint property of the above operators in question ensures all such weights are real. The set $W_\text{crit}(P_C)$ is countable and discrete.  Recall an operator between Banach spaces is called Fredholm if it has finite dimensional kernel and cokernel. An operator whose range admits a finite dimensional complementary subspace automatically has closed range, so this is in particular true of Fredholm operators. 
\begin{theorem}\thlabel{homog}
Let $E$ and $P$ be as above. Then the extension $P\colon L^p_{k+l, \mu }(E) \to L^p_{k,\mu-\nu}(F)$ is Fredholm whenever $\mu \in \R \! \setminus \! W_\text{\normalfont crit}(P_C).$
Furthermore if $[\mu,\mu^\prime]\cap W_\text{\normalfont crit}(P_C)=\emptyset $, then 
$$\text{\normalfont Ker}P_\mu=\text{\normalfont Ker}P_{\mu^\prime}.$$ 
\end{theorem}

Thus the kernel of $P$ is independent of the weight provided we do not pass through any critical weights. For any $k\geq 4$  and $\mu<0$ we define:
\begin{align}
    (\diff_A^*\diff_A)_{k,\mu}& \coloneqq \diff_A^*\diff_A \colon \Omega^0_{k+2,\mu}(M,\Ad P)\to \Omega^0_{k,\mu-2}(M, \Ad P) \label{eq:coupledlaplacian} \\
    (D_A)_{k,\mu}& \coloneqq D_A\colon L^2_{k+1,\mu}(\slashed{S}(M)\otimes \Ad P)\to L^2_{k,\mu-1}(\slashed{S}(M)\otimes \Ad P).
\end{align}

The last result we shall require is a Fredholm alternative for AC manifolds.
\begin{theorem}
Let $P$ be an AC uniformly elliptic order $l$ and rate $\nu$ operator and suppose that $\mu \not\in W_\text{\normalfont crit}(P_C)$ so that the extension
$$P\colon L^2_{k+l,\mu}(E) \to L^2_{k,\mu-\nu}(F)$$
is Fredholm. Then 
\begin{enumerate}
    \item There exists a finite dimensional subspace $\mathcal{O}_{\mu-\nu}$ of $L^2_{k,\mu-\nu}$ such that 
    \begin{equation}
        L^2_{k,\mu-\nu}(F)=P(L^2_{k+l,\mu}(E))\oplus \mathcal{O}_{\mu-\nu}
    \end{equation}
    and
    \begin{equation}
        \mathcal{O}_{\mu-\nu}\cong \text{\normalfont Ker}P^*_{-7-\mu+\nu}.
    \end{equation}
    \item If $\mu>-\frac{7}{2}+\nu$ then we can take  
    $$\mathcal{O}_{\mu-\nu} = \text{\normalfont Ker}P^*_{-7-\mu+\nu}.$$
    \item The image of  the extension $P$ is the space 
    \begin{equation}
    P(L^2_{k+l,\mu}(E))= \left\{ \eta \in L^2_{k,\mu-\nu}(F) \, ; \, \langle \eta , \kappa \rangle _{L^2(F)}=0 \, \,  \text{\normalfont for all } \kappa \in \text{\normalfont Ker}(P^*)_{-7-\mu+\nu} \right\}.
    \end{equation}
    
\end{enumerate}
\end{theorem}

 \subsection{The Space of Connections Modulo Gauge}
The aim of this subsection is to give a description of the space of connections modulo gauge. Given a  reference connection $A$ we may view the gauge orbit $\mathcal{G}_{k+1,\mu}\cdot A$  as a subset of $\Omega^1_{k,\mu-1}(M, \Ad P).$ The infinitesimal action of the $\mu$-weighted gauge group is $-\diff_A\colon \Omega^0_{k+1,\mu}(M, \Ad P)\to \Omega^1_{k,\mu-1}(M, \Ad P)$ and our strategy is to show this image is closed and hence admits a complement. We aim to find a particular complement for this image, which is called a ``slice'' of the action and as usual is given by the Coulomb gauge fixing condition. Such a splitting of $\Omega^1_{k,\mu-1}(M, \Ad P)$ shows the quotient space is a smooth Hilbert manifold if the action is free. As in the case of a compact manifold  the Coulomb gauge fixing condition may not pick out a unique class representative globally, so we also give a sufficient condition for this property to hold. Our strategy for this task is to develop the Fredholm theory of the coupled Laplacian $\diff_A^*\diff_A.$ \par 
To learn when $\diff_A^*\diff_A$ is Fredholm \thref{homog} tells us to look for homogeneous order $\lambda$ elements of the kernel of $\diff_{A_C}^*\diff_{A_C}$.  
Such a solution looks like $f=r^\lambda \xi$ for some $\xi \in \Omega^0(\Sigma, \Ad P)$ and we calculate $$\diff_{A_C}^{*}\diff_{A_C} f=r^{\lambda-2}\left( \diff_{A_\infty}^*\diff_{A_\infty}\xi-\lambda(\lambda+5)\xi\right)$$ so such a solution exists if and only if $\lambda(\lambda+5)$ is an eigenvalue of $\diff_{A_\infty}^*\diff_{A_\infty}.$ Given that the coupled Laplace operator is positive, we find that there are no critical weights in the range $(-5,0).$ Therefore:

\begin{proposition}
Let $A$ be an asymptotically conical $G_2$-instanton. If $\mu \in (-5,0)$ then the coupled Laplacian  $\diff_{A}^*\diff_A \colon \Omega^0_{k+2,\mu}(M,\text{\normalfont Ad}P)\to \Omega^0_{k,\mu-2}(M,\text{\normalfont Ad} P)$ is Fredholm.
\end{proposition}
\begin{remark}
All known examples of AC $G_2$-instantons converge at a rate in this interval.
\end{remark}

The next lemma is a gauged version of integration by parts on AC manifolds, the proof goes through identically to \cite[Lemma 4.16]{karigiannis2012deformation}.
\begin{lemma}
Let $\xi \in \Omega^{m-1}_{k,\mu}(M,\text{\normalfont Ad}P)$ and $\eta \in \Omega^{m}_{l,\nu}(M,\text{\normalfont Ad}P).$ If $k,l \geq4$ and $\mu + \nu<-6 $ then 
$$ \langle \diff_A\xi,\eta\rangle _{L^2}=\langle \xi,\diff_A^*\eta \rangle_{L^2}.$$
\end{lemma}

\begin{corollary}\thlabel{closed}
Let $f\in \text{\normalfont Ker}(\diff_A^*\diff_A)_\mu.$ If $\mu<0$ then $\diff_Af=0.$ 
\end{corollary}
\begin{proof}
We have seen that there are no critical weights of $\diff_A^*\diff_A$ in the region $(-5,0),$ so $\diff_A^*\diff_A f=0$ and $\mu<0$ then  $f\in \Omega^0(M,\Ad P)_{k+2,\mu_0}$ for some $\mu_0<-\frac{5}{2}$ and for any $k.$ It follows that $\diff_Af\in\Omega^1_{k+1,\mu_0-1}$ and integration by parts is valid for such and weight and we see that
$$\norm{\diff_A f}_{L^2}= \langle \diff_A^*\diff_A f, f \rangle_{L^2}=0.$$
\end{proof}

The following useful lemma is due to Marshall \cite{marshal2002deformations}, it is a straightforward application of the maximum principle. We denote by $\Delta = \diff ^* \diff$ the usual Laplacian on functions.
\begin{lemma}
Let $(M,g)$ be an asymptotically conical manifold. If $\mu<0$ then $\text{\normalfont Ker} (\Delta)_\mu=\{0\}. $
\end{lemma}
As an immediate corollary we find that harmonic sections of the adjoint bundle must vanish:
\begin{corollary}\thlabel{harmonicvanishes}
Let $f\in \text{\normalfont Ker}(\diff_A^* \diff_A)_{\mu}.$ If $\mu<0$ then $f=0.$
\end{corollary}
\begin{proof}
Since $f\in \text{Ker}(\diff_A^* \diff_A)_{\mu}$ we know by \thref{closed} that $\diff_Af=0$ and the connection is compatible with the inner product so that 
$$\Delta |f|^2=\diff^*\diff|f|=2\diff^*\langle \diff_A f, f\rangle =0.$$
So $|f|^2$ is a harmonic function and hence zero.
\end{proof}
Pausing for a moment we can finally justify our switch from the $G_2$-instanton equation \eqref{eq:g2instantoneqn}  to the $G_2$-monopole equation \eqref{eq:g2monopole}.
\begin{corollary}\thlabel{xivanishes}
If $\mu<0$ and  $(f, A)\in \Omega^{0}_{k+1,\mu}(M, \text{\normalfont Ad}P)\oplus \mathcal{A}$ satisfies the $G_2$-monopole equation 
$$ \diff_A f + *(\psi \wedge F_A)=0$$
then $f=0.$
\end{corollary}
\begin{proof}
We apply $\diff_A^*$ to the  $G_2$-monopole equation $\diff_Af + *(\psi \wedge F_A)=0$ and use that  $\psi$ is closed together with the Bianchi identity to find that
$\diff_A^*\diff_A f=0$ and hence  \thref{harmonicvanishes} is applicable.

\end{proof}
\begin{corollary}\thlabel{0formszero}
Let $A$ be an asymptotically conical $G_2$-instanton. If $ \mu<0 $ and $(f,a)\in L^2_{k+1,\mu}((\Lambda^0\oplus \Lambda^1)\otimes \text{\normalfont Ad} P)$ satisfies the linearised $G_2$-monopole equation $$ \diff_Af + *(\psi\wedge \diff_Aa)=0$$
then $f=0.$
\end{corollary}
\begin{proof}
Again we apply $\diff_A^*$ any observe that 
$$ \diff_A^*\diff_A f + *(\psi \wedge \diff_A^2 a)=0 $$
and since $A$ is a $G_2$-instanton $\psi \wedge \diff_A^2a=\psi\wedge [F_A,a]=0$ so we may appeal to \thref{harmonicvanishes}.
\end{proof}

We return our attention to splitting the space of 1-forms. It is sufficient for our purpose to work in the regime where $-5<\mu<0.$ Note from \eqref{eq:coupledlaplacian} that  $(\diff_A^*\diff_A)_{k,\mu}$ has trivial cokernel, since the adjoint maps from the space with weight  $-5-\mu<0.$ The bounded inverse theorem then yields:
\begin{lemma}\thlabel{covclosed}
Let $(M,g)$ be an asymptotically conical $G_2$-manifold and let $A$ be an asymptotically conical connection on a principal bundle $P \to M.$ 
If $-5<\mu<0$ then 
$$\diff_A^*\diff_A \colon \Omega^0_{k+2,\mu}(M, \text{\normalfont Ad} P)\to \Omega^0_{k,\mu-2}(M, \text{\normalfont Ad} P)$$ 
is an isomorphism of topological vector spaces.

\end{lemma}

This allows us to split the space of 1-forms, heuristically  $\Omega^1_{l,\mu-1}(M,\Ad P)= \text{Im}\diff_A \oplus \text{Ker}\diff_A^*.$ Note  however that the splitting may not be orthogonal, since the weights we require need not be in the $L^2$ integrable regime. Instead we make use of the following basic lemma from Banach space theory:
\begin{lemma}\thlabel{banachspacemap}
Suppose $T\colon X \to Y$ is a bounded linear operator between Banach spaces, so that the kernel $\text{\normalfont Ker}T$ is closed and admits a closed complement. A closed subspace $X_0\subset X$ is a complement to $\text{\normalfont Ker}T$ if and only if 
\begin{enumerate}
    \item $T|_{X_0}$ is injective
    \item $T(X)=T(X_0).$
\end{enumerate}
\end{lemma}

We would like to apply the above lemma with $X=\Omega^1_{k+1,\mu-1}(M,\Ad P),$ $T=\diff_A^*$ and $X_0=\diff_A(\Omega^0_{k+2,\mu}(M, \Ad P)),$ thus we must first establish that the image of $\diff_A$ is closed. 
\begin{lemma}
The operator $\diff_A\colon\Omega^0_{k+2,\mu}(M,\text{\normalfont Ad} P)\to \Omega^1_{k+1,\mu-1}(M, \text{\normalfont Ad})$ has closed image.
\end{lemma}
\begin{proof}
Let $\{ \diff_A f_n\}_{n=1}^\infty$ be a sequence in $\diff_A(\Omega^0_{k+2,\mu}(M,\Ad P)),$ and let  $a\in \Omega^1_{k+1,\mu-1}(M,\Ad P)$ be such that
$$\norm{\diff_A f_n-a}_{k+1,\mu-1}\to 0.$$
Applying the bounded operator $\diff_A^*$ we see that $\diff_A^*\diff_Af_n$ converges to $\diff_A^*a$ in $\Omega^0_{k,\mu-1}(M,\Ad P).$ Since $\diff_A^*\diff_A$ admits a bounded inverse we find that $f_n$ converges to $f\coloneqq (\diff_A^*\diff_A)^{-1}\diff_A^*a$ in $\Omega^0_{k+2,\mu}(M,\Ad P).$ Finally we apply the bounded operator $\diff_A$ and see that 
$$\norm{\diff_A f_n-\diff_Af}_{k+1,\mu-1}\to 0.$$
So $a=\diff_Af$ by uniqueness of limits and hence $\text{Im}\diff_A$ is closed.
\end{proof}
\begin{theorem}[Slice Theorem]\thlabel{1formsplit}
Let $-5<\mu<0$  then
\begin{equation}
    \Omega^1_{k+1,\mu-1}(M, \text{\normalfont Ad} P)= \text{\normalfont Ker}(\diff_A^*\colon \Omega^1_{k+1,\mu-1}(M, \text{\normalfont Ad}  P) \to \Omega^0_{k,\mu -2}(M, \text{\normalfont Ad} P)) \oplus \diff_A(\Omega^0_{k+2,\mu}(M, \text{\normalfont Ad}  P)).
\end{equation}
\end{theorem}
\begin{proof}
We apply \thref{banachspacemap} to the operator $\diff_A^*\colon \Omega^1_{k+1,\mu-1}(M, \text{\normalfont Ad}  P) \to \Omega^0_{k,\mu -2}(M, \text{\normalfont Ad} P) .$ Since $\diff_A^*$ is AC this extension is bounded and hence the kernel is a closed subspace. We claim that  $\diff_A(\Omega^0_{k+2,\mu}(M, \Ad P))$ satisfies the hypothesis of \thref{banachspacemap}. Firstly, as noted above this is a closed subspace.  \\
Claim 1: $\diff_A^*$ is injective when restricted to $\diff_A(\Omega^0_{k+2,\mu}(M, \Ad P)).$\\
To see this suppose that $\diff_A^*\diff_Af = \diff_A^*\diff_Ag$ for $f,g \in \Omega^0_{k+2,\mu}(M, \Ad P). $ Then $f-g$ is harmonic and hence 0, so certainly $\diff_A f - \diff_A g=0.$
\\
Claim 2: $\diff_A^*\diff_A(\Omega^0_{k+2,\mu}(M, \Ad P))=\diff_A^*(\Omega^1_{k+1,\mu-1}(M, \Ad P)).$\\
This follows from \thref{covclosed}.\\
\end{proof}
The importance of this result is that it gives us a local description of the space  $\mathcal{B}_{k+1,\mu}=\mathcal{A}_{k+1,\mu-1} / \mathcal{G}_{k+2,\mu}$ of connections modulo gauge. The infinitesimal action of the gauge group $\mathcal{G}_{k+2,\mu}$ is $-\diff_A\colon \Omega^0_{k+2,\mu}(M, \Ad P) \to \Omega^1_{k+1,\mu-1}(M, \Ad P),$ so we can interpret  \thref{1formsplit} as a so called ``slice'' theorem-- we have found complements for the action of the gauge group. If the action is free, it will follow from general theory that the quotient space is a smooth manifold.\par 
To see that the action is free set  $\Gamma_{A,\mu}= \{  g\in \mathcal{G}_{k,\mu} \, ; \,  g\cdot A =A\}.$ Recall we are viewing gauge transformations as sections of $\text{End}(V)$ as defined in \eqref{eq:weightedgaugegroup}. By a standard argument $\Gamma_{A,\mu}$ is closed Lie subgroup of $\text{End}(V_x)$, for some base point $x$, whose elements are covariantly constant sections of $\text{End}(V)$ \cite[Section 4.2.2]{donaldson1990geometry} (i.e sections $g$ for which $\diff_Ag=0).$  Recall $G$ is assumed to be semi-simple and the inner product  on the representation $V$ is assumed to be invariant. The connection $A$ has Hol$(A)\subset G$ so it preserves the inner product on $V$ and thus also preserves the induced inner product on $\text{End}(V)=V \otimes V^*.$ Thus, regarding gauge transformations as sections of $\text{End}(V)$ as in \eqref{eq:weightedgaugegroup}, if $g\in \Gamma_{A,\mu}$ then $|g-\text{Id}|\in \Omega^0_{k+2,\mu}(M,\Ad P)$ and 
$\Delta |g-\text{Id}|^2=2\diff^*\langle \diff_A(g-\text{Id}), g-\text{Id}\rangle =0.$
We have seen that such a function must vanish, so that $\Gamma_{A,\mu}=\text{Id}.$ This is in contrast to the case when $M$ is compact since, in the AC case, reducible connections  (those with $\text{Hol}(A)$ a proper subgroup of $G$) do not lead to singularities in the space of connections modulo gauge. As a consequence, if we set
\begin{equation}
    T_{A,\mu, \epsilon } = \{ a \in \text{Ker}\diff_A^*\colon \Omega^1_{k+1,\mu-1}(M, \Ad P)\to \Omega^0_{k, \mu-2}(M, \Ad P) \, ; \, \norm{a}_{k,\mu-1}<\epsilon \}.
\end{equation}
then $T_{A, \epsilon}$ models a local neighbourhood of $A$ in $\mathcal{B}_{k+1,\mu}.$ 
We summarise this in the following corollary:
\begin{corollary}\label{couloumb}
Let $P\to M$ be a principal $G$-bundle with $G$ semisimple. If $-5<\mu<0$ then the moduli space $\mathcal{B}_{k+1,\mu}=\mathcal{A}_{k+1,\mu-1}/\mathcal{G}_{k+2,\mu}$ is a smooth manifold and the sets $T_{A,\mu,\epsilon}$ provide charts near $[A].$
\end{corollary}
Working with the local models $T_{A,\mu,\epsilon}$ amounts to solving \emph{Coulomb gauge} condition $\diff_A^*a=0.$ This condition picks out a unique gauge equivalence class locally but may not do so sufficiently far away from $A.$ The failure of a global gauge fixing is a reflection of the rich topology of $\mathcal{B}$ \cite{donaldson2006mathematical}. Using the surjectivity of the coupled Laplacian one can prove a weighted version of \cite[Proposition 2.3.4]{donaldson1990geometry} which gives a sufficient condition for solving the Coulomb gauge condition:

\begin{proposition}\thlabel{coulomb}
Let $P\to M$ be a principal $G$-bundle with $G$ a compact Lie group. Let $A$ be an $L^2_{k+1, \mu-1}$ connection on $P$ for $\mu\in(-5,0)$. There is a constant $c(A)>0$ such that if $B$ is another $L^2_{k+1, \mu-1}$ connection on $P$ and $a=B-A$ satisfies 
$$\norm{a}_{L^2_{4,\mu-1}} < c(A)$$
then there is a gauge transformation $g\in \mathscr{G}_{k+2,\mu}$ such that $u(B)$ is in Coulomb gauge relative to $A.$
\end{proposition}

\subsection{A Regularity Result}
  We observed that elements of the zero set of the operator $\mathcal{D}_A$ from \eqref{eq:nonlinearop}  consists of (smooth) $G_2$-instantons $B$ together with the Coulomb gauge condition $\diff_A^*(B-A)=0$ which fixes a gauge near to $A.$  The decay condition \eqref{eq:asympconnection} we impose on connections in the moduli space $\mathcal{M}(A_\infty, \mu)_k$, that a neighbourhood of $0$ in the zero set of $\mathcal{D}_A\colon C^\infty_{\mu-1} \to C^\infty_{\mu-2}$ is homeomorphic to  a neighbourhood of $[A]$ in $\mathcal{M}(A_\infty, \mu).$ We now show that we may instead study the moduli space using the weighted Sobolev spaces; this is advantageous since these spaces are the natural setting for studying elliptic operators on AC manifolds. Let us denote by $\mathcal{M}(A_\infty,\mu)_k$ the space of $L^2_{k,\mu-1}$ connections by $L^2_{k+1,\mu}$ gauge transformations.
Note that the Sobolev multiplication theorem (\thref{weightedmult})  ensures $\mathcal{D}_A\colon L^2_{k+1,\mu-1}(\slashed{S}(M)\otimes \Ad P)\to L^2_{k,\mu-1}(\slashed{S}(M)\otimes \Ad P)$ is bounded if $k\geq 3$.

We use the regularity of uniformly elliptic, asymptotically conical operators to obtain a result comparable to Donaldson and Kronheimer \cite[Proposition 4.2.16]{donaldson1990geometry}. 
\begin{proposition}\thlabel{regularity}
Let $k\geq 4$ and $-5<\mu<0.$
Then the natural inclusion $\mathcal{M}(A_\infty, \mu)_{k+1} \hookrightarrow \mathcal{M}(A_\infty, \mu)_k$ is a homeomorphism. 
\end{proposition}
\begin{proof}
  Suppose $A$ is an   asymptotically conical $G_2$-instanton of rate $\mu$ and class  $L^2_{k, \mu-1.}$, we will show that there exists a gauge transform $g$ such that $g(A)$ is in $L^2_{k+1,\mu-1}$. Firstly we know by \thref{coulomb} there is an $\varepsilon>0$ such that any $L^2_{k, \mu-1}$ connection $B$ with $\norm{A-B}_{L^2_{k,\mu-1}}< \varepsilon$ can be gauge transformed into Coulomb gauge relative to $A.$ Since $C^\infty_{\mu-1}$ lies densely in $L^2_{k,\mu-1}$ we may pick a smooth connection $B$ satisfying this condition, then we know there exists a $g$ in $L^2_{k+1, \mu}$ with 
  $$ \diff_A^* ( g^{-1}(B)-A)=0.$$
  The Coulomb gauge condition is symmetric, so that $A$ is also in Coulomb gauge relative to $g^{-1}(B),$ i.e
 \begin{equation}\label{eq:coulomb1}
 \diff^*_{g^{-1}(B)}(A-g^{-1}(B))=0. 
 \end{equation} 
  Let us denote $ g(A)=B+a,$ then we can apply $g$ to \eqref{eq:coulomb1} deduce  
  $$\diff_B^*a=0.$$
 Furthermore we have the relation 
$$ *(\psi \wedge \diff_Ba )=-* (\psi   \wedge F_B) - *(\psi \wedge a \wedge a).$$
Now  the weighted multiplication theorem      \thref{weightedmult}                                 ensures that $a \wedge a \in L^2_{k, 2(\mu-1)} $  and the curvature of $B$ lies in $C^\infty_{\mu-2}$. Thus $D_B(a)\in L^2_{k,\mu-2}$ and the  asymptotically conical uniformly elliptic estimates for the smooth operator $D_B$ allow us deduce that $a \in L^2_{k+1, \mu-1}.$ That is, we have bootstraped to gain a degree of differentiability.   This establishes the surjectivity of the inclusion.\\
The map is clearly injective and continuous; to see it is a homeomorphism we show that the two spaces have the same convergent sequences with their induced topologies. Let $\{a_n\}_{n=1}^\infty$ be a sequence in $\mathcal{M}(A_\infty,\mu)$ which is convergent in the $L^2_{k,\mu-1}$ norm, i.e there is some $a_\infty\in \mathcal{M}(A_\infty, \mu)$ with $\norm{a_n-a_\infty}_{k\mu-1}\to0$ as $n\to \infty.$ 
Observe that, since $a_n$ and $a_\infty$ are in the zero set of $\mathcal{D}_B$, we have that $$|D_B(a_n-a_\infty)|=|*(\psi\wedge(a_n\wedge a_n - a_\infty \wedge a_\infty))|=|\pi_7(a_n\wedge a_n-a_\infty \wedge a_\infty)| $$ where the final equality follows from the facts that the Hodge star operator acts isometrically and since the operation of wedging  with $\psi$ is an isomorphism of representations between $\Lambda^2_7$ and $\Lambda^6$ which preserves the norm.
To see that $\{a_n\}_{n=1}^\infty$ also converges in the $L^2_{k+1,\mu}$ norm we observe that
\begin{align*}
\norm{a_n-a_\infty}_{k+1,\mu-1}&\leq C(\norm{D_B(a_n-a_\infty)}_{k,\mu-2}+\norm{a_n-a_\infty}_{0,\mu-1})\\
&\leq  C( \norm{a_n\wedge a_n - a_\infty \wedge a_\infty}_{k,\mu-2} + \norm{a_n -a_\infty}_{0,\mu-1}) \\
&\leq C^\prime(\norm{a_n}_{k,-1}\norm{a_n - a_\infty}_{k,\mu-1} + \norm{a_\infty}_{k,-1}\norm{a_n-a_\infty}_{k,\mu-1} + \norm{a_n-a_\infty}_{0,\mu-1})
    \end{align*}
with the constants $C,C^\prime>0$ coming from the elliptic estimate for the smooth operator $D_B$ and the weighted Sobolev multiplication  \thref{weightedmult}. Now since $\{a_n\}_{n=1}^\infty $ converges in $L^2_{k,\mu-1}$ and there is a continuous embedding $L^2_{k,\mu-1}\hookrightarrow L^2_{k,-1},$ the sequence $\norm{a_n}_{k,-1}$ is bounded independent of $n$ and therefore $\norm{a_n-a_\infty}_{k+1,\mu-1}\to 0.$ Since the spaces $\mathcal{M}(A_\infty,\mu)_k$ and $\mathcal{M}(A_\infty,\mu)_{k+1}$ have the same convergent sequences, the inclusion is a homeomorphism.

\end{proof}
Note the weighted Sobolev embedding theorem ensures the spaces $\mathcal{M}(A_\infty,\mu)$ consists of smooth connections.  This yields the following important corollary; the moduli space near $[A]$ is modelled  on a neighbourhood of $0$ in the zero set of the non-linear elliptic operator $\mathcal{D}_A$ in any Sobolev extension.
 \begin{corollary}\thlabel{modulihomeo}
Let $A$ be an AC $G_2$-instanton, asymptotic to $A_\infty$ with rate $\mu$ where $-5<\mu<0.$ The zero set of the operator $\mathcal{D}_A\colon L^2_{k+1,\mu-1}(\slashed{S}(M) \otimes \text{\normalfont Ad} P) \to L^2_{k,\mu-2}(\slashed{S}(M) \otimes \text{\normalfont Ad} P) \  $ is independent of $k\geq 3$ and a neighbourhood of $A$ in $\mathcal{M}(A_\infty, \mu)$ is homeomorphic to a neighbourhood of $0$ in $(\mathcal{D}_A)_{\mu-1}^{-1}(0).$
 \end{corollary}

\subsection{Fredholm and Index Theory of the Twisted Dirac Operator}
 Recall that the operator $(D_A)_{\mu-1}$ is Fredholm when there are no non-zero solutions to $$D_{A_C}(r^{\mu-1}s_\Sigma)=0$$ where $s_\Sigma$ is a spinor on $\slashed{S}(\Sigma)$, lifted to $\slashed{S}(C).$  We will find an expression for the Dirac operator on the cone in terms of the operator on the link $\Sigma$ to determine the solutions to this equation.

We begin by comparing the Dirac operators on the link $\Sigma$ and the cone $C.$ Let $\Sigma$ a nearly K\"ahler  6-manifold and $(C,g_C)=(\Sigma\times \R_{>0},dr^2+r^2g)$ be the cone on $\Sigma.$  We consider first the case of the ordinary Dirac operators
\begin{align*}
    D^0&\colon \Gamma(\slashed{S}(\Sigma)) \to \Gamma (\slashed{S}(\Sigma)) \\
    D^C&\colon \Gamma(\slashed{S}(C)) \to \Gamma(\slashed{S}(C))
\end{align*}
arising from the Levi-Civita connection acting on the spin bundle.

 Let $e^i$ be a local orthonormal frame for $T^*\Sigma,$ then $E^0=dr,$ $E^i=re^i$ forms a local orthonormal frame for $T^*C.$  We  use the convention that an index of $\mu$ or $\nu$ runs from $0$ to $6$ whilst an index of $i,j$ or $k$ runs from $1$ to $6.$
 Denote by $\partial_i$ differentiation with respect to the vector field dual to $e^i$ using the metric $g$ and denote by $D_i$ differentiation with respect to the vector field dual to $E^i$ with respect to the metric $g_C.$
 We write $\nabla_C$ for the Levi-Civita connection on the cone and $\nabla$ for the Levi-Civita connection on $\Sigma$, similarly we write $\diff_C, \diff$ for the exterior derivatives on the cone and $\Sigma$ respectively.
 
 The connection one form of the Levi-Civita connection on the cone should be metric compatible  and torsion free, which here means that $\diff_C =\nabla_C \wedge.$ Let $\omega^i_j$ be the connection one form of the Levi-Civita connection on $T^*\Sigma,$ so that $\nabla e^i=-\omega^i_je^j$ in this frame. Then $\omega^i_j = - \omega^j_i$ and 
\begin{equation}\label{eq:dwedge}
\diff e^i+\omega^i_j\wedge e^j=0
\end{equation}
(this is equivalent to $\diff=\nabla \wedge$). \\
On the cone we let $\Omega^\mu_\nu=-\Omega^\nu_\mu$ be the Levi-Civita form, then testing that \eqref{eq:dwedge} holds  we first note that  $\diff E^0=\diff^2 r=0$ so that $\Omega^0_i\wedge E^i=0,$ and testing when $i\geq 0$ we find
\begin{align*}
    \diff_C E^i&=-\Omega^i_\mu\wedge E^\mu \\
    &= -\Omega^i_j\wedge re^j -\Omega^i_0\wedge \diff r 
\end{align*}
and also 
\begin{align*}
\diff_C E^i&=\diff r \wedge e^i + r\diff e^i\\
&=\diff r \wedge e^i - \omega^i_j \wedge e^j.
\end{align*}
Since we know \eqref{eq:dwedge} holds we conclude that 
$$ \Omega^i_j=\omega^i_j, \qquad \Omega^i_0=e^i.$$
The connection acts on one forms as 
$$\nabla_C v_\mu e^\mu=\left(\diff_C v_\mu - v_\nu\Omega^\nu_\mu\right) \otimes e^\mu.$$
Let $\Gamma_{\sigma \nu}^\mu$ be the Christoffel symbols of the Levi-Civita connection on the cone, so that
$$E^\sigma \Gamma_{\sigma \nu}^\mu=\Omega^\mu_\nu$$
and $\gamma_{kj}^i$ be the Christoffel symbols on the Levi-Civita connection on $\Sigma$ so that
$$e^k\gamma_{kj}^i=\omega_j^i.$$ 

Note that $\Gamma_{0\nu}^\mu=0$ and $\Gamma_{k0}^i=\frac{1}{r}\delta^{ik}.$  We take the natural Clifford algebra embedding of $\Cl_n$ into $\Cl_{n+1}^\text{even}$ given by $e^i \mapsto E^iE^0.$ The action of the Dirac operator on $\Sigma$ is 
\begin{equation}\label{eq:diraclink} D^0s=E^iE^0\left( \partial_is + \frac{1}{4}\gamma^k_{ij}E^0E^jE^0E^k s\right).\end{equation}
The operator on the cone acts as 
\begin{align*}
D^C(s)&=E^0\nabla_0s +E^i\left(D_is + \frac{1}{4} \Gamma^\nu_{i\mu}E^\mu E^\nu s\right) \\
&=E^0\frac{\partial s }{\partial r} + E^i\left( D_is + \frac{1}{4} \left( \Gamma^j_{i0} E^0E^js + \Gamma^0_{ij}E^jE^0s + \Gamma^k_{ij}E^jE^ks\right) \right) \\
&=E^0\frac{\partial s}{\partial r} + E^i\left( \frac{1}{r}\partial_i s + \frac{1}{4r} \left(\delta_{ij}E^0E^js-\delta_{ij}E^jE^0s+ \gamma_{ij}^kE^jE^ks \right) \right) \\
&=E^0\frac{\partial s}{\partial r} - \frac{1}{2r}E^iE^iE^0s+\frac{1}{r}E^i\left( \partial_is + \frac{1}{4}\gamma_{ij}^kE^0E^jE^0E^ks\right). 
\end{align*}
Thus we have that 
$$ D^C=E^0\frac{\partial s}{\partial r} + \frac{3}{r}E^0s+ \frac{1}{r}E^i\left(\partial_is + \frac{1}{4}\gamma_{ij}^kE^0E^jE^0E^ks\right) $$
and comparing to \eqref{eq:diraclink} one finds 
\begin{equation}
   D^C=E^0\cdot \left( \frac{\partial }{\partial r} + \frac{1}{r}\left( 3 + D^0 \right) \right) .
\end{equation}
\begin{remark}
The above calculation generalises easily to the case of $(X^n,g)$ an $n$-dimensional Riemanninan manifold and $(C(X),\diff r^2+r^2g)$ the $(n+1)$-dimensional cone of $X.$
\end{remark}
We would like to generalise this to the case of twisted spinors and twisted Dirac operators. Recall the notation $A_C=\pi^*A_\infty$ and consider the operators
\begin{align*}
    D^0_{A_\infty}&\colon \Gamma(\slashed{S}(\Sigma)\otimes \Ad Q) \to \Gamma (\slashed{S}(\Sigma)\otimes \Ad Q) \\
    D_{A_C}&\colon \Gamma(\slashed{S}(C)\otimes \Ad \pi^*Q) \to \Gamma(\slashed{S}(C)\otimes \Ad \pi^*Q).
\end{align*}
These operators factor as follows:
\begin{align*}
    D^0_{A_\infty}&=\text{cl}_6\circ (\nabla\otimes 1 + 1 \otimes \nabla^{A_\infty})=D^0\otimes 1 + \text{cl}_6\circ (1\otimes \nabla^{A_\infty}) \\
    D_{A_C}&=\text{cl}_7\circ (\nabla_C\otimes 1 + 1 \otimes \nabla^{A_C})=D^C\otimes 1 + \text{cl}_7\circ (1\otimes \nabla^{\pi^*A_\infty}) 
\end{align*}
it suffices consider the terms $\text{cl}_6(1\otimes \nabla^{A_\infty})$ and $\text{cl}_7(1\otimes \nabla^{A_C}) .$ We choose a local frame $\{v_a\}$ of $\Ad Q$ and  let $\tilde{\omega}_a^b=e^i\tilde{\gamma}_{ia}^b$ the the connection 1-form of $\nabla^A$ and let $\tilde{\Omega}_a^b=E^i\tilde{\Gamma}_{ia}^b$ be the connection 1-form  of $\nabla^{\pi^*A}.$ 
Then $\pi^*\tilde{\omega}=\tilde{\Omega}$ and we can apply a similar analysis to before. Let $s_A$ be a local frame for the spin bundle.  One finds 
$$D^0_{A_\infty}(f_{Aa}s_A\otimes v_a)=D^0(f_{Aa}s_A)\otimes v_a + E^iE^0f_{Aa}s_A\otimes \tilde{\gamma}_{ia}^bv_b$$
whilst 
\begin{align*}
    D_{A_C}(f_{Aa}s_A\otimes v_a) &= E^0\cdot \left( \frac{\partial}{\partial r}f_{Aa}s_A+ \frac{1}{r}(3+D^0)(f_{Aa}s_A)\right)\otimes v_a + (E^if_{Aa}s_A) \otimes \tilde{\Gamma}_{ia}^bv_b \\
    &= E^0\cdot \left( \frac{\partial}{\partial r}f_{Aa}s_A+ \frac{1}{r}(3+D^0)(f_{Aa}s_A)\right)\otimes v_a + \frac{1}{r}(E^if_{Aa}s_A) \otimes \tilde{\gamma}_{ia}^bv_b \\
    &=E^0\cdot \left( \frac{\partial }{\partial r} + \frac{1}{r}\left( 3 + D^0_{A_\infty} \right) \right)(f_{Aa}s_A\otimes v_a).
\end{align*}
So the twisted Dirac operators satisfy the same relation as in the case of the ordinary spin bundle.

A simple calculation shows that the volume form $\text{Vol}_6$ anti-commutes with the Dirac operator, so the spectrum is symmetric about $0$ and hence $W_\text{crit}(D_{A_C})$ is symmetric about $-3.$ Furthermore the spectrum of a Dirac operator is unbounded and discrete.\\

\begin{proposition}\thlabel{eigenvalueonlink}
The set of critical weights for the twisted  Dirac operator $D_A$ is $W_\text{crit}(D_{A_C})=\{ \mu \in \R \colon \mu + 3 \in \text{Spec}D^0_{A_\infty}\}$ where $D^0_{A_\infty}$ is the Dirac operator on the link twisted by the asymptotic connection $A_{\infty}.$ Furthermore this set is discrete, unbounded and symmetric about $-3.$  The operator $(D_A)_{\mu}$ is therefore Fredholm whenever $\mu+3 \in \R \setminus \text{Spec}D^0_{A_\infty}.$
\end{proposition}
Recall that in studying the deformation theory of $\mathcal{M}(A_\infty,\mu)$ we are lead to work with the operator $(D_A)_{\mu-1}$ on a $\mu-1$ weighted Sobolev space, since the kernel of this operator consists solutions to the linearised $G_2$-instanton equation converging with rate $\mu.$ Therefore let us set 
\begin{equation}
    W\coloneqq \{ \mu \in \R \, ; \, \mu + 2 \in \text{Spec}(D^0_{A_\infty}) \},
\end{equation}
so that $(D_A)_{\mu-1}$ is Fredholm whenever $\mu \in \R \setminus W$ and $W$ is symmetric about $-2.$

\begin{remark}
In \cite{charbonneau2016deformations} Charbonneau and Harland show that the linearised nearly K\"ahler instanton equation on the link $\Sigma$ is solved by one forms $a$ satisfying $D_{A_\infty}(a\cdot s)=2a\cdot s$ . Therefore we may think of rate $0 $ deformations on cone $C(\Sigma)$ as being deformations of the nearly K\"ahler instanton $A_\infty.$

\end{remark}
Recall the index of a Fredholm operator $P$ is the quantity $\text{ind}P=\text{dim}(\text{ker}P)-\text{dim}(\text{coker} P).$ The Lockhart McOwen theory provides an \emph{index change} formula that describes how the index varies as we vary the weight of the Sobolev space.
\begin{mydef}\thlabel{kmudef}
Let $\mu\in \R$ and define  
$$\mathcal{K}(\mu)=\left\{ \eta_C \in \text{\normalfont Ker}D_{A_C} \, ; \, \eta_C(r,\sigma)=r^{\mu-1}\sum_{j=0}^m(\log r)^j\eta_j(\sigma) \text{ \normalfont each  } \eta_j \text{\normalfont is a section of }\slashed{S}(\Sigma)\otimes \text{\normalfont Ad} Q\right\}.$$
That is, $\mathcal{K}(\mu)$ consists of sections in the kernel of $D_{A_C}$ which are polynomials in $\log r$ whose coefficients are homogeneous order $\mu$ spinors. 
\end{mydef}
The importance of these spaces is that they describe the change in index as the weight varies. We  let $k(\mu)$ be the dimension of the dimension of the space $\mathcal{K}(\mu)$ in \thref{kmudef}.
\begin{theorem}
 Let $\text{ind}_\mu$ denote the index of $D_A\colon L^2_{k+1,\mu-1}(\slashed{S}(M)\otimes \text{\normalfont Ad} P) \to L^2_{k,\mu-2}(\slashed{S}(M)\otimes \text{\normalfont Ad} P).$ If $\mu,\mu^\prime\in\R\setminus W$ are such that $\mu\leq\mu^\prime,$ then 
$$\text{ind}_{\mu^\prime}D_A-\text{ind}_{\mu}D_A=\sum_{\nu\in W\cap(\mu, \mu^\prime)}k(\nu).$$
\end{theorem}
\begin{remark}\thlabel{virtdim}
 Being self adjoint the operator $(D_A)_{-2}$ has index zero, since the dual of the target space has the same weight as the domain. It follows that when $D^0_{A_\infty}$ has non-trivial kernel we have  $ind(D_A)_{-2+\epsilon}=\frac{1}{2}(\text{dim Ker}D^0_{A\infty})$ for $\epsilon$ sufficiently small that $[-2,-2+\epsilon)$ contains no other critical weights. Observe that in this situation the virtual dimension is negative for all rates $\mu<-2.$
\end{remark}
In analogy with the results of \cite{marshal2002deformations} we show that being self adjoint ensures that elements of $\mathcal{K}(\mu)$  have no polynomial terms. 
\begin{proposition}
Suppose $$D_{A_C} \left(r^{\mu-1}\sum_{j=0}^m(\log r)^jv_j(\sigma)\right)=0.$$
Then $m=0.$
\end{proposition}
\begin{proof}
Let $v_C\in \mathcal{K}(\mu),$ then we may write
$$ v_C(r,\sigma)=r^{\mu-1}\sum_{j=0}^m(\log r)^jv_j(\sigma).$$
Suppose for a contradiction that $m>0$. Thinking of $D_{A_C}v_C$ as a polynomial in $\log r$ we first compare coefficients of $(\log r)^m$ to find 
$$D^0_{A_\infty}v_m=-(\mu+2)v_m.$$ Now comparing coefficients of $(\log r)^{m-1}$ we find 
$$(\mu+2)v_{m-1}+mv_m+D^0_{A_\infty} v_{m-1} =0$$
and we us the self-adjointness of $D^0_{A_\infty}v$ to compute
\begin{align*}
    m\langle v_m,v_m\rangle _{L^2(\Sigma)}&= -\langle D^0_{A_\infty} v_{m-1},v_m\rangle_{L^2(\Sigma)}-(\mu+2)\langle v_{m-1},v_m\rangle_{L^2(\Sigma)}\\
    &=-\langle v_{m-1},D^0_{A_\infty} v_m\rangle_{L^2(\Sigma)}-(\mu+2)\langle v_{m-1},v_m \rangle_{L^2(\Sigma)} \\
    &=-\langle v_{m-1},-(\mu+2)v_m\rangle_{L^2(\Sigma)}-(\mu+2)\langle v_{m-1},v_m\rangle_{L^2(\Sigma)}=0.
\end{align*}
Thus $v_m=0$ which yields our contradiction. 
\end{proof}
This proposition shows that $\mathcal{K}(\mu)$ is simply the $-(\mu+2)$ eigenspace for the operator $D^0_{A_\infty}.$ The dimension of these spaces therefore determines how $\text{ind}_\mu D_A$ varies as we change the weight $\mu.$

\subsection{Structure of the Moduli Space}

Suppose we work in a small enough neighbourhood of $A$ in the space of $L^2_{k , \mu-1}$ connections so that we may solve the Coulomb gauge condition. Then we have seen that the the zero set of $(\mathcal{D}_A)_{\mu -1}$ consists of smooth sections and that its linearisation is Fredholm whenever $(\mu + 2 )\not\in \text{Spec}D^0_{A_\infty}.$
\begin{mydef}
For a given weight $\mu <0$ we define the rate $\mu$ infinitesimal deformation space to be
$$\mathcal{I}(A, \mu) \coloneqq \left\{ (f,a)\in\Omega^0_{k+1,\mu-1}(M, \text{\normalfont Ad} P) \oplus \Omega^1_{k+1,\mu-1}(M, \text{\normalfont Ad} P) \, ; \, D_A(f,a)=0 \right\}.$$
By AC uniform elliptic regularity this is independent of $k$ and is finite dimensional. 
\end{mydef}

When $(D_A)_{k+1,\mu-1}$ is Fredholm it  has closed range and finite dimensional kernel, furthermore we can choose a finite dimensional subspace $\mathcal{O}(A, \mu)$ of $\Omega^0_{k,\mu-2}(M, \Ad P) \oplus \Omega^1_{k,\mu-2}(M, \Ad P)$, called the obstruction space, such that 
$$\Omega^0_{k,\mu-2}(M, \Ad P) \oplus \Omega^1_{k,\mu-2}(M, \Ad P)= D_A \left( \Omega^0_{k+1,\mu-1}(M, \Ad P) \oplus \Omega^1_{k+1,\mu-1}(M, \Ad P) \right) \oplus \mathcal{O}(A, \mu).$$
Again by elliptic regularity we have that $\mathcal{O}(A, \mu)$ is isomorphic to the kernel of the adjoint map  $(D_A)_{l+1, -5-\mu}$ for any $l \in \mathbb{N}$ . If $-\frac{5}{2}<\mu<0$ then the kernel of the adjoint is contained in the target space and we may choose $\mathcal{O}(A, \mu)=\text{ker}(D_A)^*.$

\begin{remark}
One can also describe the framework for study deformations of $G_2$-instantons in the form of an elliptic complex. In this weighted setting, the complex takes the form \\
\begin{figure}[H] \label{fig:chaincomplex}
\centering
\makeatletter
\tikzset{join/.code=\tikzset{after node path={%
 \ifx\tikzchainprevious\pgfutil@empty\else(\tikzchainprevious)%
 edge[every join]#1(\tikzchaincurrent)\fi}}} 
\makeatother
  \tikzset{>=stealth',every on chain/.append style={join},
  every join/.style={->}} 
\centering
  \begin{tikzpicture}[start chain] {
    \node[on chain, join={node[above] }] {$\Omega^0_\mu(M, \Ad P)$};
     \node[on chain, join={node[above] {$\diff_A$}}] {$\Omega^1_{\mu-1}(M, \Ad P)$};
      \node[on chain, join={node[above] {$\psi \wedge \diff_A$}}] {$\Omega^6_{\mu-2}(M, \Ad P)$};
       \node[on chain, join={node[above] {$\diff_A$}}] {$\Omega^7_{\mu-3}(M, \Ad P)$.};
    }
  \end{tikzpicture}
  \end{figure}
\flushleft
Denote the cohomology groups of this complex by $H^k_{A,\mu}.$
We have already noted that the zeroth cohomology group is trivial $H^0_{A,\mu}=\{0\}$, whilst $H^1_{A,\mu}\cong \mathcal{I}(A,\mu)$ and $H^2_{A,\mu}\cong \mathcal{O}(A,\mu)$ provided $-5<\mu<0.$

\end{remark}

With this in hand we can apply the implicit function theorem to integrate our infinitesimal deformation theory. We state here the version of the implicit function theorem we shall apply.
\begin{theorem}
Let $\mathcal{X}$ and $\mathcal{Y}$ be Banach spaces and let $\mathcal{U} \subset \mathcal{X}$ be an open neighbourhood of $0.$ Let $\mathcal{F} \colon \mathcal{U} \to \mathcal{Y}$ be a $C^k$ map, for some $k\geq 1$, with  $\mathcal{F}(0)=0.$ Suppose $d\mathcal{F}|_0 \colon \mathcal{X} \to \mathcal{Y}$ is surjective with kernel $\mathcal{K}$ such that $\mathcal{X}=\mathcal{K}\oplus \mathcal{Z}$ for some closed subspace $\mathcal{Z}.$ \\
Then there exists open sets $\mathcal{V} \subset \mathcal{K}$ and $\mathcal{W} \subset \mathcal {Z}$ both containing $0,$ with $\mathcal{V}\times \mathcal{W} \subset \mathcal{U}$ and a unique $C^k$ map $\mathcal{G} \colon \mathcal{V} \to \mathcal{W}$ such that 
$$\mathcal{F}^{-1}(0) \cap (\mathcal{V} \times \mathcal{W} ) = \{(x,\mathcal{G}(x))\colon x \in \mathcal{V}\}.$$
\end{theorem}
\begin{theorem}\thlabel{moduli}
Let $A_\infty$ be a nearly K\"ahler instanton and let  $A$ be an AC $G_2$-instanton converging to $A_\infty.$ Suppose that $\mu \in (\R \setminus W)\cap (-5,0).$ There exists a smooth manifold $\hat{\mathcal{M}} (A, \mu)$, which is an open neighbourhood of
$0$ in $\mathcal{I}(A, \mu)$, and a smooth map $\pi \colon  \hat{\mathcal{M}} (A, \mu) \to \mathcal{O}(A, \mu)$, with $\pi(0) = 0$, such
that an open neighbourhood of $0$ in $\pi^{-1}(0)$ is homeomorphic to a neighbourhood of $A$ in $\mathcal{M}(A_\infty, \mu ).$ Thus, the virtual
dimension of the moduli space is $\text{dim} \mathcal{I}(A, \mu) - \text{dim} \mathcal{O}(A, \mu)$ and $\mathcal{M}(A_\infty, \mu)$ is
smooth if $\mathcal{O}(A, \mu) = \{0\}.$
\end{theorem}
\begin{proof}
  For $k \geq 5$ let 
  $$\mathcal{X} = \left( \Omega^0_{k+1,\mu-1}(M, \Ad P) \oplus \Omega^1_{k+1,\mu-1}(M, \Ad P)\right) \times \mathcal{O}(A, \mu)$$ 
  and let 
  $$\mathcal{Y}= \Omega^0_{k,\mu-2}(M, \Ad P) \oplus \Omega^1_{k,\mu-2}(M, \Ad P).$$
Pick a sufficiently small neighbourhood of $A$  so that we may solve the Coulomb gauge condition. This in turn gives an open neighbourhood $\mathcal{U}$ of $(0,0)$ in $ \Omega^0_{k,\mu-1}(M, \Ad P) \oplus \Omega^1_{k,\mu-1}(M, \Ad P).$\\
  We define a map of Banach spaces $\mathcal{F} \colon \mathcal{X} \to \mathcal{Y}$ by 
  $$ \mathcal{F}(v,w)=\mathcal{D}_A(v)+w$$
  and note $\mathcal{F}(0,0)=0.$ The differential of $\mathcal{F}$ at $0$ acts as 
  \begin{align*} 
  d\mathcal{F}|_{(0,0)} &\colon \mathcal{X} \to \mathcal{Y} \\
  (v,w) &\mapsto D_Av +w.
  \end{align*}
By the definition of the obstruction space $d\mathcal{F}|_{(0,0)}$ is surjective and $d\mathcal{F}|_{(0,0)}(v,w)=0$ if and only if 
$(D_Av,w)=(0,0).$
Thus ker$d\mathcal{F}|_{(0,0)}= \mathcal{K}=\mathcal{I}(A, \mu)\times \{0\}$ is finite dimensional and splits $\mathcal{X}.$ That is, there exists a closed $\mathcal{Z}\subset \mathcal{X}$ such that $\mathcal{K}\oplus\mathcal{Z}=\mathcal{X}$ and we can moreover write $\mathcal{Z}=\mathcal{Z}_1\times \mathcal{O}(A, \mu)$ for some closed $\mathcal{Z}_1\subset \Omega^0_{k+1,\mu-1}(M, \Ad P) \oplus \Omega^1_{k+1,\mu-1}(M, \Ad P).$ We are now in a position to apply the implicit function theorem-- we deduce that there are open sets 
\begin{align*}
\mathcal{V} &\subset \mathcal{I}(A, \mu) \\
\mathcal{W}_1&\subset \mathcal{Z}_1 \\
\mathcal{W}_2& \subset \mathcal{O}(A, \mu)
\end{align*}
and smooth maps $\mathcal{G}_j\colon \mathcal{V} \to \mathcal{W}_j$ for $j=1,2$ such that 
$$\mathcal{F}^{-1}(0) \cap \left( \left( \mathcal{V} \times \mathcal{W}_1\right) \times \mathcal{W}_2 \right) = \left\{ \left( \left( v,\mathcal{G}_1(v)\right),\mathcal{G}_2(v) \right) \colon v \in \mathcal{V} \right\} $$
in $\mathcal{X}=\left( \mathcal{I}(A, \mu ) \oplus \mathcal{Z}_1 \right) \times \mathcal{O}(A, \mu).$ Therefore the kernel of $\mathcal{F}$ near $(0,0)$ is diffeomorphic to an open subset of $\mathcal{I}(A,\mu)$ containing $0.$ \par
Define $\hat{\mathcal{M}}(A, \mu)=\mathcal{V}$ and $\pi \colon \hat{\mathcal{M}} (A, \mu) \to \mathcal{O}(A, \mu)  $ by $\pi(v)=\mathcal{G}_2(v).$ Then an open neighbourhood of $0$ in $(\mathcal{D}_A)^{-1}(0)$ is homeomorphic to an open neighbourhood of $0$ in $\pi^{-1}(0).$ Finally, \thref{moduli} says a neighbourhood of $0$ in the zero set of $\mathcal{D}_A$ is homeomorphic to a neighbourhood of $A$ in the moduli space. 
  
\end{proof}
\begin{remark}
 When $\mu \in W $ or when $\mathcal{O}(A,\mu)\neq\{0\}$ the moduli space may not be smooth, or may have larger than expected dimension. 
\end{remark}

\section{A Lichnerowicz Formula and Eigenvalue Bounds}\label{sec:lichnerowicz}

The results of the previous section inform us that the virtual dimension of the $G_2$-instanton moduli space is determined by (part of) the spectrum of a twisted Dirac operator on a nearly K\"ahler 6-manifold. The required eigenvalues lie in an interval determined by the fastest rate of converge of the example being considered. This section develops methods for determining  these eigenvalues. In all known examples of AC $G_2$-instantons the nearly K\"ahler link is a homogeneous space and the limiting connection is the canonical connection. To study these examples we are thus able to develop a representation theoretic approach. \par

In this section the link $\Sigma$ of the AC $G_2$-manifold $M$  is assumed to be a compact homogeneous nearly K\"ahler 6-manifold $\Sigma=G/H.$ We shall denote by $\widehat{G}$ the set of isomorphism classes of irreducible, complex  unitary representations of $G$ and for $\gamma \in \widehat{G}$ we write $(V_\gamma, \rho_{V_\gamma})$ to denote any class representative. Similarly $\widehat{H}$ denotes  the set of isomorphism classes of irreducible, complex  unitary representations of $H$ and for $\gamma \in \widehat{H}$ a class representative is denoted $(W_\gamma, \rho_{W_\gamma})$. \par
\subsection{Homogeneous Bundles}
When constructing examples of $G_2$-instantons on cohomogeneity one manifolds, it is natural to seek to extend the group action on the manifold and look for connections invariant under this action. The framework of homogeneous bundles is for this reason important when studying the examples of $G_2$-instantons considered in this article. We give here a brief introduction to homogeneous bundles, before revisiting in Section~\ref{sec:stdinstsection}.

Recall on a principal $K$-bundle $\pi \colon Q\to X$  the structure group $K$ acts on the right and for any $p\in P$ this action defines a homomorphism $K\cong Q_{\pi(p)}=\pi^{-1}(\{p\}).$
\begin{mydef}
 Let $Q\to G/H$ be a principal $K$-bundle. We say that $Q$ is a $G$-homogeneous $K$-bundle  if there is a lift of the natural left action of $G$ on $G/H$ to the total space $Q$ which commutes with the right action of $K$.
\end{mydef}
Let $Q$ be a homogeneous $K$-bundle over $G/H. $ Choose a point $q_0\in \pi^{-1}(\{eH\}),$ then for all $h\in H$, we see that $h\cdot q_0\in \pi^{-1}(\{eH\}).$ Thus, for each $h\in H$ there exists a unique $k\in K$ such that $h\cdot q_0=q_o\cdot k.$ This defines a map $\lambda \colon H \to K$ and it is not difficult to see that this is a homomorphism.  We call $\lambda$ the \emph{isotropy homomorphism}. This allows one to reconstruct the bundle $Q$ as follows: Consider the associated bundle 
$$Q\times_{(H,\lambda)}K=(G\times K)/\sim $$
where $\sim $ is the equivalence relation $(gh,k)\sim (g,\lambda(h)k)$ for all $g\in G, h\in H$ and $k\in K.$ One can then check that the map
$$ G\times_{(H,\lambda)}K\to Q, \qquad [(g,k)] \mapsto g\cdot q_0\cdot k$$
defines an isomorphism of principal bundles. Thus, $G$-homogeneous $K$-bundles over $G/H$ are determined by isotropy homomorphisms $\lambda \colon H \to K.$ More precisely, isomorphism classes of homogeneous $K$-bundles are in bijection with conjugacy classes of homomorphism $\lambda \colon H \to K.$\par
Suppose now we have a homogeneous $K$-bundle $Q\cong G\times_{(H,\lambda)}K$ and a representation $(V,\rho)$ of $K.$ Then the lift of the $G$-action to $Q$ endows the associated bundle $E=Q\times_{(K,\rho)}V$ with an an action of $G.$ Furthermore there is an isomorphism of homogeneous bundles 
$$E\cong G\times_{(H,\rho \circ \lambda) }V.$$

A section $S\in\Gamma(E)$ is then said to be invariant if, once viewed as and $H$-equivariant map $s\colon G \to V,$ it is constant.
We can also understand gauge transformations in the homogeneous setting. Let $Q=G\times_{(H,\lambda)}K$ be  a homogeneous $K$-bundle and consider gauge transformations of $Q$ as sections of $c(Q)=Q\times_{(K,c)}K$ where $c(k_1)k_2=k_1k_2k_1^{-1}.$ Using the isomorphism $c(Q)=G\times _{(H,c\circ \Lambda)}K$ we can define a $G$-invariant section of $c(Q)$ from an element $k$ of the centraliser of $\lambda(H)$ via the map
$$[g]\mapsto[g,k]$$
where $[\, \cdot \, ]$ denotes an equivalence class.

\subsection{Frobenius Reciprocity}

To calculate eigenvalues explicitly we utilise results from harmonic analysis. The main tool we shall require is the Frobenius reciprocity theorem which generalises the classical Peter-Weyl theorem to the space of sections of an associated vector bundle. \par

The Frobenius reciprocity theorem provides a method for decomposing $(\rho_L,L^2(G, V)_H)$ as a representation of $G.$ The left action of $G$ on this space gives a representation which is said to be \emph{induced} by the representation $\rho_V$ and which is denoted $\text{Ind}^G_H(V).$ Suppose we have an irreducible representation $V_\gamma$ of $G$ and a non-trivial element $\Phi$ of $\text{Hom}(V_\gamma,V)_H.$ Here $\text{Hom}(V_\gamma,V)_H$ denotes the space of $H$-equivariant maps $V_\gamma\to V$
$$ \text{Hom}(V_\gamma, V)_{H}=\{ \Phi \in \text{Hom}(V_\gamma, V) \, ; \, \Phi \circ \rho_{V_\gamma}(h)=\rho_V(h)\circ \Phi \text{ for all $h\in H$}\}.$$ Then  for any $v\in V_\gamma$ we have  a map
\begin{equation}\label{eq:frob}
V_\gamma \to L^2(G,V)_H, \qquad v \mapsto (g\mapsto \Phi(\rho_{V_\gamma}(g^{-1})v)).
\end{equation}
Frobenius reciprocity uses this construction to show that an irreducible representation $V_\gamma$ is contained in the induced representation if and only if $\text{Hom}(V_\gamma,V)_H\neq\{0\}$ and the multiplicity of $V_\gamma$ in the induced representation is $\text{dim}(\text{Hom}(V_\gamma,V)_H).$ Thus if we denote by $\text{Res}^G_H(V_\gamma)$ the restriction of $(\rho_\gamma,V_\gamma)$ to the subgroup $H$, we have 
$$\text{mult}(V_\gamma,\text{Ind}^G_H(V))=\text{mult}(\text{Res}^G_H(V_\gamma),V).$$
This construction enables one to decompose the space of sections of $E\coloneqq G\times_{H}V$ into an  orthogonal Hilbert sum

\begin{equation}\label{eq:hilbertsum}
L^2(E)\cong L^2(G,V)_H\cong\bigoplus_{\gamma \in \widehat{G}}\text{Hom}(V_\gamma,V)_H\otimes V_\gamma. 
\end{equation}
The element of $L^2(G,V)_H$ that one obtains via \eqref{eq:frob} from an element $\Phi \otimes v\in \text{Hom}(V_\gamma, V)_H\otimes V_\gamma$ will be denoted 
\begin{equation}\label{eq:frobfn}
    F^\gamma_{\Phi, v}(g)=\Phi(\rho_{V_\gamma}(g^{-1})v)
\end{equation}
Since $G$ is assumed to be compact any irreducible representation must be finite dimensional.
Furthermore, each summand in the above Hilbert space sum in fact lies in $C^\infty(G,E)_H.$

\subsection{The Family of Dirac Operators}
Let $\Sigma=G/H$ be a compact, homogeneous nearly K\"ahler 6-manifold. We choose to work with complex spinor bundle $\slashed{S}_\C(\Sigma)$ which is the associated bundle
$$\slashed{S}_\C(\Sigma)=G\times_{(H,\rho)}S. $$
Here $S$, which we refer to as spinor space, is a complex eight dimensional vector space which as an $H$-module is 
\begin{equation} \label{eq:spinorspace}
S=\C \oplus \mf{m}_\C^* \oplus \C
\end{equation}
where $\mf{m}_\C^*$ carries the adjoint action of $H$ and $\rho$ is the representation defined by this action. The splitting \eqref{eq:spinorspace} of $S$ as an $H$-module comes from the splitting $\slashed{S}_\C(\Sigma)=\Lambda^0_\C \oplus \Lambda^1_\C \oplus \Lambda^6_\C$ of the spinor bundle.  We shall consider a twisting of the spinor bundle by an associated bundle 
$$E=G\times_{(H, \rho_V)}V$$
constructed from some representation $V$ of $H$. Thus the twisted spinor bundle is given by 
$$\slashed{S}_\C(\Sigma)\otimes E=G\times_{(H,\rho_S\otimes \rho_V)}S\otimes V.$$ 
The space of twisted spinor fields, which is isomorphic to $L^2(G,S\otimes V)_H$,  carries several natural actions, including:
\begin{itemize}
    \item The left regular representation $\rho_L$ where $\rho_L (g)s(g^\prime) = s(g^{-1}g^\prime ), \, \, \forall g, g^\prime \in G, \, s \in L^2(G,S\otimes V)_H.$
    \item The right right representation $\rho_R$ where $\rho_R(g)s(g^\prime) = s (g^\prime g), \, \, \forall g, g^\prime \in G, \, s \in L^2(G,S\otimes V)_H.$
    \item The representation induced by the representation $\rho_V$    on $V,$ which we still denote as $\rho_V.$ 
\end{itemize}
We shall seek to apply tools from Harmonic Analysis to study the space of twisted spinor fields and the  family of twisted Dirac operators. When the operator is twisted by the canonical connection, that is $A_\infty=A_\text{can} $, the covariant derivative can be understood using the Frobenius reciprocity theorem. Recall that $D^1_{A_\text{can}}=\text{cl}\circ\nabla^{1,A_\text{can}}$ is constructed from the canonical connection acting both on the spinor bundle $A$ and the associated bundle $E.$ We can therefore understand this operator using this formalism since the covariant derivative defined by the canonical connection is determined by the right regular representation \cite{bar1992dirac}:
\begin{equation}\label{eq:nabla1can}
\nabla^{1,A_\text{can}}_Xs=\rho_R(X)s
\end{equation}
for $s\in L^2(G,S\otimes V)_H$ and where we think of a vector field $X$ as an element of $L^2(G,\mf{m})_H.$ Recall $\mf{g}$ is given the nearly K\"ahler metric \eqref{eq:killingform} and that $\{I_A\}$ denotes a basis for $\mf{g}$ such that $I_a$ for $1\leq a \leq 6$ forms a basis for $\mf{m}$ and $I_i$ for $7\leq i \leq \text{dim}(G)$ forms a basis of $\mf{h}.$ The basis $\{I_a\}$ defines a local orthonormal frame $\{e^a\}$ of $T^*\Sigma|_U$ \cite{harland2010yang} and we can think of these local 1-forms as equivariant maps from $\pi^{-1}U$ to $\mf{m}^*$. By \eqref{eq:nabla1can} the operator $D^{1}_{A_\text{can}}$ takes the form 
\begin{equation}\label{eq:dcaneqn}
    D^{1}_{A_\text{can}}=\text{cl}(e^a)\rho_R(e^a)
\end{equation}
when acting on elements of $L^2(G, S\otimes V)_H.$
Furthermore we know from \cite{charbonneau2016deformations} that the family of Dirac operators differ by the action of the 3-form $\text{Re}\Omega$on the spin bundle
$$D^t_{A_\text{can}}=D^1_{A_\text{can}}+\frac{3(t-1)}{4}\text{Re}\Omega.$$ 

By combining these facts with the Frobenius reciprocity theorem we can understand how each of these operators act with respect to the splitting of the space of sections given in \eqref{eq:hilbertsum}. 
The next result  is due to B\"ar, although the version presented here is generalised to the case of a twisted spinor. The proof goes through identically since in either case the covariant derivative being used in determined by the right regular representation.  
\begin{proposition}\thlabel{diraconhomspace}
Let $E$ be a vector bundle associated via a representation  $V$ of $H$ and split the space of sections 
$$L^2(\slashed{S}_\C(\Sigma)\otimes E)\cong L^2(G, S\otimes V )_H\cong \bigoplus_{\gamma \in \widehat{G}}\text{\normalfont Hom}(V_\gamma, S\otimes V)_H\otimes V_\gamma.$$
For any $\gamma \in \widehat{G}$ and for every $t\in \R$, the operator $D^t_{A_\text{can}}$ leaves invariant the space $\text{\normalfont Hom}(V_\gamma,S\otimes V)_{H}\otimes V_\gamma$ and 
\begin{equation}\label{eq:dtonhomspace}
D^{t}_{A_\text{can}}|_{\text{\normalfont Hom}(V_\gamma,S\otimes V)_{H}\otimes V_\gamma}=(D^t_{A_\text{can}})_\gamma \otimes \text{\normalfont Id} 
\end{equation}
where $(D^{t}_{A_\text{can}})_\gamma\colon \text{\normalfont Hom}(V_\gamma,S\otimes V)_{H} \to \text{\normalfont Hom}(V_\gamma,S\otimes V)_{H}$ is the operator
\normalfont 
\begin{align}\label{eq:diracholformaction}
(D^{t}_{A_\text{can}})_\gamma \Phi &= -\text{cl}(I_a)\cdot\left(  \Phi \circ \rho_{V_\gamma}(I_a)\right) +\frac{3(t-1)}{4}\text{\normalfont Re}\Omega\cdot \Phi.  
\end{align}
\end{proposition}
\begin{proof}
The formula is calculated for $D^1_{A_\text{can}}$ in \cite[Proposition 1]{bar1992dirac}.
The fact that the family of Dirac operators differ by the above multiple of $\text{Re}\Omega$ is calculated in \cite{charbonneau2016deformations}.
\end{proof}
Note that under the isomorphism $\text{Hom}(V_\gamma, S\otimes V)\cong V_\gamma^*\otimes S\otimes V$ an $H$-equivariant homomorphism corresponds to a vector fixed by the $H$ action and the formula for $(D^t_{A_\text{can}})_\gamma$ takes the form 
\begin{align}\label{eq:D^thomspaceformula}
    (D^{t}_{A_\text{can}})_\gamma  &= \text{cl}(I_a)\rho_{V_\gamma^*}(I_a)+\frac{3(t-1)}{4}\text{\normalfont Re}\Omega.
\end{align}
This point of view will be the one we adopt when calculating eigenvalues explicitly.

\subsection{Eigenvalue Bounds}
Recall $A$ is assumed to be an AC $G_2$-instanton with fastest rate of convergence $\mu_0<0$ and the limiting connection $A_\infty=A_\text{can}$ will be assumed to be the canonical connection living on some bundle associated via a representation of $H.$ We shall consider the family of moduli spaces $\mathcal{M}(A_\text{can},\mu)$ with $\mu$ ranging from the fastest rate of convergence of the example we are studying, to 0. We have seen the virtual dimension of these moduli spaces jumps as we pass through eigenvalues in this interval shifted by 2. Our method is to develop a representation theoretic Lichnerowicz formula to determine the eigenvalues of a related Dirac operator.\par 
 The relevant Lichnerowicz type formula is calculated for the action of the Dirac operator on $\Lambda^1(\Sigma) \subset \slashed{S}(\Sigma) $ in \cite{charbonneau2016deformations} and we shall build on this work. 
 The formula gives the square of the Dirac operator as a sum of Casimir operators, so we first remind the reader how these operators are constructed. Any representation $(V, \rho)$ of a Lie algebra $\mf{g},$ equipped with an invariant inner product $B,$ yields a quadratic Casimir operator $\rho(\text{Cas}_\mf{g})$ defined as $\rho(\text{Cas}_\mf{g})v=\rho(I_A)\rho(I_A)v$ for any $v\in V$ and where $I_A$ is an orthonormal basis for $\mf{g}.$ In the case at hand the metric on $\mf{g}$ is the nearly K\"ahler metric $B(X,Y)=-\frac{1}{12}\text{Tr}_{\mf{g}}(\ad (X) \ad (Y))$ and the metrics on $\mf{m}$ and $\mf{h}$ are the restrictions of $B$.
 Since the Casimir operators commute with the group action they act as multiples of the identity on irreducible representations and so we can write
\begin{align*}
    \rho_{\gamma}(\text{Cas}_{\mf{g}})&=c^{\mf{g}}_{\gamma}\text{Id} \\
    \rho_{\gamma}(\text{Cas}_{\mf{h}})&=c^{\mf{h}}_{\gamma}\text{Id}.
\end{align*}
These  eigenvalues are calculated, for $G$ and $H$ such that $G/H$ is a nearly K\"ahler 6-manifold and using the above metric,  in \cite{charbonneau2016deformations}. \par
To obtain a suitable Lichnerowicz formula  we combine  \cite[Proposition 8]{charbonneau2016deformations} with the results of \cite{moroianu2010hermitian} to obtain a formula for $(D^\frac{1}{3}_{A_\text{can}})^2.$ The operator we would like to calculate eigenvalues for is $D^0_{A_\text{can}},$ however we are unable to calculate these eigenvalues directly using our Lichnerowicz formula. We shall see that we are nonetheless  able to gain useful information in the form of eigenvalue bounds.   
 \begin{lemma}\thlabel{lichnerowicz}
Let $G/H$ be a homogeneous nearly K\"ahler manifold. Let $E$ be the vector bundle  obtained from $G\to G/H$ through some representation $V$ of $H$.  Let $A_\text{\normalfont can}$ be the canonical  connection on $E$, then  $(D^\frac{1}{3}_{A_\text{can}})^2$ preserves the decomposition $\Gamma(\slashed{S}_\C(\Sigma)\otimes E)=\Gamma(\Lambda^0_\C\otimes E)\oplus \Gamma(\Lambda^1_\C\otimes E)\oplus \Gamma(\Lambda^6_\C\otimes E)$ and \normalfont 
 \begin{align}\label{eq:square}
     ( D^{\frac{1}{3}}_{A_\text{can}})^2\eta&= ( -\rho_L(\text{Cas}_\mf{g})\eta+\rho_V(\text{Cas}_\mf{h})\eta+4\eta)
 \end{align}
\textit{ for any $\eta \in \Gamma(\slashed{S}_\C(\Sigma)\otimes  E)$.}
\end{lemma}

\begin{proof}
The  formula is calculated for sections of $\Lambda^1_\C \otimes E$ in \cite[Proposition 8]{charbonneau2016deformations} so we need only consider the case where $\eta \in (\Lambda^0_\C\oplus \Lambda^6_\C)\otimes E.$ It is also shown in \cite{charbonneau2016deformations}   that
$$ ( D^{\frac{1}{3}}_{A_\text{can}})^2 \kappa =(\nabla^{1,A_\text{can}})^*\nabla^{1,A_\text{can}}\kappa + 4 \kappa$$
for $\kappa\in \Gamma( (\Lambda^0_\C\oplus \Lambda^6_\C)\otimes E)$ and so we aim to show this admits the required Casimir expression.
It is a standard fact, see \cite{moroianu2010hermitian} for example,   that the rough Laplacian $(\nabla^{1,A_\text{can}})^*\nabla^{1,A_\text{can}}$ is identified with the action of a Casimir operator
$$(\nabla^{1,A_\text{can}})^*\nabla^{1,A_\text{can}}=-\rho_R(\text{Cas}_\mf{m})$$
on $L^2(G, (\Lambda^0(\R^6)^*\oplus \Lambda^6(\R^6)^*)\otimes V)_{H}.$ Note if $\kappa \in L^2(G,(\Lambda^0_\C(\R^6)^*\oplus \Lambda^6_\C(\R^6)^*)\otimes V)_{H}$ then we have $\rho_R(X)\kappa+\rho_V(X)\kappa=0$ for any $X\in \mf{h}$ and therefore $\rho_R(\text{Cas}_\mf{h})=\rho_V(\text{Cas}_\mf{h}).$ Combining this with the fact that $\rho_L(\text{Cas}_\mf{g})=\rho_R(\text{Cas}_\mf{g})$ yields  the result.
\end{proof}
 The operator $(D^\frac{1}{3}_{A_\text{can}})^2$ preserves the decomposition  \eqref{eq:hilbertsum}, so as in \thref{diraconhomspace} we can define an endomorphism $(D^\frac{1}{3}_{A_\text{can}})^2_\gamma$ of $\text{Hom}(V_\gamma,S\otimes V)_H$ such that 
\begin{equation}\label{eq:diracsquaredef}
(D^\frac{1}{3}_{A_\text{can}})^2_\gamma \otimes \text{Id}= (D^\frac{1}{3}_{A_\text{can}})^2|_{\text{Hom}(V_\gamma,S\otimes V)_H\otimes V_\gamma}.
\end{equation}
Since we are considering unitary representations the space  $\text{Hom}(V_\gamma, S\otimes V)_H$ carries a natural inner product given by 
$$\langle X, Y\rangle = \text{Tr}(X^*Y)$$
where $X^*$ is the hermitian adjoint with respect to the hermitian inner products on $V_\gamma$ and $S\otimes V.$ The self-adjointness of $(D^\frac{1}{3}_{A_\text{can}})^2$ ensures that the restriction to any of the subspaces $\text{Hom}(V_\gamma, S\otimes V)_H\otimes V_\gamma$ defines a hermitian opertor, hence it is diagonalisable with real eigenvalues. Furthermore the spectrum satisfies  $$\text{Spec}(D^{\frac{1}{3}}_{A_\text{can}})^2=\bigcup_{\gamma\in\widehat{G}}(D^\frac{1}{3}_{A_\text{can}})^2_\gamma .$$\par  
It follows from Frobenius reciprocity and \thref{lichnerowicz} that $(D^\frac{1}{3}_{A_\text{can}})^2_\gamma$  acts as the endomorphism 
\begin{equation}\label{eq:dthirdsquare}
(D^\frac{1}{3}_{A_\text{can}})^2_\gamma = -\rho_{V_\gamma^*}(\text{Cas}_\mf{g})+ \rho_V(\text{Cas}_\mf{h})+4. \end{equation}

\begin{corollary}\thlabel{eigenvalsofsquare}
Let $V_\gamma$ be an irreducible representation of $G,$ and let $V=\bigoplus_{\nu\in K}W_\nu$ be the decomposition of $V$ into irreducible representations of $H$, where $K$ is a finite sequence in $\widehat{H}.$ The eigenvalues  and multiplicities of the operator $(D^\frac{1}{3}_{A_\text{can}})_{\gamma}^2$ are 
\begin{center}
\begin{TAB}(r,0.5cm,1cm)[5pt]{|c|c|}{|c|c|} 
     Eigenvalue  & Multiplicity  \\
   
     $-c^{\mf{g}}_{\gamma}+c^\mf{h}_\nu + 4 $&$ 
     \text{\normalfont dim} \text{\normalfont  Hom}(V_{\gamma},S \otimes W_\nu)_H  $  \\
    \end{TAB}
\end{center}
where $c^{\mf{g}}_\gamma$ is the eigenvalue of the Casimir operator on the irreducible representation $V_\gamma$ with respect to the inner product $B$ from \eqref{eq:killingform} and  $c^{\mf{h}}_\nu$ is the eigenvalue of the Casimir operator on the irreducible representation $W_\nu$ with respect to this metric. The eigenvalues of $(D^\frac{1}{3}_{A_\text{can}})_\gamma $ are $\pm \sqrt{-c^{\mf{g}}_{\gamma}+c^\mf{h}_\nu + 4} $ and the $\sqrt{-c^{\mf{g}}_{\gamma}+c^\mf{h}_\nu + 4 }$ and $-\sqrt{-c^{\mf{g}}_{\gamma}+c^\mf{h}_\nu + 4} $ eigenspaces are isomorphic.
\end{corollary} 
\begin{proof}
We have a splitting 
$$\text{Hom}(V_{\gamma}, S\otimes V)_{H}= \bigoplus_{\nu \in K}\text{Hom}(V_\gamma,S\otimes W_\nu)_H$$
which is respected by the operator $(D^\frac{1}{3}_{A_\text{can}})^2_\gamma. $ 
Inspection of \eqref{eq:square} reveals that $(D^\frac{1}{3}_{A_\text{can}})^2_\gamma$ acts as a constant with the advertised value on each of these summands. \par 
It follows that the eigenvalues of $(D^\frac{1}{3}_{A_\text{can}})_\gamma$ are given by $\pm \sqrt{-c^{\mf{g}}_{\gamma}+c^\mf{h}_\nu + 4}$ (verification of this fact is given in \cite{bar1996dirac} for example). Furthermore one can check that $(D^\frac{1}{3}_{A_\text{can}})_\gamma$ anti-commutes with the action of the volume form $\text{Vol}$ on the spin bundle and since $\text{Vol}^2=-1$ one sees that the volume form provides an isomorphism between the $\sqrt{-c^{\mf{g}}_{\gamma}+c^\mf{h}_\nu + 4}$ and $-\sqrt{-c^{\mf{g}}_{\gamma}+c^\mf{h}_\nu + 4}$ eigenspaces.
\end{proof}

Note that the eigenvalues of $-\rho_{V_\gamma}(\text{Cas}_{\mf{g}})$ change as $\gamma$ varies, whereas the eigenvalue of   $\rho_{V}(\text{Cas}_{\mf{h}})+4$ are fixed. Furthermore, large dimensional irreducible representations $V_\gamma$ lead to the Casimir operator $-\rho_{V_\gamma}(\text{Cas}_\mf{g})$ having a large positive eigenvalue.  Intuitively  this should mean that, if $V_\gamma$ is a large dimensional irreducible representation, the eigenvalues of the operator $(D^{\frac{1}{3}}_{A_\text{can}})^2_\gamma$ are large. The following lemma allows us to compare eigenvalues of operators in this way.
\begin{lemma}\thlabel{eigenvaluecomparison}
Let $X$ and $Y$ be $n\times n$ Hermitian matrices with eigenvalues $\{\lambda^X_1, \ldots , \lambda^X_n\}$ and $\{\lambda^Y_1, \ldots,\lambda_n^Y\}$ respectively. Let $\{\lambda^{X+Y}_1, \ldots \lambda^{X+Y}_n\}$ be the set of eigenvalues of $X+Y.$ Then 
$$\min\{ |\lambda^{X+Y}_1|,\ldots , |\lambda^{X+Y}_n|\}\geq \left| \min\{|\lambda_1^X|,\ldots , |\lambda_n^X|\}-\max \{|\lambda^Y_1|,\ldots , |\lambda^Y_n|\} \right| .$$
\end{lemma}
\begin{proof}
See \cite[Lemma 5.3.4]{mythesis}.
\end{proof}

Recall the operators $(D^t_{A_\text{can}})_\gamma$ from \eqref{eq:dtonhomspace}, as was the case for $(D^\frac{1}{3}_{A_\text{can}})_\gamma$
these anti-commute with the volume form $\text{Vol}$ so have spectra symmetric about 0 and satisfy
$$\text{Spec}(D^t_{A_\text{can}})=\bigcup_{\gamma\in\widehat{G}}\text{Spec}(D^t_{A_\text{can}})_\gamma.$$
Recall that our task is to determine the eigenvalues of $D^0_{A_\text{can}}$ in the region $(\mu_0+2,2)$ where $\mu_0$ is the fastest rate of convergence of $A$, therefore we would like to know which representations $V_\gamma$ have eigenvalues in this region. \par 
Suppose that the operator $(D^{\frac{1}{3}}_{A_\text{can}})^2_\gamma$ has eigenvalues $\lambda_1^2, \cdots , \lambda_n^2,$ then the eigenvalues of $(D^{\frac{1}{3}}_{A_\text{can}})_\gamma$ are precisely $\pm\lambda_1, \cdots \pm\lambda_n$.   Now since 
\begin{equation}\label{eq:lcdirac}
(D^0_{A_\text{can}})_\gamma=(D^{\frac{1}{3}}_{A_\text{can}})_\gamma-\frac{1}{4}\text{Re}\Omega
\end{equation}
and the eigenvalues of $-\frac{1}{4}\text{Re}\Omega$ are contained in the set $\{\pm1,0\},$ we can apply \thref{eigenvaluecomparison} to obtain lower bound on the smallest positive eigenvalue of $(D^0_{A_\text{can}})_\gamma.$

\begin{theorem}\thlabel{eigenvaluelowerbound}
Let $V_\gamma$ be an irreducible representation of $G$. If the quantity 
$$L_\gamma \coloneqq \sqrt{\min_\nu \{ -c^\mf{g}_\gamma+c^\mf{h}_\nu +4\} }-1$$ 
is positive, then it is a lower bound on the smallest positive eigenvalue of $(D^0_{A_\text{can}})_\gamma.$
\end{theorem}

\begin{proof}
Inspection of \thref{eigenvalsofsquare} reveals that the smallest eigenvalue of $(D^\frac{1}{3}_{A_\text{can}})_\gamma$ is $\sqrt{-c^\mf{g}_\gamma+c^\mf{h}_\nu+4}.$ Since the eigenvalues of $-\frac{1}{4}\text{Re}\Omega$ are contained in the set $\{\pm1,0\}$ we can apply \thref{eigenvaluecomparison} to yield the result.
\end{proof}
With this in hand we can outline our strategy for calculating the eigenvalues of $D^0_{A_\text{can}}$ in $(\mu_0+2,2)$:
\begin{enumerate}
    \item Use \thref{eigenvaluelowerbound} to rule out the irreducible representations $V_\gamma$ such that $(D^0_{A_\text{can}})_\gamma$  has no eigenvalues in the interval $(m,2)$ where $m=\min\{-2,\mu_0+2\}$ and $\mu_0$ is the fastest rate of converge of the example we are studying (in every case but one we just need to consider eigenvalues in the region $[0,2)$).
    \item For the remaining representations, if all maps in $\text{Hom}(V_\gamma,S\otimes V)_H$ factor through $\Lambda^1_\C\otimes V \subset S\otimes V$ then $(D^0_{A_\text{can}})_\gamma=(D^\frac{1}{3}_{A_\text{can}})_\gamma$ on this space (this is thanks to \thref{eigenvals}) and \thref{eigenvalsofsquare} informs us of the eigenvalues. 
    \item If there are maps factoring through $(\Lambda^0_\C\oplus \Lambda^6_\C)\otimes E \subset S\otimes E $ then $\text{Re}\Omega$ acts non-trivially and one must work harder to calculate the relevant eigenvalues. 
\end{enumerate}
In the last case here, when $\text{Re}\Omega$ acts non-trivially, we shall attack the known examples using two different methods. When the nearly K\"ahler 6-manifold is  $\C \mathbb{P}^3$ we adapt the work of B\"ar \cite{bar1992dirac} and develop a representation theoretic framework to calculate directly the matrix (and hence the eigenvalues) of the Dirac operator on one of the homomorphism spaces from \eqref{eq:hilbertsum}  via the formula \eqref{eq:D^thomspaceformula}. On $S^6$ we instead write the Dirac operator as a sum of Casimir operators and this again allows us to calculate the matrix and eigenvalues of the Dirac operator on the relevant homomorphism spaces.

\section{Deformations of \texorpdfstring{$G_2$}{G2}-Instantons over the Bryant-Salamon Manifolds}\label{sec:bsinstsection}
Having  developed the deformation theory in Section~\ref{sec:gtonac} and provided a suitable framework for studying twisted Dirac operators on homogeneous nearly K\"ahler 6-manifolds in the previous section, we are now ready to apply these results to concrete examples. We do so by first considering known examples of Clarke \cite{clarke2014instantons}, Lotay-Oliveira \cite{lotay2018su2insantons} and Oliveira \cite{oliveira2014monopoles}, which live on the Bryant-Salamon manifolds. In each example that we study, the aim shall be to calculate the virtual dimension of the moduli space on which the connection lives. The results developed in the previous section are sufficient for this calculation, with one exception: In studying Oliveira's $G_2$-instanton on $\Lambda^2_-(S^4)$ we present a different method for calculating the eigenvalues of the operator $(D^0_{A_\text{can}})_\gamma$ given in \eqref{eq:lcdirac}, in the case where $V_\gamma=V_{(0,1)}$ is the standard representation of $\text{Sp}(2).$

\subsection{Clarke's Family of Instantons on \texorpdfstring{$\R^4 \times S^3$}{R4xS3}}\label{sec:clarke}
We begin by briefly introducing the construction of the Bryant-Salamon metric on $\R^4\times S^3$, based on the presentation given in \cite{lotay2018su2insantons}. The idea is to consider a 1-parameter family of $\text{SU}(3)$ structures $(\omega(t), \gamma_2(t))$ on $S^3 \times S^3=\text{SU}(2)^2$ such that, setting $\gamma_1(t) = J\gamma_2(t):$
$$ \phi = \diff t \wedge \omega(t) + \gamma_1(t), \qquad \psi = \frac{\omega^2(t)}{2} = \diff t \wedge \gamma_2(t) $$
defines a torsion free $G_2$-structure on $S^3 \times S^3 \times (0,T).$ To construct such an $\text{SU}(3)$-structure, let $I_i$ for $ i=1,2,3$ be a basis of $\mf{su}(2)$ with $[I_i,I_j]=2\epsilon_{ijk}I_k,$ then we may split $\mf{su}(2)^2=\mf{su}(2)^+\oplus \mf{su}(2)^-$  where $I_i^+= (I_i,I_i)$ and $I_i^-= (I_i,-I_i) $ provide bases for $\mf{su}(2)^+$ and $\mf{su}(2)^-$ respectively. We define $\eta_i^+$ and $\eta_i^-$ to be dual to $I_i^+$ and $I_i^-$ respectively. The Maurer-Cartan relations take the form 
\begin{align}
    \diff \eta_i^+ &=-\epsilon_{ijk} ( \eta_j^+\wedge \eta_k^+ + \eta_j^-\wedge \eta_k^+), \\
    \diff \eta_i^- &= -2\epsilon_{ijk}\eta_j^-\wedge \eta_k^+.
\end{align}
The group $\text{SU}(2)^3$ acts on $\text{SU}(2)^2$ as follows:
$$(g_1,g_2,g_3)\cdot (\tilde{g}_1,\tilde{g}_2)=(g_1\tilde{g}_1g_3^{-1}, g_2\tilde{g}_2g_3^{-1})$$
and under this action we can  impose $\text{SU}(2)^3$ symmetry on our $\text{SU}(3)$ structure; this yields the following expressions:
\begin{align}
    \omega &= 4XY\eta^-_i\wedge \eta_i^+\\
    \gamma_1&=8 Y^3\eta_1^-\wedge \eta_2^- \wedge \eta_3^- - 4 X^2Y\epsilon_{ijk}\eta_i^+\wedge \eta_j^+\wedge \eta_k^-\\
    \gamma_2&=8X^3\eta_1^+\wedge \eta_2^+ \wedge \eta_3^+ - 4 Y^2X\epsilon_{ijk}\eta^-_i\wedge \eta^-_j \wedge \eta_k^+
\end{align}
where $X$ and $Y$ are real valued functions of $t\in \R^+.$ The equations defining a torsion free $G_2$-structure, which in this situation are known as the Hitchin flow equations, are:
\begin{align}\label{eq:ode}
    \dot{X}=\frac{1}{2}\left( 1 - \frac{X^2}{Y^2}\right) , \qquad \dot{Y}=\frac{X}{Y}.
\end{align}
Set $Y=s$ and $X=sF(s)$, then \eqref{eq:ode} becomes 
\begin{equation}
\frac{\diff }{\diff s}(sF)=\frac{1-F^2}{2F}.
\end{equation}
Solutions are given by 
\begin{equation}\label{eq:Fsolution}
F(s)=\sqrt{\frac{1-c^3s^{-3}}{3}}
\end{equation}
where $c>0$ and $s\geq c.$This yields 
$$X(s)=\frac{s}{\sqrt{3}}\sqrt{1-c^3s^{-3}} \qquad \text{and } \qquad Y(s)=s.$$
Let us choose $c=1$ and define a coordinate $r\in [1, +\infty)$ as follows: recall $t$ is the arclength parameter along the geodesic parameterised by $s$, so we may set 
$$t(r)=\int_1^r\frac{\diff s}{\sqrt{1-s^{-3}}}$$
then our functions take the form 
$$X=\frac{r}{3}\sqrt{1-r^{-3}} \qquad \text{and } \qquad Y=\frac{r}{\sqrt{3}}.$$ \par

We now switch focus to Clarke's family of $G_2$-instantons. These live on the trivial  $\text{SU}(2)$-bundle  over the Bryant-Salamon manifold $\R^4\times S^3$. \par

We will denote by $P_\lambda$ the bundle over $\R^4\times S^3$ whose extension over the singular orbit is determined by an isotropy homomorphism  $\lambda \colon  \text{SU}(2) \to \text{SU}(2).$
Clarke's $G_2$-instantons live on the bundle $P_1$, the bundle whose isotropy homomorpshism is the trivial homomorphism $\lambda(g) \equiv 1.$ The connections are globally well-defined and take the form 
$$A=xX\left( \sum_{i=1}^3 I_i\otimes \eta_i^+\right) + yY\left( \sum_{i=1}^3I_i\otimes \eta^+_i \right) $$
with 
\begin{equation}
x(t)=\frac{2x_0X(t)}{1+x_0(Y^2(t)-\frac{1}{3})} \qquad \text{and} \qquad y(t)=0
\end{equation}
for some $x_0\in \R.$ The connections are generically AC with fastest rate of convergence $-2$, although note that when $x_0=0$ the connection $A$ is the trivial flat connection.

\subsection{Eigenvalues of the Twisted Dirac Operator on \texorpdfstring{$S^3\times S^3$}{S3xS3}}
Since Clarke's $G_2$-instantons are AC with fastest rate of convergence $-2$ and limiting connection $A_{\text{can}}$, the results of Section~\ref{sec:gtonac} tell us that the virtual dimension of $\mathcal{M}(A_\text{can},\mu),$ for $-2<\mu<0$, is determined by the eigenvalues of $D^0_{A_\text{can}}$ in the interval $[0,2).$ In this section the eigenvalues in this interval are calculated, using the method developed in \ref{sec:lichnerowicz}.\par

An asymptotic framing for the bundle $P_1$ is provided by the homogeneous bundle $Q=\text{SU}(2)^3\times_{(\Delta \text{SU}(2),\text{id})}\text{SU}(2).$ Away from the associative $S^3,$ which is the singular orbit for the action of $\text{SU}(2)^3,$ we have $P_1=\pi^*Q$ where $\pi\colon C(S^3\times S^3)\to S^3\times S^3$ is the projection map.  The bundle $\Ad Q$ is associated to the canonical bundle via the adjoint representation of $\text{SU}(2)$ on its Lie algebra. As a complexified representation the vector space is $(\mf{su}(2))_\C=\mf{sl}(2,\C).$ Let us denote by $W_i$ the unique irreducible representation of $\text{SU}(2)$ with dimension  $(i+1)$. Irreducible representations of $\text{SU}(2)^3$ are then given by $V_{(i,j,k)}\coloneqq W_i\otimes W_j\otimes W_k.$ As a representation of $\text{SU}(2)$ we have $\mf{m}^*_\C=W_2\oplus W_2$ so the complexified spinor space is $S=W_0\oplus W_2 \oplus W_2 \oplus W_0$ and an application of the Clebsch-Gordan rule shows that the twisted spinor space is 
\begin{equation}\label{eq:su(2)twistedspinorspace}
S\otimes \mf{su}(2)_\C = 2 W_4\oplus 4W_2 \oplus 2W_0.
\end{equation}
The Frobenius reciprocity theorem says that the space of sections splits as 
\begin{align}
L^2(\slashed{S}_\C(S^3\times S^3)\otimes \Ad Q )&\cong L^2(\text{SU}(2)^3, S\otimes \mf{su}(2)_\C)_{\text{SU}(2)}\nonumber \\
&\cong \bigoplus_{\gamma \in \widehat{\text{SU}(2)^3}}\text{Hom}(V_\gamma,S\otimes \mf{su}(2)_\C)_{\text{SU}(2)}\otimes V_{\gamma} \label{eq:su(2)twisteddecomp}
\end{align}
where $\widehat{\text{SU}(2)^3}$ denotes the set of isomorphism classes of irreducible representations of $\text{SU}(2)^3$ and $V_\gamma$ is a class representative for $\gamma.$
The action of the Casimir operators on irreducible representations is given by
\begin{align*}
    \rho_{V_{(i,j,k)}}(\text{Cas}_{\mf{su}(2)^3})&=c^{\mf{su}(2)^3}_{(i,j,k)}\text{Id}\\
    \rho_{W_i}(\text{Cas}_{\mf{su}(2)})&=c^{\mf{su}(2)}_i\text{Id}
\end{align*}
where these eigenvalues of the Casimir operators are with respect to the nearly K\"ahler metric. The eigenvalues are calculated in \cite{charbonneau2016deformations} to be 
\begin{align*}
    c^{\mf{su}(2)^3}_{(i,j,k)}&=-\frac{3}{2}\left( i(i+2) + j(j+2) + k(k+2)\right) \\
    c^{\mf{su}(2)}_i&=-\frac{1}{2}i(i+2).
\end{align*}

 Since the adjoint representation of $\text{SU}(2)$ is the 3-dimensional irreducible representation, we are able to combine \thref{lichnerowicz} with the above Casimir eigenvalue formula to see that the operators $(D^\frac{1}{3}_{A_\text{can}})^2_\gamma$ of $\text{Hom}(V_\gamma, S\otimes \mf{su}(2)_\C)_{\text{SU}(2)}$, which were defined in \eqref{eq:D^thomspaceformula}, take the form

\begin{equation} \label{eq:su(2)diracsquarehomspace}
(D^{\frac{1}{3}}_{A_\text{can}})^2_\gamma =-\rho_{V_\gamma^*}(\text{Cas}_{\mf{su}(2)^3}).
\end{equation}
Thus by the results of Section~\ref{sec:lichnerowicz}, we have found an expression for the eigenvalues of $ (D^{\frac{1}{3}}_{A_\text{can}})^2: $
\begin{lemma}\thlabel{su(2)lichnerowiczevals}
The spectrum of $(D^\frac{1}{3}_{A_\text{\normalfont can}})^2$ is 
$$\text{\normalfont Spec}(D^\frac{1}{3}_{A_\text{\normalfont can}})^2=\left\{ -c^{\mf{su}(2)^3}_{(i,j,k)} \, ; \, \text{\normalfont dim Hom}(V_{(i,j,k)},S\otimes \mf{su}(2)_\C)_{\text{\normalfont SU}(2)}\neq 0\right\}.$$
\end{lemma}
We can use this to rule out representations that do not lead to eigenvalues of the twisted Levi-Civita Dirac operator in the interval $[0,2).$ First let us recall the relation between the various operators that we need:\par
The operator $D^1_{A_\text{can}}$ is constructed from the canonical connection acting on both the spinor space and the adjoint bundle. We can define endomorphisms $(D^1_{A_\text{can}})_\gamma$ of $\text{Hom}(V_\gamma, S\otimes \mf{su}(2)_\C)_{\text{SU}(2)}$ such that 
 $$ (D^1_{A_\text{can}})_\gamma \otimes \text{Id}=(D^1_{A_\text{can}})|_{\text{Hom}(V_\gamma, S\otimes \mf{su}(2)_\C)_{\text{SU}(2)}\otimes V_\gamma}$$
 and thus a family of endomorphisms $(D^t_{A_\text{can}})_\gamma$ of $\text{Hom}(V_\gamma, S\otimes \mf{su}(2)_\C)$ via the formula 
 $$(D^t_{A_\text{can}})_\gamma= (D^1_{A_\text{can}})_\gamma +\frac{3(t-1)}{4}\text{Re}\Omega.$$
Then one has $$\text{\normalfont Spec}(D^t_{A_\text{can}})=\bigcup_\gamma \text{Spec}(D^t_{A_\text{can}})_\gamma$$ and we  use this to obtain eigenvalue estimates for the Levi-Civita Dirac operators $(D^0_{A_\text{can}})_\gamma$.
\begin{corollary}
If $V_\gamma\neq \C$ then $\text{\normalfont Spec}(D^0_{A_\text{\normalfont can}})_\gamma\cap [0,2)=\emptyset.$ 
\end{corollary}
\begin{proof}
Let $V_\gamma$ be an irreducible representation of $\text{SU}(2)^3$ and suppose that dim Hom$(V_\gamma, S\otimes \mf{sl}(2,\C))_{\text{SU}(2)}\neq 0.$  Application of \thref{eigenvaluelowerbound} shows that the smallest non-negative eigenvalue of $(D^0_{A_\text{can}})_\gamma$ is bounded below by 
$$L_\gamma=\sqrt{-c^{\mf{su}(2)^3}_{(i,j,k)}}-1$$
where $V_\gamma=V_{(i,j,k)}.$ The quantity $L_\gamma$ does not yield a lower bound when $V_\gamma=\C$ and since dim Hom$(V_{(1,0,0)},S\otimes\mf{sl}(2,\C))_{\text{SU}(2)}=0$ the representation $V_{(1,0,0)}$ need not be considered. The next representation to consider is $V_{(1,1,0)},$ this yields the bound $L_\gamma=\sqrt{9}-1=2,$ so this bound is sufficient for the statement of the theorem. Any other irreducible representation $V_\gamma$ leads to a large lower bound, which completes the proof.
\end{proof}
To calculate the relevant eigenvalues we therefore only need consider those coming from the trivial representation. Let $V_\gamma=\C$ and let us write $S=W_0^{(1)}\oplus W_2^{(1)}\oplus W_2^{(2)}\oplus W_0^{(2)}$ where $W_{i}^{(a)}$ are distinct copies of the irreducible representation $W_i.$  A basis of $\text{Hom}(\C , S\otimes \mf{su}(2)_\C)_{\text{SU}(2)}=\text{Hom}(\C, S\otimes W_2)_{\text{SU}(2)}$ is given by the $\text{SU}(2)$-equivariant maps $q^{2,2}_0$ and $\tilde{q}^{2,2}_0$ which factor as maps 
\begin{align*}
    q^{2,2}_0&\colon \C\to  W_2^{(1)}\otimes W_2\to S\otimes W_2 \\
    \tilde{q}^{2,2}_0& \colon \C \to W^{(2)}_2\otimes W_2 \to S\otimes W_2.
\end{align*}
Notice that all maps in this homomorphism space factor through $\Lambda^1\subset S.$
On this space \eqref{eq:su(2)diracsquarehomspace} says that $(D^\frac{1}{3}_{A_\text{can}})^2_\gamma\equiv 0.$ Furthermore $\text{Ker}(D^\frac{1}{3}_{A_\text{can}})^2_\gamma=\text{Ker}(D^\frac{1}{3}_{A_\text{can}})_\gamma$ so we see that $(D^\frac{1}{3}_{A_\text{can}})_\gamma$ also vanishes on this space. Finally, since all basis vectors factor through $\Lambda^1\subset S$ \thref{eigenvals} ensures that $\text{Re}\Omega$ acts trivially on this space. Therefore $(D^t_{A_\text{can}})_\gamma=0$ for all $t$, when $V_\gamma$ is the trivial representation. We have shown:
\begin{proposition}
The only eigenvalue of $D^0_{A_\text{can}}$ in the interval $[0,2)$ is $0$ and has multiplicity 2.
\end{proposition}
This result tells us, by the observation in  \thref{virtdim} that $\text{ind}_{-2+\epsilon}D_A=\frac{1}{2}\text{dim Ker}D^0_{A_\text{can}}$ for $\epsilon$ sufficiently small, the virtual dimension of the moduli space:
\begin{theorem}
Let $P$ be the trivial $\text{\normalfont SU}(2)$-bundle over $\R^4\times S^3$, framed at infinity by the homogeneous $\text{\normalfont SU}(2)$-bundle over the nearly K\"ahler $S^3 \times S^3$ whose isotropy homomorphism is the identity. Let $A_\text{\normalfont can}$ be the canonical connection on $Q$, so that Clarke's $G_2$-instantons $A$ live in $\mathcal{M}(A_\text{\normalfont can},\mu)$ for $\mu \in (-2,0)$. The virtual dimension of the moduli space this family of moduli spaces  is 
$$\text{\normalfont \scriptsize{virt}\normalsize dim}\mathcal{M}(A_\text{\normalfont can},\mu)= 1 $$
for all $\mu \in (-2,0).$ 
\end{theorem}
Note that the above virtual dimension coincides with the dimension of the $\text{SU}(2)^3$-invariant moduli space, which is known from the work of Clarke \cite{clarke2014instantons} and Lotay-Oliveira \cite{lotay2018su2insantons}. This is perhaps not surprising- the $\text{SU}(2)^3$-invariant deformation parameter also defines an AC deformation, so it would appear we are only seeing this in our calculation, although further work is required to rule out the case of having a larger kernel and a non-trivial cokernel compensating.
We shall observe similar behaviour in the other examples that we study- the virtual dimension of the AC moduli space matches the number of deformation parameters known from the invariant setting.

\subsection{The Limiting Connection of Lotay-Oliveira}

Consider now the bundle $P_\text{id}$ which extends the action over the singular orbit via the identity homomorphism $\text{id}\colon \text{SU}(2) \to \text{SU}(2).$ The $G_2$-instanton constructed by Lotay and Oliveira \cite[Theorem 5]{lotay2018su2insantons} can be considered as a limiting connection of Clarke's family. Furthermore, the authors interpret its existence as a removable singularity phenomena.  The connection is well defined on the entire bundle $P_\text{id}$ and takes the form 
$$A=xX\left( \sum_{i=1}^3 I_i\otimes \eta_i^+\right) + yY\left( \sum_{i=1}^3I_i\otimes \eta^+_i \right) $$
with 
$$x(t)=\frac{X(t)}{\frac{1}{2}(Y^2(t)-\frac{1}{3})} \qquad \text{and} \qquad y(t)=0.$$
This connection is asymptotically conical with fastest rate of convergence -3. 
Note that the $G_2$-instantons of Clarke and Lotay-Oliveira live on (topologically) the same bundle and both converge to the canonical connection on the asymptotic framing bundle $Q$ but (in the terminology of \eqref{eq:asympconnection}) have fastest rates of convergence $-2$ and $-3$ respectively. Thus the moduli space we shall consider is identical to that considered previously, although the range of weights we shall consider is now $(-3,0).$ Consequently we must determine eigenvalues of the twisted Dirac operator on the link which fall in the interval $(-1,2).$ Note however, that since the spectrum we consider is symmetric about 0, no extra work is required. Hence:
\begin{theorem}
Let $P$ be the trivial $\text{\normalfont SU}(2)$-bundle over $\R^4\times S^3$, framed at infinity by the homogeneous $\text{\normalfont SU}(2)$-bundle over the nearly K\"ahler $S^3 \times S^3$ whose isotropy homomorphism is the identity. Let $A_\text{\normalfont can}$ be the canonical connection on $Q$, so that Lotay-Oliveira's $G_2$-instanton $A$ lives in $\mathcal{M}(A_\text{\normalfont can},\mu)$ for $\mu \in (-3,0)$. The virtual dimension of the moduli space this family of moduli spaces  is 
$$\text{\normalfont \scriptsize{virt}\normalsize dim}\mathcal{M}(A_\text{\normalfont can},\mu)=\begin{cases}
1 & \text{if }\mu\in (-2,0)\\
-1&\text{if } \mu \in (-3,-2).\end{cases}$$
\end{theorem}
Observe that the above moduli space is necessarily obstructed when the weight is less than $-2.$ This reflects what we know from the general theory; suitably large negative weights will lead to the moduli space having a negative virtual dimension. This connection is the only known AC $G_2$-instanton that decays with rate faster than $-2$ and as we shall see it is also the only example that yields a negative virtual dimension.

\subsection{The Bryant-Salamon Manifolds \texorpdfstring{$\Lambda^2_-(N^4)$}{L2-(N4)}}

Next we turn our attention to the $G_2$-instantons constructed by Oliveria in \cite{oliveira2014monopoles}.  These are connections on the Bryant-Salamon manifolds $\Lambda^2_-(N^4),$ where $N^4$ is either $\mathbb{CP}^2$ or $S^4,$ so we first briefly recap the contruction of these manifolds. We consider one example on each manifold. Note in \cite{oliveira2014monopoles} Oliveira also constructs examples of $\text{SU}(3)$-instantons over $\Lambda^2_-(\mathbb{CP}^2)$ and shows that the pullback of an ASD instanton on $S^4$ defines a $G_2$-instanton on $\Lambda^2_-(S^4).$ These examples are not considered in this article, although the virtual dimension of the moduli space in the latter case is calculated in the author's thesis \cite{mythesis}.  \par
If $(N^4,g_N)$ is a 4-dimensional Riemannian manifold, then the twistor space of $N$ is the unit sphere bundle in $\Lambda^2_-(N).$ When $N=\mathbb{CP}^2$ or $N=S^4$ the twistor space carries a natural nearly-K\"ahler structure \cite{eells1985twistorial} whose metric is 
$$g_6 =\frac{1}{2}g_{S^2} + \pi^*g_N.$$ 
The Bryant-Salamon metric on the total space of $\Lambda^2_-(N)$ takes the form
\begin{equation}\label{eq:bsmetric}
    g = f^2(s)g_{\R^3} + f^{-2}(s)\pi^* g_N
\end{equation}
where $s$ is the Euclidean distance to the zero section given by the fiber metric and where
\begin{equation}\label{eq:bsf}
f(s) = (1 + s^2)^{-\frac{1}{4}}.
\end{equation}

\subsection{Oliveira's \texorpdfstring{$G_2$}{G2}-Instantons on \texorpdfstring{$\Lambda^2_-(\mathbb{CP}^2)$}{L2-(CP2)}}

Consider now the case when $N=\C\mathbb{P}^2,$ then the twistor space is the space of flags in $\C^3$ i.e the space of complex lines contained in planes  in $\C^3$ and will be denoted $\mathbb{F}_{1,2,3}.$ The standard action of $\text{SU}(3)$ on $\C^3$ gives rise to a transitive action of $\text{SU}(3)$ on the space of flags with isotropy subgroup  $T^2.$ At a suitable point of $\mathbb{F}_{1,2,3}$ the  $T^2$ isotropy subgroup of $\text{SU}(3)$ that one obtains is the standard one \cite{oliveira2014monopoles}
$$T^2=\left \{ \begin{pmatrix} e^{i\theta} & 0 & 0 \\
0 & e^{-i(\theta + \phi)} & 0 \\
0 & 0 & e^{i\phi} \end{pmatrix} \right\} 
$$
and we will work with this fixed subgroup throughout this section. \par
The geodesic distance to the zero section then takes the form $t(s)=\int_0^sf(u)\diff u$ and this allows us to rewrite the metric \eqref{eq:bsmetric} as 
$$ g= \diff t^2 +s^2(t)f^2(s(t))g_{S^2} + f^{-2}(s(t))\pi^*g_{\C\mathbb{P}^2}$$
where $g_{S^2}$ is the round metric on the unit sphere of a fibre in $\Lambda^2_-(\C \mathbb{P}^2).$ One can also describe the $G_2$-structure explicitly. To gain a suitable local expression we first study homogeneous structure of $\text{SU}(3)/T^2.$\par
Under the adjoint action the real representation $\mf{m}$ splits into three irreducible subspaces $\mf{m}=\mf{m}_1\oplus \mf{m}_2\oplus \mf{m}_3$ and we choose orthonormal bases $\{I_1, I_2\}, \{I_3, I_4\}$ and $\{I_5,I_6\}$ respectively (if the reader is interested this decomposition is covered more explicitly in \cite{oliveira2014monopoles}).This induces a local frame for $T^*\mathbb{F}_{1,2,3}$ which we denote $e^1, \ldots e^6.$ Away from the zero section of $\Lambda^2_-(\C \mathbb{P}^2)$ we can think of the geodesic distance $t$ as a coordinate. Let us define $2$-forms 
\begin{equation} \begin{array}{ccc}\label{eq:2forms}
\Omega_1=e^{12}-e^{34}, & \Omega_2=e^{13}-e^{42}, & \Omega_3=e^{14}-e^{23}\\
\overline{\Omega}_1=e^{12}+e^{34}, & \overline{\Omega}_2=e^{13}+e^{42}, & \overline{\Omega}_3=e^{14}+e^{23}
\end{array}
\end{equation}
then the 3-form $\varphi$ of the Bryant-Salamon $G_2$-structure takes the form 
\begin{equation}
    \varphi = \diff t \wedge (a^2(t)e^{12}  -a^2(t)e^{34}+c^2(t)e^{56})+ a^2(t)c(t)(\Omega_2\wedge e^6-\Omega_3\wedge e^5)
\end{equation}
where $a(t)=2f^{-1}(t)$ and $c(t)=2s(t)f(t).$  From this viewpoint one sees that the $G_2$-structure is AC with $|g-g_C|_g=O(t^{-4}).$ \par

Let us fix the standard basis $I_i$ of $\mf{so}(3)$ with $[I_i,I_j]=2\epsilon_{ijk}I_k.$ We choose  the maximal torus in SO$(3)$ to be the one generated by $\frac{1}{2}I_1$. Recall that $\text{SU}(3)$-homogeneous  $\text{SO}(3)$-bundles over $\mathbb{F}_{1,2,3}$ are determined by isotropy homomorphisms from $T^2$ to $\text{SO}(3).$ It is shown in \cite{oliveira2014monopoles} that there is a unique isotropy homomorphism $\lambda$ which yields a bundle $Q=\text{SU}(3)\times_\lambda \text{SO}(3)$ with non-trivial invariant connections and for which the pullback $\pi^*Q$ extends over the zero section $\C \mathbb{P}^2$ of $\Lambda^2_-(\C \mathbb{P}^2).$ \par
To describe this bundle we first note that the singular orbit of $\Lambda^2_-(\C \mathbb{P}^2)$ is $\text{SU}(3)/\text{U}(2)=\C \mathbb{P}^2$ and so $\text{SU}(3)$-homogeneous $\text{SO}(3)$-bundles over the singular orbit are determined by isotropy homomorphisms $\tilde{\lambda}\colon \text{U}(2)\to \text{SO}(3).$ The adjoint action of $\text{U}(2)$ on $\mf{su}(2)$ defines a homomorphism 
\begin{align}
    \begin{split} 
    \tilde{\lambda} \colon  \text{U}(2) &\to \text{SO}(3)\\
    g&\mapsto  \Ad_g 
    \end{split}
    \label{eq:isotropyhomt2u2}
\end{align}
where we view  $\text{SO}(3)$ as a subgroup of $\text{GL}(\mf{su}(2)).$
By viewing $T^2$ as the subgroup of $\text{U}(2)$
$$T^2=\begin{pmatrix} 
e^{i\theta} & 0 \\
0 & e^{i\phi}\end{pmatrix}\subset \text{U}(2)$$
we obtain a homomorphism $T^2 \to \text{SO}(3)$ by restriction of \eqref{eq:isotropyhomt2u2}, this is
\begin{align}
    \begin{split} 
    \lambda \colon  T^2 &\to \text{SO}(3)\\
    g&\mapsto  \Ad_g .
    \end{split}
    \label{eq:isotropyhomt2}
\end{align}
We take $\lambda$ to be the isotoropy homomorphism of the homogeneous bundle
\begin{equation}\label{eq:su3hombundle}
    Q=\text{SU}(3)\times_\lambda \text{SO}(3).
\end{equation}
Let us denote by $P$ the extension of  $\pi^*Q$ to $\Lambda^2_-(\C \mathbb{P}^2)$  determined by the isotropy homomorphism \eqref{eq:isotropyhomt2u2}. We denote by $A_\text{can}$ the pullback of the canonical connection on $Q$, this  lives  on $P|_{\Lambda^2_-(\C \mathbb{P}^2)\setminus \C \mathbb{P}^2}.$ An invariant connection on $P$ then takes the form 
$$A=A_\text{can}+ h(t)(e^5 \otimes I_2 + e^6\otimes I_3)$$
and one finds \cite[Section 4.3.2]{oliveira2014monopoles} that the $G_2$-instanton equation for such a connection becomes 
\begin{equation}\label{eq:su3invarianteqn}
f^{-4}h^2=1, \qquad f^{-4}\frac{\diff h }{\diff s} +sh=0
\end{equation}
with boundary data $\left.\frac{\diff}{\diff s}\right|_{s=0}f^{-2}h=0, \lim_{s\to \infty}h=0$ and where $f$ is defined as in \eqref{eq:bsf}. The first equation here, which is algebraic, implies $h=\pm f^2$ and one can check that the second equation is then automatically satisfied. Note that the paired equations have a 0-dimensional space of solutions, whilst the linearisation of the system has a 1-dimensional space of solutions. We summarise this result in the following theorem.

\begin{theorem}[Oliveira, {\cite[Theorem 8]{oliveira2014monopoles}}]\thlabel{oliveiraso3inst}
The connection $A$ given by 
$$A=A_\text{\normalfont can}\pm f^2(s)(e^5\otimes I_2 + e^6 \otimes I_3)$$
defines a $G_2$-instanton on $P.$ Moreover $A$ is AC with limiting connection the canonical connection living on the bundle $Q.$
\end{theorem}
The conenction $A$ satisfies $|A-A_\text{can}|_g=O(t^{-3}),$ so in the notation of \eqref{eq:asympconnection} defines an AC $G_2$-instanton with fastest rate of convergence $-2.$

\subsection{Eigenvalues of the Twisted Dirac Operator on \texorpdfstring{$\mathbb{F}_{1,2,3}$}{F1,2,3}}
We now consider the calculations necessary for determining the virtual dimension of the moduli space on which $A$ lives. Again we denote by $Q$ the  bundle defined in \eqref{eq:su3hombundle}.
The basis of $\mf{su}(3)_\C=\mf{sl}(3,\C)$ that we choose is $\{ E_{12},  E_{13},  E_{23}, E_{21},  E_{31},  E_{32},H_{12},H_{23}\},$ where $E_{ij}$ is the elementary matrix with a $1$ in the $(i,j)$\textsuperscript{th} entry and zeros elsewhere, and where
\begin{align*}
    H_{12}=\text{diag}(1,-1,0), \qquad  H_{23}=\text{diag}(0,1,-1).
\end{align*}
The matrices $H_{12}$ and $H_{23}$ form a basis of a Cartan subalgebra given by the Lie algebra of the group of diagonal matrices in $\text{SU}(3),$ we denote this $\mf{t}^2=\langle H_{12},H_{23} \rangle.$ Irreducible representations of $\text{SU}(3)$ are indexed by a pair natural numbers $(a,b)$ with $V_{(1,0)}$ being the standard representation $\C^3.$  
The representation $\mf{m}_\C$ carries the adjoint action of $T^2$ and, as in \cite{charbonneau2016deformations}, one can calculate the weight space decomposition to be 
$$\mf{m}_\C =  W_{(2,-1)}\oplus W_{(-1,2)}\oplus W_{(-1,-1)} \oplus W_{(-2,1)} \oplus W_{(1,-2)}\oplus W_{(1,1)}.$$ \par
The eigenvalues of the Casimir operator one obtains from the nearly K\"ahler metric are calculated in \cite{charbonneau2016deformations}.  We may write
\begin{align*}
    \rho_{V_{(a,b)}}(\text{Cas}_{\mf{su}(3)})&= c^{\mf{su}(3)}_{(a,b)}\text{Id}\\
    \rho_{W_{(c,d)}}(\text{Cas}_{\mf{t}^2})&=c^{\mf{t}^2}_{(c,d)}\text{Id}
\end{align*}
and the eigenvalues are 
\begin{align}
    c^{\mf{su}(3)}_{(a,b)}&=-\frac{4}{3}\left( a^2 + b^2 + ab +3a +3b\right) \label{eq:evalssu3cas} \\
    c^{\mf{t}^2}_{(c,d)}&=-\frac{4}{3}\left( c^2+cd + d^2\right).\label{eq:evalsoftorus}
\end{align}

Let us now examine the relevant Lichnerowicz formula. In the case at hand this takes the form 
$$(D^\frac{1}{3}_{A_\text{can}})^2=-\rho_L(\text{Cas}_{\mf{su}(3)})+\rho_{\Ad \circ \lambda}(\text{Cas}_{\mf{t}^2})+4$$
where $\rho_{\Ad \circ \lambda}$ acts on $\mf{so}(3)_\C.$  We split the space of sections using the Frobenius reciprocity theorem
\begin{equation}
    L^2(\slashed{S}_\C(\mathbb{F}_{1,2,3})\otimes \Ad Q)\cong \bigoplus_{\gamma \in \widehat{\text{SU}(3)}}\text{Hom}(V_\gamma, S\otimes \mf{so}(3)_\C )_{T^2}\otimes V_\gamma
\end{equation}
and define an operator $(D^\frac{1}{2}_{A_\text{can}})^2_\gamma$ by requiring that 
\begin{equation}\label{eq:diracsquareso3case}
    (D^\frac{1}{3}_{A_\text{can}})^2_\gamma \otimes \text{Id} = (D^\frac{1}{3}_{A_\text{can}})^2|_{\text{Hom}(V_\gamma, S\otimes \mf{so}(3))_\C\otimes V_\gamma}.
\end{equation}
The decomposition of $\mf{so}(3)_\C$ under the action $\Ad \circ \lambda $ is 
$$\mf{so}(3)_\C = W_{(-1,-1)}\oplus W_{(0,0)}\oplus W_{(1,1)}.$$ 
\begin{remark}
The reader is warned that we have chosen different generators of the maximal torus of $\text{\normalfont SU}(3)$ to those in \cite{oliveira2014monopoles}, so our labelling conventions are not the same.
\end{remark}
With this in hand we can use \thref{eigenvalsofsquare} and \eqref{eq:evalsoftorus} to  state the eigenvalues of $(D^\frac{1}{3}_{A_\text{can}})^2:$
\begin{lemma}
Let $V_\gamma=V_{(a,b)}$   be an irreducible representation of $\text{\normalfont SU}(3)$ such that $\text{\normalfont dim Hom}(V_\gamma, S\otimes \mf{so}(3)_\C)\neq 0$ and let  $(D^\frac{1}{3}_{A_\text{can}})^2_\gamma $ be the operator defined by \eqref{eq:diracsquareso3case}. The eigenvalues and multiplicities of this operator are 
 \begin{center}
\begin{TAB}(r,0.5cm,1cm)[5pt]{|c|c|}{|c|c|c|} 
     Eigenvalue & Multiplicity  \\
    
     $ -c^{\mf{su}(3)}_{(a,b)}+4$& $\text{\normalfont dim Hom}(V_\gamma, S)_{T^2} $   \\ 
 
     $-c^{\mf{su}(3)}_{(a,b)}$ & $\text{\normalfont dim Hom}(V_\gamma, S\otimes (W_{(-1,-1)}\oplus W_{(1,1)}))_{T^2}.$

    \end{TAB}

\end{center}
\end{lemma}
Again this allows us to bound the eigenvalues of $(D^0_{A_\text{\normalfont can}})_\gamma:$
\begin{corollary}\thlabel{boundsu3}
 Let $V_\gamma$ be an irreducible representation of $\text{\normalfont SU}(3)$. If $V_\gamma$ is not the trivial representation  then $\text{\normalfont Spec}(D^0_{A_\text{can}})_\gamma \cap [0,2]=\emptyset$.
 
\end{corollary}
\begin{proof}
Let $V_\gamma$ be an irreducible representation of $\text{SU}(3)$, then the lower bound on smallest positive eigenvalue of $(D^0_{A_\text{can}})_\gamma$ that is given by \thref{eigenvaluelowerbound} is 
$$L_\gamma=\sqrt{-c^{\mf{su}(3)}_\gamma}-1.$$ 
Using \eqref{eq:evalssu3cas} we see that when $V_\gamma=V_{(1,1)}$ the lower bound that one obtains is $L_\gamma = \sqrt{12}-1>2.$
Any irreducible representation $V_\gamma$ of dimension greater than that of $V_{(1,1)}$ is bounded by a larger lower bound $L_\gamma$ so that the only remaining possibilities are $V_{(0,0)}, V_{(1,0)}$ and $V_{(0,1)}.$
Whilst $\text{Hom}(\C, S\otimes \mf{so}(3)_\C)_{T^2}$ is a non-trivial vector space, one finds that the only $T^2$-equivariant map from $\C^3$ or $(\C^3)^*$ to $S\otimes \mf{so}(3)_\C$ is the zero map and so these representations do no contribute any eigenvalues to the spectrum of $(D^0_{A_\text{can}})_\gamma$.
\end{proof}

Let us thus set $V_\gamma=V_{(0,0)}=\C.$  We should expect to see a 2-dimensional kernel for $(D^0_{A_\text{can}})_\gamma$ since invariant solutions appear from the trivial representation and we know that the linearisation of the invariant $G_2$-instanton equations \eqref{eq:su3invarianteqn}  has a 1-dimensional space of solutions.\par
An application of Schur's lemma shows
$\text{Hom}(\C , S\otimes \mf{so}(3)_\C)_{T^2}$ is four dimensional.
 Let $m^{(i,j)}$ be a fixed non-zero vector lying in the copy of $W_{(i,j)}$ contained in   $ \mf{m}_\C\subset S$. Similarly we let  $w^{(k,l)}$ be a fixed non-zero vector in the copy of $W_{(k,l)}$ contained in $\mf{so}(3)_\C.$ We set 
\begin{align*}
    q^{(0,0)(0,0)}&={\bf{1}} \otimes w^{(0,0)}\\
    \tilde{q}^{(0,0)(0,0)}&=\text{Vol}\otimes w^{(0,0)}\\
    q^{(1,1)(-1,-1)}&=m^{(1,1)}\otimes w^{(-1,-1)}\\
    q^{(-1,-1)(1,1)}&=m^{(-1,-1)}\otimes w^{(1,1)} 
\end{align*}
then $\{q^{(0,0)(0,0)},\tilde{q}^{(0,0)(0,0)},q^{(1,1)(-1,-1)}q^{(-1,-1)(1,1)}\}$ forms a basis of 
$$(S\otimes \mf{so}(3)_\C)_{T^2}\cong \text{Hom}(\C, S\otimes\mf{so}(3)_\C)_{T^2}$$
where $(S\otimes \mf{so}(3)_\C)_{T^2}$ is the space of $T^2$-invariant vectors in $S\otimes \mf{so}(3)_\C.$ Let $\{I_a\}$ be an orthonormal basis for $\mf{m}$ and recall from \eqref{eq:D^thomspaceformula} that 
$$(D^0_{A_\text{can}})_\gamma =\text{cl}(I_a)\rho_\C(I_a)-\frac{3}{4}\text{Re}\Omega$$
on this space. Since the action involved here is the trivial one we see that $(D^0_{A_\text{can}})_\gamma =-\frac{3}{4}\text{Re}\Omega$ and the vectors $q^{(0,0)}$ and $q^{(3,3)}$ are, by \thref{eigenvals}, thus eigenvectors with eigenvalues $-3$ and $3$ respectively. In contrast the vectors $q^{(1,1)(-1,-1)},q^{(-1,-1)(1,1)}$ define equivariant maps that factor through $\Lambda^1\otimes \mf{so}(3)_\C \subset S\otimes \mf{so}(3)_\C$ and it follows from \thref{eigenvals} that $\text{Re}\Omega$ acts trivially on these vectors. These observation yield the next proposition:
\begin{proposition}

Let $V_\gamma=\C$ and let $(D^0_{A_\text{can}})_\gamma$ be the operator on $\text{\normalfont Hom}(V_\gamma, S\otimes \mf{so}(3)_\C)_{T^2}$ given in \eqref{eq:dtonhomspace}. The eigenvalues of this operator are
\begin{center}
\begin{TAB}(r,0.5cm,1cm)[5pt]{|c|c|}{|c|c|c|c|} 
     Eigenvalue $\lambda$ & Multiplicity  \\
   
    $0$ & $2$   \\ 
    
     $3$ & $1$   \\ 
 
      $-3$ & $1$
    \end{TAB}

\end{center}
\end{proposition}

Observe we have shown $2$ is not an eigenvalue of $(D^0_{A_\text{can}})_\gamma$ when $V_\gamma$ is either the trivial representation, the standard representation $\C^3$ or the dual of the standard representation $(\C^3)^*.$ Furthermore 2 cannot occur as an eigenvalue coming from any other representation by \thref{boundsu3}. According to \cite{charbonneau2016deformations} perturbations of a nearly K\"ahler instanton $A_\infty$ on a 6-manifold are given by 1-forms $a\in \Omega^1(\Sigma,\Ad Q) $ such that $D_{A_\infty}(a\cdot s_6)=2a\cdot s_6,$ so this observation yields the following corollary:
\begin{corollary}
 The canonical connection living on the homogeneous bundle $Q$ defined in \eqref{eq:su3hombundle} is a rigid nearly K\"ahler instanton.
\end{corollary}
We conclude this section by giving the virtual dimension of the moduli space $\mathcal{M}(A_\text{can},\mu)$ of which Oliveira's instanton is an element. This follows from  \thref{virtdim} since $0$ is the only eigenvalue of $D^0_{A_\text{can}}$ in $[0,2)$ and has multiplicity 2.
\begin{theorem}\thlabel{virtdimflag}
Let $A$ be Oliveria's $G_2$-instanton with gauge group $\text{\normalfont SO}(3)$ on $\Lambda^2_-(\C \mathbb{P}^2)$ given in \thref{oliveiraso3inst}. The virtual dimension of the moduli space is 
$$\text{\normalfont \footnotesize{virt}}\text{\normalfont \normalsize dim}\mathcal{M}(A_\text{\normalfont can},\mu)=1$$
for all $\mu\in(-2,0).$
\end{theorem}
\begin{remark}\thlabel{invarianceremarks}
Observe that the linearisation of the $\text{SU}(3)$-invariant $G_2$-instanton equation (this is just the equation on the right of \eqref{eq:su3invarianteqn}), is satisfied by $\lambda f$ for any $\lambda \in \R$ and where $f$ is defined as in \eqref{eq:bsf}. Although the linearised invariant equation has a one dimensional space of solutions, the full (non-linear) invariant equation is satisfied by $\pm f $ only. It follows that the  $\text{SU}(3)$-invariant moduli space is obstructed. \\
Note that this one-parameter family of linear deformations should also appear in our calculation, as they are also AC deformations. Whilst we emphasise that investigations into possible obstructions of the AC moduli space are beyond the scope of this paper, we note that the AC moduli space being  \emph{unobstructed} would mean there is a one-parameter family of $G_2$-instantons which must break the  $\text{SU}(3)$ symmetry, for they cannot lie in the invariant moduli space. 

\end{remark}

\subsection{Oliveira's Instantons on \texorpdfstring{$\Lambda^2_-(S^4)$}{L2-(S4)}}
Under the isomorphism $\text{Sp}(2)\cong \text{Spin}(5)$ we see that $\text{Sp}(2)$ acts transitively on $S^4$ with isotropy $\text{Spin}(4)\cong \text{Sp}(1)\times \text{Sp}(1). $ We can lift this  to an isometric action on the total space of $\Lambda^2_-(S^4)$ by asking that $\text{Sp}(2)$ acts on an ASD 2-form via pull back. One can understand the action of the isotropy group $\text{Sp}(1)\times \text{Sp}(1)$ on a fiber of $\Lambda^2_-(S^4)$ by modelling a fibre as the imaginary quaternions \cite{oliveira2014monopoles}, the action is then $(p,q)\cdot x=qx\overline{q}$ for $(p,q)\in \text{Sp}(1)\times \text{Sp}(1)$ and $x\in \text{Im}\mathbb{H}.$ The stabiliser of a non-zero imaginary quaternion under this action is $\text{Sp}(1)\times\text{U}(1)$, so away form the zero section the principal orbits are $$\text{Sp}(2)/\text{Sp}(1)\times \text{U}(1) \cong \C \mathbb{P}^3.$$ 
Moreover since the action is isometric the principal orbits are diffeomorphic to level sets of the norm function $s=|\cdot |^2$ on the fibers. The metric $| \cdot |$ on the fibers is Euclidean and will be denoted $g_{\R^3}.$ The unit sphere bundle of $\Lambda^2_-(S^4)$ is the twistor fibration $\pi \colon \C \mathbb{P}^3 \to S^4$ and carries a nearly K\"ahler structure.\par 
 The geodesic distance to the zero section takes the form 
$$t(s)=\int_0^sf(u)\diff u$$
and this allows us to rewrite the metric \eqref{eq:bsmetric} as 
$$ g= \diff t^2 +s^2(t)f^2(s(t))g_{S^2} + f^{-2}(s(t))\pi^*g_{S^4}$$
where $g_{S^2}$ is the round metric on the unit sphere of a fibre in $\Lambda^2_-(S^4).$  Again we can also describe the $G_2$-structure explicitly in local coordinates. To gain a suitable local expression we first study the homogeneous structure of $\text{Sp}(2)/\text{Sp}(1)\times \text{U}(1).$\par
Recall the Lie algebra $\mf{sp}(2)$ is defined as 
$$ \mf{sp}(2) = \{ X \in \text{Mat}_2(\mathbb{H}) \, ; \, X + X^\dagger =0\}$$
where $X^\dagger $ is the quaternionic conjugate transpose of $X.$ Let $\mf{t}=\text{Lie}(\text{U}(1))$ and let $\bm{i}, \bm{j}$ and $\bm{k}$ denote the unit imaginary quaternions. We embed the subalgebra  $\mf{sp}(1)\oplus \mf{t}$   as 
$$\begin{pmatrix} q & 0 \\
0 & a\bm{i}
\end{pmatrix}$$ 
for $a\in \R.$ Under the reductive decomposition $\mf{sp}(2)=(\mf{sp}(1)\oplus \mf{t})\oplus \mf{m}$ a suitable model for $\mf{m}$ is 
$$\mf{m}=\begin{pmatrix}0 & h \\
-h^\dagger & b\bm{j}+c\bm{k}
\end{pmatrix}$$
where $h$ is a quaternion, $b,c\in \R$ and $h^\dagger$ is the quaternionic conjugate of $h.$ Thus we have a splitting of vector spaces $\mf{m}=\mathbb{H}\oplus \langle \bm{j},\bm{k} \rangle_\R.$ A choice of orthonormal basis for $\mf{m}$ determines a local frame for $T^*\C P^3,$ we label the 1-forms arising from the standard basis of $\mathbb{H}$ as $e^1, \ldots e^4$ and the one-forms corresponding to $\bm{j}$ and $\bm{k}$ respectively are labelled $e^5$ and $e^6.$   Let us define local $\text{Sp}(2)$-invariant 2-forms $\Omega_i$ identically to \eqref{eq:2forms}, then the 3-form $\varphi$ of the Bryant-Salamon structure takes the form 
\begin{equation}
    \varphi=\diff t \wedge (a^2(t) e^{56} + b^2(t)\Omega_1) + a(t)b^2(t)(e^6\wedge \Omega_2 - e^5\wedge \Omega_3)
\end{equation}
where $a(s)=2sf(s^2)$ and $b(s)=\sqrt{2} f^{-1}(s^2).$
This viewpoint allows us to see that the metric and indeed the $G_2$-structure are again AC with $|g-g_C|_g=O(t^{-4}).$ \par

With a description of the metric in hand, we can now explain Oliveira's  \cite{oliveira2014monopoles} examples of $\text{Sp}(2)$-invariant $G_2$-instantons with gauge group $\text{SU}(2)$ on $\Lambda^2_-(S^4)$. 
Homogeneous $\text{SU}(2)$-bundles over $\mathbb{CP}^3=\text{Sp}(2)/\text{Sp}(1)\times \text{U}(1)$ are determined by isotropy homomorphisms $\lambda \colon \text{Sp}(1)\times \text{U}(1) \to \text{SU}(2).$ By \cite[Proposition 5]{oliveira2014monopoles} such a homomorphism is either trivial, $\lambda(g,e^{i\theta })=\text{diag}(e^{il\theta},e^{-il\theta})$ for $l\in \mathbb{Z}$ or $\lambda(g,e^{i\theta})=g$ where we use the standard isomorphism $\text{Sp}(1)\cong \text{SU}(2).$ We shall only consider  the homomorphisms $\lambda_l(g,e^{i\theta})=\text{diag}(e^{i\theta},e^{-i\theta}). $  By \cite[Lemma 1]{oliveira2014monopoles} the only one of the bundles $$Q_l\coloneqq \text{Sp}(2)\times_{\lambda_l}\text{SU}(2)$$ to admit a non-trivial family of invariant connections is $Q_1.$ Let us therefore fix $Q=Q_1.$ In the standard basis $I_1,I_2,I_3$ of $\mf{su}(2)$ an invariant connection on $Q$ is of the form
$$A=A_\text{can}+ c(e^5\otimes I_2 + e^6\otimes I_3)$$
for $c\in \mathbb{R}.$ Let $p\colon (\Lambda^2_-(S^4)\setminus S^4)\to \mathbb{CP}^3$ denote the projection to the unit sphere bundle of $\Lambda^2_-(S^4)$, then the bundle $p^*Q$ admits an extension to the total space of $\Lambda^2_-(S^4)$ which we denote $P$. Thus an invariant connection on $P$ is of the form 
$$A=A_\text{can}+h(e^5\otimes I_2 + e^6\otimes I_3)$$
with $h$  a function of the geodesic distance coordinate $t$ (or alternatively the Euclidean distance $s$). The $G_2$-instanton equation for $A$ then becomes the system \cite[Proposition 6]{oliveira2014monopoles}
$$f^{-4}h^2=1, \qquad f^{-4}\frac{\diff h}{\diff s}+sh=0$$
together with boundary data $\left.\frac{\diff }{\diff s}\right|_{s=0}f^{-2}h=0, \lim_{s\to \infty }h=0.$ Note these equations are identical to \eqref{eq:su3invarianteqn}-- for more on the origins of this duality see \cite[Remark 13]{oliveira2014monopoles}. Again the first equation here, which is algebraic, implies the differential equation is automatically satisfied. Thus:
\begin{theorem}[{\cite[Theorem 5]{oliveira2014monopoles}}]\thlabel{oliveirasp(2)}
The connection 
$$A=A_\text{\normalfont  can}\pm f^2(s)(e^5\otimes I_2 + e^6\otimes I_3)$$
is a $G_2$-instanton on the principal $\text{\normalfont SU}(2)$-bundle $P\to \Lambda^2_-(S^4).$ Moreover $A$ is AC with limiting connection the canonical connection living on the bundle $Q\to \C \mathbb{P}^3.$
\end{theorem}
This connection satisfies $|A-A_\text{can}|_g=O(t^{-3})$ so in the notation of \eqref{eq:asympconnection} defines an AC $G_2$-instanton with fastest rate of convergence $-2.$

\subsection{Eigenvalues of the Twisted Dirac Operator on \texorpdfstring{$\C \mathbb{P}^3$}{CP3}}
Let us again denote by $Q\to \text{Sp}(2)/\text{Sp}(1)\times \text{U}(1)$ the homogeneous SU$(2)$-bundle associated via the homomorphism $\lambda_1(g,e^{i\theta})=\text{diag}(e^{i\theta},e^{-i\theta}).$ Using the framework developed in Section~\ref{sec:lichnerowicz} we aim here to determine which representations of $\text{Sp}(2)$ could lead to eigenvalues of $D^0_{A_\text{can}}$ in the interval $[0,2).$ We look for eigenvalues in this range since the connection $A$ from \thref{oliveirasp(2)} has fastest rate of convergence $-2.$ \par
To achieve this we must first review the representation theory of the groups $\text{Sp}(2)$ and $\text{Sp}(1)\times \text{U}(1).$ Following \cite{charbonneau2016deformations} we choose the Cartan subalgebra of $\mf{sp}(2)_\C$ to be the space of $2\times 2$ quaternionic matrices of the form $\text{diag}(z{\bf{i}},w\bf{i})$ with $z,w\in \C. $  Let us call a weight  positive if it evaluates to a positive real number on the matrix $i\text{diag}(2\bf{i},\bf{i}).$ The matrices $H_1=i\text{diag}(0,\bf{i})$ and $H_2=i\text{diag}(\bf{i},-\bf{i})$ are dual to the fundamental weights $\lambda_1$ and $\lambda_2.$  The irreducible complex representations of $\text{Sp}(2)$ are thus determined by their highest weight $a\lambda_1 + b \lambda_2 $ for $a,b\in \mathbb{N}$ and we write $(V_{(a,b)},\rho_{V_{(a,b)}})$ for such a representation. The first few representations are as follows
\begin{itemize}
    \item $V_{(0,0)}=\C$ the trivial representation
    \item $V_{(0,1)}=\mathbb{H}^2=\C^4$ the standard representation. The group $\text{Sp}(2)$ acts on $\mathbb{H}^2$ by matrix multiplication and hence on $\C^4$ via the isomorphism $\mathbb{H}^2\cong\C^4$ of complex vector spaces.
    \item $V_{(1,0)}$ is a 5 dimensional representation which under the isomorphism $\text{Sp}(2)\cong \text{Spin}(5)$ corresponds to the vector representation of $\text{Spin}(5).$ This means that Spin$(5)$ acts adjointly on $\R^5\otimes \C \subset \text{Cl}(\R^5)\otimes \C.$ As noted in \cite{charbonneau2016deformations} the real 5 dimensional representation 
    $$V_{(1,0)}^\R=\left\{\begin{pmatrix} x & h \\
    h^\dagger & -x\end{pmatrix} \, ; \, x\in \R \text{ and } h\in \mathbb{H} \right\}$$
    with the action being matrix commutation, satisfies $V_{(1,0)}^\R\otimes_\R \C=V_{(1,0)}.$
\end{itemize}

The representation theory of $\text{Sp}(1)\times \text{U}(1)$ is straightforward since  irreducible representations are precisely those of the form $W\otimes W^\prime$ where $W$ is an irreducible representation of $\text{Sp}(1)$ and $W^\prime $ is an irreducible representation of  $\text{U}(1).$ For $a\in \mathbb{N}$ and $b\in \mathbb{Z}$ let us therefore denote by $(W_{(a,b)},\rho_{W_{(a,b)}})$ the unique $(a+1)$-dimensional representation of $\text{Sp}(1)\times \text{U}(1)$ on which $\text{U}(1)$ acts with weight $b$. We realise the Lie algebra as 
\begin{equation}\label{eq:sp1u1subalgebra}
\mf{sp}(1)\oplus \mf{t}=\begin{pmatrix} p & 0 \\
0 & x\bf{i} \end{pmatrix}\end{equation}
for $p\in \mf{sp}(1)=\text{Im}\mathbb{H}$ and $x\in \R $ so that $\mf{sp}(1)\oplus \mf{t}$ naturally forms a subalgebra of $\mf{sp}(2). $ The spinor space $S=\C \oplus \mf{m}_\C\oplus \C$ is the following representation of $\text{Sp}(1)\times \text{U}(1)$ \cite{charbonneau2016deformations}
$$S=\langle {\bf{1}} \rangle_\C \oplus W_{(1,-1)}\oplus W_{(1,1)}\oplus W_{(0,-2)}\oplus W_{(0,2)}\oplus \langle \text{Vol}\rangle_\C.$$
\par
Next we consider the Casimir operators on the irreducible representations $V_{(a,b)}$ and $W_{(c,d)}$ of $\text{Sp}(2)$ and $\text{Sp}(1)\times \text{U}(1)$ respectively. We know that the Casimir operators act as multiples on the identity on these representations so
\begin{align*}
    \rho_{V_{(a,b)}}(\text{Cas}_{\mf{sp}(2)})&=c^{\mf{sp}(2)}_{(a,b)}\text{Id} \\
    \rho_{W_{(c,d)}}(\text{Cas}_{\mf{sp}(1)\oplus \mf{t}})&=c^{\mf{sp}(1)\oplus \mf{t}}_{(c,d)}\text{Id}
\end{align*}
and these constants are found in \cite{charbonneau2016deformations}  to be 
\begin{align}
    c^{\mf{sp}(2)}_{(a,b)}&= -\left( 2a^2 +2ab +b^2 +6a +4b\right) \label{eq:sp2evals}\\
    c^{\mf{sp}(1)\oplus \mf{t}}_{(c,d)}&=-\left( c(c+2)+d^2\right). 
\end{align}\par

Let us begin to apply these facts to study the spectrum of $D^\frac{1}{3}_{A_\text{can}}. $ 
By \thref{eigenvalsofsquare} we need to decompose $(\mf{su}(2)_\C, \Ad \circ \lambda_1)$ as a representation of $\text{Sp}(1)\times \text{U}(1).$ The action is simply the adjoint action of the maximal torus of $\text{SU}(2)$ on $\mf{su}(2)_\C=\mf{sl}(2,\C)$ so the decomposition is 
$$ (\mf{su}(2)_\C, \Ad \circ \lambda_1)=W_{(0,-2)}\oplus W_{(0,0)}\oplus W_{(0,2)}$$
and $\rho_{\Ad \circ \lambda_1}(\text{Cas}_{\mf{sp}(1)\oplus \mf{t}})$ acts as $-4$ on the $W_{(-2,0)}\oplus W_{(0,2)}$ subspace  and acts trivially on $W_{(0,0)}.$ We can once again use \thref{eigenvalsofsquare} to state the eigenvalues of $(D^\frac{1}{3}_{A_\text{can}})^2_\gamma $ are

\begin{lemma}
Let $V_\gamma=V_{(a,b)}$   be an irreducible representation of $\text{\normalfont SU}(3)$ such that $\text{\normalfont dim Hom}(V_\gamma, S\otimes \mf{so}(3)_\C)\neq 0$ and let  $(D^\frac{1}{3}_{A_\text{can}})^2_\gamma $ be the operator defined by \eqref{eq:D^thomspaceformula}. The eigenvalues and multiplicities of this operator are 

\begin{center}
\begin{TAB}(r,0.5cm,1cm)[5pt]{|c|c|}{|c|c|c|} 
     Eigenvalue  & Multiplicity  \\
   
     $-c^{\mf{sp}(2)}_{(a,b)}+  4 $&$ 
     \text{\normalfont dim} \text{\normalfont  Hom}(V_{(a,b)},S )_{\text{Sp}(1)\times \text{U}(1)}  $  \\
     $-c^{\mf{sp}(2)}_{(a,b)}$&$\text{dim} \text{Hom}(V_{(a,b)}, S\otimes (W_{(0,-2)}\oplus W_{(0,2)}))_{\text{Sp}(1)\times \text{U}(1)}$\\
    \end{TAB}
\end{center}
\end{lemma}

Again this will prove a useful consistency check when we come to calculate the matrices of the twisted Dirac operators $(D^t_{A_\text{can}})_\gamma.$ 
\par
Recall that we are looking for eigenvalues of $D^0_{A_\text{can}}$ in the interval $[0,2)$.  We can use \thref{eigenvaluelowerbound} to eliminate most of the representation $V_{(a,b)}$ under consideration. 
\begin{lemma}\thlabel{boundssp2}
Let $V_\gamma$ be an irreducible representation of $\text{\normalfont Sp}(2).$ If $V_\gamma$ is not one of the following representations then $(D^0_{A_\text{can}})_\gamma$ has no eigenvalues in the interval $[0,2]$:
\begin{itemize}
    \item $V_{(0,0)}$ the trivial representation
    \item $V_{(1,0)}$ the vector representation of $\text{\normalfont Spin}(5).$
\end{itemize}
\end{lemma}
\begin{proof}
This is a simple application of \thref{eigenvaluelowerbound}. When $V_\gamma=V_{(0,2)}$ we find that $L_\gamma=\sqrt{12}-1>2$, this leaves the possibility of the above irreducible representations, as well as the standard representation $V_{(0,1)}.$ To rule out the case of the standard representation of Sp$(2)$, note that restricting the action to the subgroup Sp$(1)\times \text{U}(1)$ one finds that $V_{(0,1)}=W_{(1,0)}\oplus [[W_{(0,1)}]]$ but inspection of \eqref{eq:spinorspacesp2} and \eqref{eq:spinorspacesp2prime} then reveals that there are no non-trivial Sp$(1)\times \text{U}(1)$ equivariant maps from $V_{(0,1)}\to S\otimes \mf{su}(2)_\C,$ so this case need not be considered.
\end{proof}

Our task is thus to calculate the eigenvalues of the operator $(D^0_{A_\text{can}})_\gamma$ for the irreducible representations listed in \thref{boundssp2}. Whilst the case $V_\gamma = \C$ is straightforward, the case $V_\gamma = V_{(1,0)}$ requires more work. In this case, the tools developed in Section~\ref{sec:lichnerowicz} are unable to rule out the possibility that the eigenvalues of $(D^0_{A_\text{can}})_\gamma$ lie in the interval $[0,2).$ To remedy this we directly calculate the matrix of the operator in a suitable basis and thus find the eigenvalues in question. Whilst the calculation involved is somewhat lengthy, and therefore deferred to Appendix~\ref{sec:appendixcp3}, the method used can be described as follows: \\
\begin{enumerate}
    \item Calculate orthonormal bases of $\{ I_a \}$ of $\text{m}$ and $\{I_i\}$ of $\text{sp}(2)\times \mf{t}$, this yields an orthonormal basis $\{I_A\}$ of $\text{sp}(2).$
    \item Look for elements  of $V_{(0,1)}^*\otimes S \otimes \mf \mf{su}(2)_\C$ which are invariant under the action of $\text{Sp}(1)\times U(1)$ and hence pick a basis for the space of equivariant maps $V_{(0,1)} \to S\otimes \mf{su}(2)_\C.$ 
    \item Apply the formula \eqref{eq:D^thomspaceformula} with $t=0:$
    $$ (D^0_{A_\text{can}})_\gamma = \text{cl}(I_a)\rho_{V_\gamma^*}(I_a) - \frac{3}{4} \text{Re}\Omega $$
    where $\text{Re}\Omega $ is an element of $\Lambda^3(\mf{m}^*)$ determined by the structure constants via \eqref{eq:holformincoords}. Clifford multiplication can be understood using $\text{Re}\Omega$ and the almost complex structure $J$ via \eqref{eq:nkclifford}. 
\end{enumerate}
This method is similar to that of B\"ar in \cite{bar1992dirac}, where the spectrum of the Dirac operator on $S^3$ and a number of 3-dimensional lens spaces was calculated. The main difference in our homogeneous space $G/H$ has $H$ a connected Lie group, as opposed to a discrete group. Furthermore we are only interested in calculating part of the spectrum- to calculate the entire spectrum would require knowledge of the relevant branching rule and this is in general a difficult problem. \par
The calculations presented in Appendix~\ref{sec:appendixcp3} yield the following result:

\begin{proposition}\thlabel{sp2evals}
Let $V_\gamma$ be an irreducible representation of $\text{Sp}(2)$ and let $(D^0_{A_\text{can}})_\gamma$ denote the twisted Dirac operator on $\text{\normalfont Hom}(V_\gamma,S\otimes \mf{su}(2)_\C)_{\text{\normalfont Sp}(1)\times \text{\normalfont U}(1)}$ as in \eqref{eq:D^thomspaceformula}. The eigenvalues are symmetric about zero and the $\pm \lambda$ eigenspaces are isomorphic.  The non-negative eigenvalues and multiplicities when $V_\gamma =V_{(0,0)}$ are: 
\begin{center}
\begin{TAB}(r,0.5cm,1cm)[5pt]{|c|c|}{|c|c|c|c|} 
     Eigenvalue  & Multiplicity  \\
   
     $0 $&$ 2  $  \\
     $-3$&$1$\\
     $3$ & $1$
    \end{TAB}
\end{center}

When $V_\gamma = V_{(1,0)}$ is the vector representation of $\text{Sp}(2),$ the non-negative eigenvalues and multiplicities are:
\begin{center}
\begin{TAB}(r,0.5cm,1cm)[5pt]{|c|c|}{|c|c|c|c|} 
     Eigenvalue  & Multiplicity  \\
   
    $\sqrt{8}$&$2$\\
    $\frac{1}{2}+\frac{\sqrt{57}}{2}$&$ 1  $  \\
    $-\frac{1}{2}+\frac{\sqrt{57}}{2}$&$1$
    \end{TAB}
\end{center}
\end{proposition}

Observe that $2$ is not an eigenvalue of $(D^0_{A_\text{can}})_\gamma$ when $V_\gamma$ is either the trivial representation or the vector representation and  \thref{boundssp2} ensures it is not an eigenvalue for any other representation. Therefore $2$ is not in the spectrum of $D^0_{A_\text{can}}.$ Since deformations of a nearly K\"ahler instanton correspond to 1-forms in the $+2$ eigenspace of the operator  $D^0_{A_\text{can}}$  \cite{charbonneau2016deformations} we obtain a corollary: 
\begin{corollary}
 Let $Q\to \C \mathbb{P}^3$ be the homogeneous $\text{\normalfont SU}(2)$-bundle 
 $$Q=\text{\normalfont Sp}(2)\times_{\lambda_1}\text{\normalfont SU}(2)$$
 where $\lambda_1\colon \text{\normalfont Sp}(1)\times \text{\normalfont U}(1)\to \text{\normalfont SU}(2)$ is the homomorphism $\lambda_1(g,e^{i\theta})=\text{\normalfont diag}(e^{i\theta},e^{-i\theta}).$ Then the canonical connection $A_\text{can}$ on $Q$ is a rigid nearly K\"ahler instanton.
\end{corollary}
We conclude this section by stating the virtual dimension of the moduli space we have been studying. This follows from the observation in \thref{virtdim}:
\begin{theorem}
Let $A_\text{\normalfont can}$ be the canonical connection on the homogeneous $\text{\normalfont SU}(2)$-bundle $Q=\text{\normalfont Sp}(2)\times_{\lambda_1}\text{\normalfont SU}(2).$ Let $p\colon \Lambda^2_-(S^4)\to \mathbb\C \mathbb{P}^3$ denote the natural projection and let $P$ be the extension of the bundle $p^*Q$ as described in \cite{oliveira2014monopoles}. The virtual dimension of the moduli space $\mathcal{M}(A_\text{\normalfont can},\mu)$ of  $G_2$-instantons on $P$ asymptotic to $A_\text{can}$ with rate $\mu$  is 
$$\text{ \normalfont \footnotesize virt} \text{\normalfont dim}\mathcal{M}(A_\text{\normalfont can},\mu)=1 $$
for all $\mu \in (-2,0).$
\end{theorem}
Once again the one-parameter deformations known from the invariant theory also define AC deformations, so we expected at least a one-dimensional kernel. The dimension of the cokernel is not considered in this paper.

\section{Deformations of the Standard Instanton}\label{sec:stdinstsection}
We now turn our attention to the  $G_2$-instanton of G\"unaydin-Nicolai, which lives on $\R^7$ and is referred to in this paper as the standard instanton. This example is AC with fastest rate of convergence $-2$, so our task is to once more calculate the eigenvalues of the operator at infinity which lie in the interval $[0,2)$. We follow a similar path to that of the previous section, using eigenvalue bounds to narrow down which representations can contribute such eigenvalues. Again we are left needing to calculate the eigenvalues of a few endomorphims defined by $D^0_{A_\infty}.$ To do so we write the twisted Dirac operator as sum of Casimir operators in \thref{diraccasform} and work with this expression under the Frobenius reciprocity formalism. The calculations required here are lengthy, so are presented in Appendix~\ref{sec:appendixs6}. The result of this work is \thref{mainresult}, which determines the virtual dimension of the moduli space on which the standard instanton.

\subsection{The Standard Instanton}
Throughout this section we work with the AC $G_2$-manifold $M=\R^7$, whose asymptotic link is the homogeneous nearly K\"ahler manifold $\Sigma =S^6.$
There is a well known example of a $G_2$-instanton on $\R^7$ constructed by G\"unaydin and Nicolai \cite{gunaydin1995seven}. This example was recovered by Harland et al \cite{harland2010yang} via a different construction method, namely they considered the instanton equation on the cylinder over the nearly K\"ahler 6-sphere (which is conformally equivalent to $\R^7\setminus \{0\}$). This viewpoint makes it easier to understand the asymptotics of the connection, in particular one easily observes that it is AC to the canonical connection. We give here a very brief overview of the construction of this connection, the important fact we shall need is the rate at which the instanton converges. \par
 Form a bundle $Q$ over $S^6$ via
$$Q=G_2\times_{( \text{SU}(3) , \iota)}G_2$$
where $\iota \colon  \text{SU}(3) \hookrightarrow G_2$ is the inclusion homomorphism. Topologically $Q=G_2\times S^6 \to S^6$ is the trivial $G_2$-bundle over $S^6$ but it helps to keep in mind the homogeneous structure.  Since the bundle $Q$ is associated from the canonical bundle $G_2\to G_2/ \text{SU}(3) $ the canonical connection lives on $Q$. Essentially $Q$ extends the canonical bundle to a $G_2$-bundle and the canonical connection is the connection whose horizontal spaces are defined by the inclusion of left translates of $\mf{m}.$ In particular, the canonical connection on $Q$ has holonomy $ \text{SU}(3) .$\par 
 Other than the canonical connection $A_\text{can}$ there is another $G_2$-invariant  connection which is in fact flat. We denote this $A_\text{flat}=A_\text{can}+a.$
Consider $P=\R^7\times G_2$ so that $P|_{\R^7\setminus \{0\}}=\pi^*Q$, the pullback of $Q$ to the cone. Let $(r,\sigma)\in (0,\infty)\times S^6$ and make the $G_2$-invariant ansatz
$$A(r,\sigma)=A_\text{can}(\sigma)+f(r)a(\sigma)$$
where $f$ is a function on $\R^7$ depending only on the radial coordinate $r.$ It is shown in \cite[Section 5.3]{harland2010yang} that 
\begin{equation}\label{eq:diffeqn1}
f(r)=\frac{1}{Cr^2+1}
\end{equation}
with $C>0$ a constant yields a $G_2$-instanton which extends over the origin in $\R^7$.
\begin{remark}
 Harland et al \cite[Section 5.3]{harland2010yang} work on the cylinder over $S^6$, with coordinates $(\sigma ,t)$ for $\sigma \in S^6 $ and $t \in \R$,  so to change to the conical viewpoint considered here one makes the change of variables $r=e^t$. Furthermore the coefficient of $2$ in \eqref{eq:diffeqn1} differs from the one found in \cite{harland2010yang}, this is a consequence of the normalisations we have chosen. Namely the metric on $\mf{g}_2$ is $-\frac{1}{12}$\textsuperscript{th} of the Killing form and our $ \text{SU}(3) $ structure satisfies  $\diff \omega=3\text{Im}\Omega .$
\end{remark}

  For any such $f$ we define
$$A_\text{std}=A_\text{can}+fa$$
 and call this the standard $G_2$-instanton. It is clear from \eqref{eq:diffeqn1}  that $A_\text{std}$ aymptotes to the canonical connection with all rates greater than $-2.$\par
 In the notation of the previous section we set $A=A_\text{std}$ and therefore $A_\infty=A_\text{can}.$ The fastest rate at which $A_\text{std}$ converges is $-2,$ so we will consider the family of moduli spaces $\mathcal{M}(A_\text{can}, \mu)$ for $\mu \in (-2,0).$  In the case at hand we expect that $\{-2,-1\}\subset W$ for the following reason:  The deformation defined by the dilation (the $\R_+$ action on the end of $M$) is $\iota_{\rho \frac{\partial }{\partial \rho}}F_{A_\text{std}}$ and this is added with rate $-2,$ so defines an element of the kernel of $D^0_{A_\text{can}}.$  The deformations determined by translation in $\R^7$ are $\iota_{\frac{\partial}{\partial x_i}}F_{A_\text{std}}$ and these are added with rate $-1,$ therefore we expect the $+1$ eigenspace of $D^0_{A_\text{can}}$ has dimension at least 7.

\subsection{Eigenvalues the Twisted Dirac Operator on \texorpdfstring{$G_2/ \text{SU}(3) $}{G2/SU(3)}}
As in now familiar we can utilise the results of \ref{sec:lichnerowicz} to obtain the eigenvalues of $(D^\frac{1}{3}_{A_\text{can}})^2$ and to obtain bounds on the eigenvalues of $(D^\frac{1}{3}_{A_\text{can}})_\gamma$ for each irreducible representation $V_\gamma$ of $G_2.$ \par

Let us set $G=G_2,H= \text{SU}(3) $ and $\Sigma=S^6.$ All representations under consideration are complex representations for simplicity. It follows from \eqref{eq:descriptionofm} that the complexified tangent bundle is associated via the representation $(\mf{m}_\C,\Ad)=W_{(1,0)}\oplus W_{(0,1)}$ of $ \text{SU}(3) $ and it is worth noting that $\mf{m}_\C\cong \mf{m}_\C^*$ as representations. The complex spinor bundle is associated via the representation
 \begin{equation}\label{eq:spinorspacesu(3) }
     S=W_{(0,0)}\oplus W_{(1,0)} \oplus W_{(0,1)} \oplus W_{(0,0)}.
 \end{equation}

The relevant eigenvalues of the Casimir operators action on the irreducible representations $V_{(i,j)}$ of $G_2$ and $W_{(k,l)}$ of $\text{SU}(3)$ are calculated in \cite{charbonneau2016deformations} to be
\begin{align*}
    \rho_{(i,j)}(\text{Cas}_{\mf{g}_2})&=c^{\mf{g}_2}_{(i,j)}\text{Id} \\
    \rho_{(k,l)}(\text{Cas}_{\mf{su}(3)})&=c^{\mf{su}(3)}_{(k,l)}\text{Id}
\end{align*}
where
\begin{align}
    c^{\mf{g}_2}_{(i,j)}&=-(i^2+3j^2+3ij+5i+9j) \label{eq:gcaseigenvalue}\\
    c^{\mf{su}(3)}_{(k,l)}&=-(k^2+l^2+kl+3k+3l).    \label{eq:hcaseigenvalue}
\end{align}

Let us turn our attention back to the operator $D^\rho=\rho_S(I_a)\rho_R(I_a).$ For $i\neq j$  we set  $[[W_{(k,l)}]]=W_{(k,l)}\oplus W_{(l,k)}.$  The isomorphism class of the $ \text{SU}(3) $ representation $S$ is
$$S=W_{(0,0)} \oplus[[W_{(1,0)}]] \oplus W_{(0,0)}.$$

\begin{lemma}\thlabel{repeigenvalsofsquare}
Let $V_\gamma$ be an irreducible representation of $G_2,$ then the eigenvalues  and multiplicities of the operator $(D^\frac{1}{3}_{A_\text{can}})_{\gamma}^2$ are 
\begin{center}
\begin{TAB}(r,0.5cm,1cm)[5pt]{|c|c|}{|c|c|c|} 
     Eigenvalue  & Multiplicity  \\
   
     $-c^{\mf{g}_2}_{\gamma}$&$ 
     \text{dim} \text{ Hom}(V_{\gamma},S \otimes (W_{(1,0)}\oplus W_{(0,1)}))_{ \text{SU}(3) }  $  \\
    
     $-c^{\mf{g}_2}_{\gamma}-5$ & $\text{dim}  \text{ Hom}(V_\gamma , S \otimes W_{(1,1)})_{ \text{SU}(3) }$  \\ 
    \end{TAB}
\end{center}
where $c^{\mf{g}_2}_\gamma$ is the eigenvalue of the Casimir operator on the irreducible representation $V_\gamma$ with respect to the inner product $B$ from \eqref{eq:killingform}.
\end{lemma} 
\begin{proof}
Immediate from \thref{eq:dthirdsquare} and  \eqref{eq:hcaseigenvalue}.
\end{proof}

\begin{corollary}\thlabel{g2remainingreps}
If $V_{\gamma}$ is not one of the following irreducible representations of $G_2,$ then the operator $(D^0_{A_\text{can}})_\gamma,$ acting on  $\text{Hom}(V_{\gamma}, E)_{ \text{SU}(3) }$, has no eigenvalues in the range $[0,2)$:
\begin{enumerate}
    \item $V_{(0,0)},$ the trivial representation
    \item $V_{(1,0)},$ the standard representation
    \item $V_{(0,1)},$ the adjoint representation.
\end{enumerate}
\end{corollary}

\begin{proof}
 By \thref{repeigenvalsofsquare} the smallest possible eigenvalue of $(D^\frac{1}{3}_{A_\text{can}})^2_\gamma$ is $-c^{\mf{g}_2}_{\gamma}-5$ and it follows that the smallest possible non-negative eigenvalue of $(D^{\frac{1}{3}}_{A_\text{can}})_\gamma$ is $(-c^{\mf{g}_2}_\gamma-5)^\frac{1}{2}.$
 Applying \thref{eigenvaluelowerbound} to the case when $V_\gamma=V_{(2,0)}$ one finds that the lower one the smallest positive eigenvalue one obtains is $L_\gamma=\sqrt{14-5}-1=2$ so this bound is sufficient for the statement of the lemma. When $V_\gamma$ is an irreducible representation of higher dimension one obtains a greater lower bound and the only irreducible representations of lower dimension are those given in the statement of the lemma. 
\end{proof}

For the three representations in the above list we must compute the sets of eigenvalues explicitly.

\subsection{A Formula for the Twisted Dirac Operator on \texorpdfstring{$G_2/\text{SU}(3)$}{G2/SU(3)}}

We begin by establishing an explicit description of the embedding $\mf{g}_2\hookrightarrow\mf{so}(7)\cong\Lambda^2(\R^7)^*$ in term of the subalgebra $\mf{su}(3).$ Recall $\mf{so}(7)\cong \Lambda^2(\R^7)^*$ splits as into irreducible representations of $G_2$ as $\Lambda^2(\R^7)^*=\Lambda^2_{14}\oplus \Lambda^2_7$ where the subscript denotes the dimension of the irreducible component. We pick an orthonormal basis $e^1,\ldots , e^6,\diff t$ of $(\R^7)^*\cong(\R^6)^*\oplus \R^*$ so that the summand $\R^*$ is identified with $\langle \diff t \rangle.$ Under the splitting $\R^7=\R^6\oplus \R$ the action of $ \text{SU}(3) $ is the obvious one.
The image of the  embedding $\mf{g}_2\hookrightarrow \Lambda^2(\R^7)^*$ is the space of 2-forms $\alpha$ satisfying $\alpha \wedge \psi_0=0.$  We decompose the 2-forms on $\R^7$ as $ \text{SU}(3) $ modules:
\begin{align*}
    \Lambda^2(\R^7)^*&\cong \Lambda^2(\R^6)^*\oplus\left( (\R^6)^*\wedge \diff t\right) \\
    &\cong \Lambda^2_8\oplus \Lambda^2_6 \oplus\langle \omega \rangle_\R \oplus \left((\R^6)^*\wedge \diff t\right).
\end{align*}
Let us try to understand Clifford multiplication in this formalism. For this we first seek to extend the spinor space $S$ to a representation of $G_2$, this will allow elements of $\mf{m}\subset \mf{g}_2$ (which we may view as tangent vectors at a point) to act on the spinor bundle and we can compare this action with that of Clifford multiplication. Extending $S$ to a representation of $G_2$ means seeking an action that yields the splitting \eqref{eq:spinorspacesu(3) } when the action is restricted to  $ \text{SU}(3) \subset G_2.$ Note that the standard representation $V_{(1,0)}=\C^7$ of $G_2$ becomes the representation $V_{(1,0)}=W_{(0,0)}\oplus W_{(1,0)}\oplus W_{(0,1)}$ when one restricts to the subgroup $ \text{SU}(3) \subset G_2$ (this is known as the branching rule for $V_{(1,0)}$). Thus $V_{(1,0)}\oplus V_{(0,0)}$ branches to the correct representation of $ \text{SU}(3) .$ To view $S$ as the representation $V_{(1,0)}\oplus V_{(0,0)}$ of $G_2$ we use the  isomorphism 
\begin{align*}
    \C\oplus \mf{m}^*_\C\oplus \C&=(\C^7)^*\oplus \C \\
    (a,v,b) &\mapsto (a\diff t + v,b).
\end{align*}
We thus define an action $\rho_S$ of $\mf{g}_2$ on $S$ with 
$$\rho_S(\eta)(a \diff t + v , b)=(\eta \lrcorner_7 (a\diff t + v) , 0),$$
where $\eta\in\mf{g}_2\subset \Lambda^2(\R^7)^*$ and $\lrcorner_7$ denotes the contraction from the Euclidean inner product on $\R^7$ linearly extended to $(\C^7)^*$. We extend this action to sections of the spin bundle $\slashed{S}_C=\Lambda_\C^0\oplus\Lambda^1_\C\oplus \Lambda^6_\C$. \par
To compare this action to Clifford multiplication we think a vector field $u$ as an element of $L^2(G_2,\mf{m})_{ \text{SU}(3) },$ where $\mf{m}$ is modelled as 
\begin{equation}\label{eq:descriptionofm}
\mf{m}= \text{\normalfont span}\left\{e^a\wedge \diff t+ \frac{1}{2}e^a\lrcorner \text{\normalfont Re}\Omega\, ; \,  i=1,\ldots,6\right\}.
\end{equation}
If we set $F(u)= u \wedge \diff t + \frac{1}{2}u \lrcorner \text{Re}\Omega $ then $F(u)$ defined a section of $\Lambda^2_{14}.$ We can therefore apply the action of $\rho_S$ to $F(u)$ and one finds that
\begin{equation}\label{eq:naturalrep} \rho_S(F(u)) (f,v,g\text{Vol})= \left( -u \lrcorner v , fu - \frac{1}{2} (u \wedge v)\lrcorner \text{Re}\Omega,0\right).
\end{equation}
 Here $\lrcorner $ denotes the contraction map induced by the round metric, extended complex linearly.
This does not agree with the formula for Clifford multiplication given by \eqref{eq:nkclifford}, so to fix this disparity we consider another representation on the spin bundle. There is  another natural representation of $G_2$ on  given by $\rho_{\tilde{S}}(X)=\text{Vol}^{-1}\cdot \rho_S(X)\cdot \text{Vol} $ and we shall see that the two representations $\rho_S$ and $\rho_{\tilde{S}}$ together recover Clifford multiplication.  To see this we first need a simple lemma:

\begin{corollary}
After associating spinor fields with elements of $L^2(G_2,S)_{ \text{SU}(3) }$ and vector fields with elements of $L^2(G_2,\mf{m})_{ \text{SU}(3) },$ Clifford multiplication of a spinor $s$ by a tangent vector $u$ takes the form 
\begin{equation}
\text{\normalfont cl} (u) s = \left(\rho_S(F(u)) - \rho_{\tilde{S}}(F(u))s\right)s
\end{equation}
where $\rho_S$ is the  representation \eqref{eq:naturalrep} of $\mf{g}_2$ on $S$ and \normalfont $\rho_{\tilde{S}}=\text{Vol}^{-1}\cdot \rho_S\cdot \text{Vol}.$
\end{corollary}
\begin{proof}
Recall that the action of the volume from on the spin bundle is $\text{Vol}(f,v,h\text{Vol})=(-h,Jv,f\text{Vol})$, one thus calculates that 
$$(\rho_S(F(u))-\rho_{\tilde{S}}(F(u)))(f,v,h\text{Vol})=\left( -u \lrcorner v,fu-\frac{1}{2}(u \wedge v)\lrcorner \text{Re}\Omega - hJu-\frac{1}{2}J((u \wedge Jv)\lrcorner \text{Re}\Omega ),-(u\lrcorner Jv)\text{Vol}\right).$$
Using that 
$$J\left( (u \wedge Jv)\lrcorner \text{\normalfont Re}\Omega \right) = (u \wedge v)\lrcorner \text{\normalfont Re}\Omega$$ we see this exactly matches the formula for Clifford multiplication given by \eqref{eq:nkclifford}.
\end{proof}

\begin{remark}
Observe that this formula for Clifford multiplication is unique to the case $\Sigma=G_2/ \text{SU}(3) .$ The author has been unable to find a similar formula for the other homogeneous nearly K\"ahler 6-manifolds.
\end{remark}

Having found formulae for both covariant differentiation and Clifford multiplication we can now consider the Dirac operators relevant to this setting. Recall  the operators $D^t$ are built from the modified connections $\nabla^t$ and that $\nabla^1$ is the canonical connection arising from the reductive homogeneous structure. Let us fix   bases  $\{I_a\}_{a=1}^{6}$ for $\mf{m}$ and $\{I_i\}_{i=7}^{14}$ for $\mf{h}$ which are orthonormal with respect to $-\frac{1}{12}$ of the Killing form of $\mf{g}_2.$

\begin{corollary}\thlabel{diraccasform}
After associating spinor fields with elements of $L^2(G_2,S)_{ \text{SU}(3) }$ and vector fields with elements of $L^2(G_2,\mf{m})_{ \text{SU}(3) },$  
the Dirac operator of the canonical connection $D^1 \colon \Gamma(\slashed{S}_\C(\Sigma)) \to \Gamma(\slashed{S}_\C(\Sigma)) $
is 
\begin{align*}
   D^1=(\rho_S(I_a) - \tilde{\rho}_S(I_a))\rho_R(I_a)s.
\end{align*}
\end{corollary}
Let us therefore define operators 
\begin{align*}
D^{\rho}_{A_\text{can}} &\colon \Gamma(\slashed{S}_\C(\Sigma)) \to \Gamma(\slashed{S}_\C(\Sigma)) & \tilde{D^\rho}_{A_\text{can}}&\colon \Gamma(\slashed{S}_\C(\Sigma)) \to \Gamma(\slashed{S}_\C(\Sigma))\\
D^\rho_{A_\text{can}}&=\rho_S(I_a)\rho_R(I_a) & \tilde{D^\rho}_{A_\text{can}}&=\rho_{\tilde{S}}(I_a)\rho_R(I_a)
\end{align*}
so that $D^1_{A_\text{can}}=D^{\rho}_{A_\text{can}} - \tilde{D^\rho}_{A_\text{can}}.$

 One can show that the actions $\rho_S$ and $\rho_R$ commute  and this observation allows us to rewrite $D^\rho_{A_\text{can}}$ as:
$$ D^\rho_{A_\text{can}}= \frac{1}{2}\left(\rho_{S\otimes R}(\text{Cas}_\mf{m})-\rho_S(\text{Cas}_\mf{m})-\rho_R(\text{Cas}_\mf{m}) \right) $$
where a representation of $G_2$ defines a representation of $ \text{SU}(3) $ by restriction and  $\rho(\text{Cas}_\mf{m})\coloneqq \rho(\text{Cas}_{\mf{g}_2})-\rho(\text{Cas}_{\mf{su}(3)}).$ Similarly the expression for $\tilde{D^\rho}$ is 
$$ \tilde{D^\rho}_{A_\text{can}}= \frac{1}{2}\left(\rho_{\tilde{S}\otimes R}(\text{Cas}_\mf{m})-\rho_{\tilde{S}}(\text{Cas}_\mf{m})-\rho_R(\text{Cas}_\mf{m}) \right) $$
and in fact $\tilde{D^\rho}_{A_\text{can}}=\text{Vol}^{-1}D^\rho_{A_\text{can}} \text{Vol}.$ Combining this with \eqref{eq:diracdifference} yields a representation theoretic formula for the Levi-Civita Dirac operator:

\begin{equation}\label{eq:outoflabelnames}
    D^0_{A_\text{can}}=D^\rho_{A_\text{can}}-(\text{Vol}^{-1} D^\rho_{A_\text{can}} \text{Vol}) -\frac{3}{4}\text{Re}\Omega.
\end{equation}

This discussion transfers easily to the homomorphism space decomposition of that space of twisted spinor fields. If $V_\gamma$ is an irreducible representation of $G_2$ then $\text{Hom}(V_\gamma, S\otimes \mf{g}_2)_{\text{SU}(3)}\otimes V_\gamma$  embeds into the space of sections and we can define operators $D^\rho_\gamma$ such that
$$(D^\rho_{A_\text{can}})_\gamma \otimes \text{Id} = (D^\rho_{A_\text{can}}) |_{\text{Hom}(V_\gamma, S\otimes \mf{g}_2)_{\text{SU}(3)}\otimes V_\gamma}.$$ One finds that 

\begin{equation}\label{eq:diracascasimirs}
(D^\rho_{A_\text{can}})_\gamma= \frac{1}{2}\left(\rho_{S\otimes V_\gamma^*}(\text{Cas}_\mf{m})-\rho_S(\text{Cas}_\mf{m})-\rho_{V_\gamma^*}(\text{Cas}_\mf{m}) \right)
\end{equation}
and the operators $(D^t_{A_\text{can}})_\gamma $ from \eqref{eq:dtonhomspace} are of the from 
$$(D^t_{A_\text{can}})_\gamma = (D^\rho_{A_\text{can}})_\gamma - (\tilde{D}^{\rho}_{A_\text{can}})_\gamma + \frac{3(t-1)}{4}\text{Re}\Omega $$
where $(\tilde{D}^{\rho}_{A_\text{can}})_\gamma = \vol ^{-1} (D^{\rho}_{A_\text{can}})_\gamma \vol .$
\begin{remark}
The formula \eqref{eq:outoflabelnames} holds for any twisting of the spin bundle. When the bundle is not twisted, this formula recovers the usual family $D^t$ of Dirac operators. 
\end{remark}

\subsection{Virtual Dimension of the Moduli Space}

By applying \thref{g2remainingreps}, we see that in order to calculate the virtual dimension of $\mathcal{M}(A_\text{can},\mu)$ for $\mu \in (-2,0)$, we must explicitly calculate the eigenvalues of the operators $(D^0_{A_\text{can}})_\gamma$ for the the cases when $V_\gamma$ is the trivial representation, the standard representation and the adjoint representation of $G_2.$ 
This is done using the formula for the twisted Dirac operator that has just been developed. Further details of the required calculations are given in Appendix~\ref{sec:appendixs6}. The result is the following:
 
\begin{proposition}\thlabel{evalsstdinst}
Let $V_\gamma$ be an irreducible representation of $G_2,$ then the eigenvalues of  $(D^0_{A_\text{can}})_\gamma$ are symmetric about $0$ and the $\pm\lambda$ eigenspaces are isomorphic.\par 
When $V_\gamma = V_{(0,0)}$ is the trivial representation of $G_2$, the unique eigenvalue of $(D^0_{A_\text{can}})_\gamma$ is $0$ and has multiplicity 2. \par
When $V_\gamma = V_{(1,0)}$ is the standard representation of $G_2$, the non-negative eigenvalues and multiplicities of $(D^0_{A_\text{can}})_\gamma$  are:
\begin{center}
\begin{TAB}(r,0.5cm,1cm)[5pt]{|c|c|}{|c|c|c|c|} 
     Eigenvalue $\lambda$ & Multiplicity  \\
   
     $1$&$ 1$  \\
    
     $-\frac{1}{2}+\frac{\sqrt{33}}{2}$ & $2$   \\ 
 
     $\frac{1}{2}+\frac{\sqrt{33}}{2}$ & $2$  \\ 
    \end{TAB}
\end{center}
Finally, when 
Let $V_\gamma=V_{(0,1)}$ is the adjoint representation of $G_2$, the non-negative eigenvalues and multiplicities of $(D^0_{A_\text{can}})_\gamma$  are:
\begin{center}
\begin{TAB}(r,0.5cm,1cm)[5pt]{|c|c|}{|c|c|c|c|c|} 
     Eigenvalue $\lambda$ & Multiplicity  \\
    
     $\frac{1}{2}+\frac{\sqrt{57}}{2}$ & 2    \\ 
 
     $-\frac{1}{2}+\frac{\sqrt{57}}{2}$ & 2  \\ 
     
     $\frac{1}{2}+\frac{\sqrt{37}}{2}$ & 1 \\
     
     $-\frac{1}{2}+\frac{\sqrt{37}}{2}$ & 1

    \end{TAB}

\end{center}
\end{proposition}

With this in hand our next theorem follows immediately:
\begin{theorem}\thlabel{mainresult}
The virtual dimension of the moduli space $\mathcal{M}(A_\text{\normalfont can},\mu)$ of AC $G_2$-instantons on $P,$ decaying to $A_\text{\normalfont can}$ with rate $\mu\in(-2,0)\setminus\{1\}$ is 
$$\text{\footnotesize virt\normalsize dim}\mathcal{M}(A_\text{can}, \mu)=\begin{cases} 1 & \text{ if } \mu \in (-2,-1)\\
8 & \text{ if }  \mu \in (-1,0).\end{cases}$$
\end{theorem}
\begin{remark}
The dimension of this moduli space appears to be related to known deformations of the standard instanton: Dilation (radial scaling) is added with rate $-2$, whilst the seven translation parameters are added with rate $-1.$ 
\end{remark}

Establishing the virtual dimension of the moduli space is an important step towards proving a uniqueness theorem for the standard instanton. If we assume this instanton to be unobstructed then the above result provides a local uniqueness theorem for this instanton. In other words, there are no other genuinely different instantons nearby in the moduli space since the only deformations are those defined by the obvious scaling and translation maps. Proving the unobstructedness of connections in the moduli space is a difficult task since curvature terms complicate the usual method of applying Lichnerowicz type formulae to $L^2$ twisted harmonic spinors. We are however still able to apply the deformation theory to study the class of unobstructed instantons, in the next section we aim to build on \thref{mainresult} to attain a uniqueness result in this setting.

\section{An Application of Deformation Theory}\label{sec:invariance}

In this section we explore some applications of the deformation theory. We show in \thref{invariance} that unobstructed connections in the AC moduli space must be invariant under the action of $G_2.$ By classifying $G_2$-invariant connections on the relevant bundle we prove \thref{globaluniqueness}, which says: Among unobstructed connections, the standard instanton is the unique $G_2$-instanton on $P=G_2\times \R^7\to \R^7$ which is asymptotic to $A_\text{can}.$

\subsection{Invariance of AC Instantons}\label{sec:invariancer7}
Let $A$ be an AC $G_2$ instanton on $P=G_2\times \R^7\to \R^7$, converging to $A_\text{can}.$ Recall we may  study the deformations of $A$ in terms an elliptic complex. The cohomology group
$$H^1_{A,\mu}=\frac{\text{Ker}\left(\psi \wedge \diff_{A}\colon \Omega^1_{\mu-1}(M, \Ad P)\to \Omega^6_{\mu-2}(M, \Ad P)\right)}{\diff_{A}(\Omega^0_\mu(M, \Ad P))}$$
satisfies $H^1_{A,\mu}\cong I(A,\mu)$ so we know from \thref{mainresult} that the dimension of these vector spaces are 
$$\text{dim}H^1_{A,\mu}=\begin{cases} 1 & \text {if } \mu\in(-2,-1) \\
8 & \text{if } \mu \in(-1,0).
\end{cases}$$
Furthermore  we know by \cite[Proposition 9]{charbonneau2016deformations} that these deformations are given by the cohomology classes 
\begin{align*}
    \left[\iota_{\frac{\partial}{\partial x_i}}F_{A}\right]&=\left[\mathcal{L}_{\frac{\partial}{\partial x_i}}A\right]\\
     \left[\iota_{r\frac{\partial}{\partial r}}F_{A}\right]&=\left[\mathcal{L}_{r\frac{\partial}{\partial r}}A\right]
\end{align*}
where $\frac{\partial}{\partial x_i}$, for $i=1,\ldots ,7$, are coordinate vector fields and $r^2=x_ix_i.$ In fact any  Killing  field determines a deformation of $A$ and one can ask which Killing fields actually preserve the connection. Here we think of Killing fields $X$ as elements of Lie$(G_2 \ltimes \R^7)$ and define a map
\begin{align*}
    L&\colon \text{Lie}(G_2\ltimes \R^7)\to H^1_{A,-\frac{1}{2}}\\
    L&(X)=\left[\mathcal{L}_XA\right].
\end{align*}

Before investigating the properties of this map we pause to collect some facts about the Lie group $G_2 \ltimes \R^7,$  the group generated by translations and rotation by a $G_2$ matrix. More precisely an element of $G_2\ltimes \R^7$ consists of a pair $(g,v)$ where $g$ is an element of $G_2$ and $v\in\R^7.$ Denote by $R$ the standard representation of $G_2,$ then the action of $(g,v)$ on a point $p\in\R^7$ is 
$$(g,v)\cdot p=R(g)p+v.$$
Acting with two elements gives the composition formula 
$$ (g^\prime, v^\prime) \cdot (g,v)=(g^\prime g, R(g^\prime)v+v^\prime).$$
Denote by $(G_2)_p$ the elements of $G_2\ltimes \R^7$ that fix a point $p\in \R^7.$ Then $(G_2)_p=\{(g,p-R(g)p)\, ; \, g\in G_2\}$ is a subgroup of $G_2\ltimes \R^7$ isomorphic to $G_2.$ 
Other connected subgroups are of the form $H\ltimes U$ for $H$ a proper subgroup of $G_2$ and $U\subset \R^7$ a vector subspace, or $H_p$ the isotropy subgroup of $H$ fixing $p.$ 
If a subgroup $H\ltimes U\subset G_2\ltimes \R^7$ does not fix a point in $\R^7$ then $U$ is a vector space of positive dimension and it follows that the subgroups of $G_2\ltimes \R^7$ that are isomorphic to $G_2$ are precisely the groups $(G_2)_p$ for some $p\in \R^7. $ \par
The set of connected proper Lie subgroups of $G_2$, up to conjugacy, is \begin{equation}\label{eq:g2subgroups}
\{ \text{SU}(3) ,\text{SO}(4),\text{U}(2),\text{SU}(2)\times \text{U}(1),\text{SU}(2),\text{SO}(3),\text{U}(1)^2,\text{U}(1)\} 
\end{equation}
and we can use this to understand the kernel of the map $L.$
\begin{proposition}
The kernel of the map $L$ is a Lie subalgebra of \normalfont Lie$(G_2\ltimes \R^7)$ \textit{isomorphic to} $\mf{g}_2.$
\end{proposition}
\begin{proof}
First we show that Ker$L$ is a Lie subalgebra of Lie$(G_2\ltimes \R^7)$. To avoid notational clutter we shall write elements of $H^1_{A,-\frac{1}{2}}$ without square brackets indicating that they are equivalence classes. Suppose that $X,Y\in \text{Ker}L, $ then there are $f_X,f_Y\in\Omega^0(M, \Ad P)$ with $L(X)=\diff_Af_X$ and $L(Y)=\diff_Af_Y.$ Observe that 
\begin{align*}
    \mathcal{L}_X(\diff_{A}f_Y)&= \{\iota_X,\diff \}(\diff f_Y + [A,f_Y])\\
    &=[\{\iota_X,\diff \}A,f_Y]+\diff \{\iota_X,\diff \}f_Y+[A,\{\iota_X,\diff \}f_Y]\\
    &=[\mathcal{L}_XA,f_Y] +\diff_{A} (\mathcal{L}_X f_Y).
\end{align*}
Therefore we find 
\begin{align*}
    L([X,Y])&= \mathcal{L}_X\mathcal{L}_YA-\mathcal{L}_Y\mathcal{L}_XA\\
    &= \diff_{A}(\mathcal{L}_Xf_Y-\mathcal{L}_Yf_X) + [\mathcal{L}_XA,f_Y]-[\mathcal{L}_YA,f_X]\\
    &=\diff_{A}((\mathcal{L}_X f_Y-\mathcal{L}_Y f_X+[f_X,f_Y])
\end{align*}
so $L([X,Y])$ is in the trivial cohomology class.
Suppose now that $A$ is unobstructed, then $L$ maps from a 21 dimensional vector space to an 8 dimensional vector space Ker$L$ must have dimension at least 13.  Looking through the list \eqref{eq:g2subgroups} of subgroups of $G_2$ we observe that the only possibilities are 
\begin{enumerate}
    \item $\text{Ker}L=\text{Lie}( \text{SU}(3) \ltimes \R^6)$ where $ \text{SU}(3) $ acts in the obvious way
    \item $\text{Ker}L=\text{Lie}( \text{SU}(3) \ltimes (\R^6\oplus \R))$ where $ \text{SU}(3) $ acts on $\R^6$ in the obvious way and trivially on $\R$
    \item $\text{Ker}L=\text{Lie}(G_2\ltimes \R^n)$ for $1\leq n \leq 7$ where $G_2$ acts trivially
    \item $\text{Ker}L=\text{Lie}(\text{SO}(4)\ltimes \R^7)$ where $\text{SO}(4)$ acts on $\R^7$ either trivially or by restriction of the standard $G_2$ action
    \item $\text{Ker}L=\text{Lie}(G_2)_p$ for some $p\in \R^7.$
\end{enumerate}
If any of the cases $1-4$ were to hold then $A$ would have translational symmetries, so we must rule this out. Suppose then for a contradiction that $A$ has translational symmetries, then $A$ is a (globally defined) 1-form such that $\mathcal{L}_XA=0,$ where $X$ is the vector field generating the translations under which $A$ is invariant, thus $X=c_i\frac{\partial}{\partial x_i}$ where $c_i$ are constants and $\frac{\partial}{\partial x_i}$ are the coordinate vector fields on $\R^7$. Note that $\mathcal{L}_XA=0$ implies $\mathcal{L}_XF=0$ and one can check that if $F=F_{ij}\diff x^i\wedge \diff x^j$ is any 2-form on $\R^n$ then $\mathcal{L}_XF=0$ if and only if $X(F_{ij})=0$ for all $i,j.$ It follows that 
$$X(|F|^2)=X(F_{ij}F_{ij})=2F_{ij}X(F_{ij})=0$$
so $|F|$ is constant in the direction $X.$ Pick $p\in S^6$ such that $\{tp \, ; \, t\in \R\} $ is the line through the origin generated by $X$, then 
$$|F_A|(tp)=c(A)$$
where $c(A)\geq 0$ is a constant depending only on $A.$ 
\par
Since $A$ is asymptotically conical (up to gauge) there is a gauge transformation $g$ such that $g\cdot A=gAg^{-1}-\diff gg^{-1}$ satisfies 
$$|g\cdot A - A_C|=|a|=O(r^{\mu-1})$$
where $A_C=\pi^*A_\text{can}$ and $a$ is defined as the difference of $g\cdot A$ and $A_C.$
Observe that
$$|F_{g\cdot A}-F_{A_C}|=|\diff_A a + a\wedge a|=O(r^{\mu-2}),$$
in other words $r^{2-\mu}|F_{g\cdot A}-F_{A_C}|$ is a bounded function of $r$ for some $R>0$ and where $r\in[R,\infty).$ Recall that $F_{A_C}=\pi^*(F_{A_\text{can}})$ and therefore $r^2|F_{A_C}|(rp)=c(A_C)$ where $c(A_C)=|F_{A_\text{can}}|_{g_\text{round}}>0$ is a constant independent of both $r$ and $p.$
We calculate 
\begin{align*}
r^{2-\mu}|F_{g\cdot A}-F_{A_C}|(rp)&\geq r^{2-\mu} ||F_{g\cdot A}|-|F_{A_C}||(rp)\\
&= r^{-\mu}|r^2|F_A|-r^2|F_{A_C}||)(rp)\\
&=r^{-\mu}|r^2c(A)-c(A_C)|.
\end{align*}
However $-\mu>0$ and $c(A_C)>0$ ensures $r^{-\mu}|r^2c(A)-c(A_C)|$ is an unbounded function of $r,$ which yields our contradiction. 
\end{proof}
This proves the $G_2$-invariance of connections in the moduli space:
\begin{proposition}\thlabel{invariance}
Let $A$ be an unobstructed AC $G_2$-instanton on $P,$ converging to $A_\text{can}.$ Then $A$ is invariant under the action of $(G_2)_p$ for some $p\in \R^7.$
\end{proposition}
\subsection{A Uniqueness Theorem}
The constructions covered in this section come from \cite{harland2010yang}. Here we review their work using the framework of Wang's theorem, as has been done to study the moduli space of invariant monopoles on the Bryant-Salamon manifolds $\Lambda^2_-(S^4)$ and $\Lambda^2_-(\C\mathbb{P}^2)$ in \cite{oliveira2014monopoles}.\par

 We say a connection $A$ is  $G$-invariant if its connection $1$-form  $A\in \Omega^1(Q,\mf{k})$ is  left invariant.  Recall the canonical connection $A_\text{can}$ lives on any bundle associated to $G\to G/H$ and clearly defines an invariant connection. Wang's theorem \cite{wang1958invariant} gives an algebraic description of $G$-invariant connections on homogeneous bundles.
\begin{theorem}[Wang \cite{wang1958invariant}]
Let $Q=G\times_{(H,\lambda)}K$ be a principal homogeneous $K$-bundle. Then $G$-invariant connections  on $P$ are in one-to-one correspondence with morphisms $\Phi$ of $H$ representations 
$$\Phi \colon (\mf{m},\text{\normalfont Ad}) \to (\mf{k}, \text{\normalfont Ad}\circ \lambda).$$
The connection $A_\Phi$ that corresponds to such a morphism $\Phi$ satisfies $(A_\Phi-A_\text{\normalfont can})([1])=\Phi$ where we identify linear maps $\mf{m}\to \mf{k}$ with elements of $ (T^*(G/H)\otimes \text{\normalfont Ad}Q)_{[1]}.$
\end{theorem}
If $Q\to G_2/ \text{SU}(3) $ is a homogeneous bundle  we shall use Wang's theorem to study 
$$\mathcal{M}_\text{inv}(G_2/ \text{SU}(3) , Q)$$ 
the space of $G_2$-invariant nearly K\"ahler instantons modulo invariant gauge transformations.  Principal homogeneous $G_2$ bundles  are determined by homomorphisms $\lambda \colon  \text{SU}(3) \to G_2.$  There are exactly two conjugacy classes of such a homomorphism; the class of the trivial homomorphism and the class of the inclusion homomorphism. Hence there are exactly two equivalence classes of principal homogeneous $G_2$ bundles over $S^6.$\par
In the first case, when $\lambda(h)=1$ for all $h\in  \text{SU}(3) ,$ Wang's theorem says to look for morphisms of $ \text{SU}(3) $ representations
$$\Phi \colon (\mf{m}, \text{Ad})\to \R^{14}$$
where $\R^{14}$ denotes 14 copies on the trivial representation. Since there are no such non-zero maps  the only invariant connection corresponds to $\Phi=0$ and this yields the flat connection. \par 
The other case to be considered is when $\lambda$ is the inclusion map $\iota \colon  \text{SU}(3)  \to G_2$ and we denote the associated bundle 
\begin{equation}\label{eq: qbundle}
    Q=G_2\times_{( \text{SU}(3) ,\iota )}G_2.
\end{equation}
In this case, Wang's theorem instructs us to look for morphism of $ \text{SU}(3) $ representations 
$$ \Phi \colon (\mf{m},\Ad ) \to (\mf{g}_2, \Ad).$$ Working with complexified representations we see that a basis for $\text{Hom}(\mf{m}_\C,(\mf{g}_2)_\C)_{ \text{SU}(3) }$ is given by the set $\{\text{Id},J\},$ the identity map and the complex structure. If we identify such a map $a=x\text{Id}+yJ$, where $x,y\in \R$, with the complex number $z=x+iy$, then the $G_2$ invariant connection with 
\begin{equation}\label{eq:invconnex}
A([1])=A_\text{can}([1])+ a
\end{equation}
is a nearly K\"ahler instanton precisely when $\overline{z}^2-z=0$ \cite{harland2010yang}.  Other than the canonical connection $z=0$ the solutions to this equation are precisely the cube roots of unity. Let us write  $A_0,A_1,A_2$ for  the connections  obtained from the solutions $\exp(\frac{2n\pi i}{3})$ for $n=0,1,2.$ These connections are in fact flat and and since $S^6$ is simply connected these connections must be gauge equivalent.
In fact, a little thought shows that these connections are  equivalent through the invariant gauge transformations.
\par
Now let $\pi \colon C(\Sigma) \to \Sigma$ denote projection from the cone to $\Sigma.$  The action of $G_2$ on $Q\coloneqq G_2\times_{(H,\iota)}G_2$ lifts to an action on $P=\R^7\times G_2 \to \R^7$ since the usual action of $G_2$ on $\R^7$ preserves length. Recall $A$ is said to be invariant if its connection $1$-form $A\in\Omega^1(P,\mf{g}_2)$ is left invariant under this action. It suffices to consider only connections $A$ on $P$ which are in radial gauge, i.e $\diff r\lrcorner A=0,$ since such a gauge may always be chosen.
Then an invariant connection is determined as before in \eqref{eq:invconnex} but now we identify $f_1(r)\text{Id}+f_2(r)J$ with the complex valued function $f(r)=f_1(r)+if_2(r).$ As demonstrated in \cite{harland2010yang} such a connection is a $G_2$ instanton if and only if $f$ satisfies the differential equation
\begin{equation}\label{eq:diffeqn}
    rf^\prime (r)=2\left(\bar{f}^2(r)-f(r)\right).
\end{equation}
Again we must remark that the coefficient of 2 is different from that found in \cite[Section 5.3]{harland2010yang} due to our normalisation of the metric and the constraint $\diff \omega = 3 \text{Im}\Omega.$ In \cite[Section 5.3]{harland2010yang} it is shown that this is in fact the gradient flow equation for a real superpotential ``superpotential'' $W\colon \C \to \R$ where $W(z)=\frac{1}{3}(z^3+\bar{z}^3)-|z|^2.$ Clearly \eqref{eq:diffeqn} is equivalent to 
$$ \frac{r}{2} \frac{\diff }{\diff r}f=\frac{\partial W}{\partial \bar{f}}.$$
If we view $W$ as a function of two real variables then a quick calculation shows its has exactly four critical points which are (viewed as complex numbers) 0 and the three cube roots of unity. The unique local maximum is 0 whilst the other critical points are all saddle points. To see the advantage of interpreting \eqref{eq:diffeqn}  as a gradient flow equation we must first determine the relevant boundary data.\par
Note that this construction only yields a connection $A$ on $\pi^*(Q)=(\R^7\setminus \{0\})\times G_2$. To get a connection on $P$ we require $A$ to extend over the origin.  For this to happen it is necessary and sufficient that the curvature be bounded at $r=0.$ The curvature of the connection satisfies 
$$ |F_A|^2(\sigma,r)= \frac{c_1}{r^2}(|\bar{f}^2-f|^2+|\bar{f}f-1|^2)+ \frac{c_2}{r}|f^\prime|$$
with  constants $c_1,c_2>0.$ Thus $f$ is required to satisfy
\begin{itemize}
    \item $|\bar{f}^2-f|=O(r)$ as $r\to 0$
    \item $|\bar{f}f-1|  =O(r)$ as $r\to 0$
    \item $|f^\prime |=O(r)$ as $r \to 0.$
\end{itemize}
As in the previous section we ask that the connection we obtain decays to the canonical connection, in other words we also impose the boundary condition 
$$\lim_{r\to \infty}f(r)=0.$$
Thus our boundary data requires that $f(r)$ tends to a cube root of unity as $r\to 0$  and that $f(r)\to 0$ as $r\to \infty.$
The space of invariant connections can therefore be identified with the space of solutions to \eqref{eq:diffeqn} satisfying the above boundary conditions. The solution to this system 
$$f_0(r)=\frac{1}{Cr^2+1}$$
is the unique one (up to the scaling parameter $C\in \R_+$) with $f(r)\to 1$ as $r\to 0$ whilst other solutions are obtain by applying the 3-symmetry:
$f_1(r)=\exp(\frac{2\pi i}{3})f_0(r)$ and $ f_2(r)=\exp(\frac{4\pi i}{3})f_0(r).$ Uniqueness follows from observing that the solutions are subject to gradient flow away from a critical point and towards the unique local maximum of $W.$
Note that $f_0$ is precisely the solution we have seen in \eqref{eq:diffeqn1} and that this solution yields the standard instanton
 $$A_\text{std}(r[1])=A_\text{can}([1])+f_0(r)\text{Id}.$$
Let the $G_2$ instantons defined by the functions $f_i$ be denoted $\tilde{A}_i,$ then the invariant gauge transformations that related the connections $A_i$ lift to the cone to relate the connections $\tilde{A}_i.$ Thus we have uniqueness of solutions for the given boundary data together with the fact that the 3 families of solutions are gauge equivalent. Therefore, the invariant moduli space is determined precisely by the paramater $C$:
\begin{proposition}
Let $A$ be a $G_2$-instanton on the homogeneous principal bundle $P$ which decays to the canonical connection of the nearly-K\"ahler $S^6$ at infinity. If $A$ is invariant under the action of $G_2$ on $P$, then $A=A_\text{std}$ is the standard $G_2$-instanton.
\end{proposition}
\begin{remark}
The constant $C$ can be interpreted as the ``size'' of the instanton, in other words how concentrated it is around the origin, and is related to the conformal invariance of the $G_2$-instanton equation \cite{charbonneau2016deformations}. It is worth noting that this parameter also arises from the dilation map, so the invariant moduli space coincides with the moduli space \eqref{eq:modulispacedefinition} for rates in the range $-2<\mu<-1.$
\end{remark}

The importance of this result is that is can be combined with the results of the previous subsection to prove a global uniqueness result for unobstructed instantons. Namely, \thref{invariance} says that any unobstructed instanton on $P$, AC to $A_\text{can}$, must be invariant under the action of $G_2$ on $\R^7$ which fixes some point, so uniqueness follows from the above result.
\begin{theorem}\thlabel{globaluniqueness}
Let $A$ be an AC $G_2$-instanton on $P$, converging to $A_\text{\normalfont can}$. Then either $A$ is obstructed or $A$ is the standard $G_2$-instanton. 
\end{theorem}

\newpage 
\setcounter{secnumdepth}{0}
\begin{center}
\section{Appendices}
\end{center}
\setcounter{secnumdepth}{2}
\setcounter{section}{0}
\renewcommand\thesection{\Alph{section}}

\section{Calculation of Eigenvalues: \texorpdfstring{$\mathbb{CP}^3$}{CP3} Case}\label{sec:appendixcp3}
This appendix presents the calculations required to prove \thref{sp2evals}. The eigenvalues coming from two different irreducible representations are calculated.
\subsection{Eigenvalues from the Trivial Representation of Sp\texorpdfstring{$(2)$}{(2)}}

Thus we need only consider the 3 representations stated in \thref{boundssp2}. For notational convenience let us set $[[W_{(a,b)}]]\coloneqq W_{(a,b)}\oplus W_{(a,-b)}$ and recall the twisted spinor space splits 
$$S\otimes \mf{su}(2)_\C = S\oplus \left( S\otimes [[W_{(0,2)}]]\right). $$
As representations of Sp$(1)\times \text{U}(1)$ these spaces are 
\begin{align}
    S&=2W_{(0,0)}\oplus W_{(1,1)}\oplus W_{(1,-1)}\oplus W_{(0,2)}\oplus W_{(0,-2)} \label{eq:spinorspacesp2}\\
    S\otimes [[W_{(0,2)}]]&=[[W_{(1,3)}]]\oplus [[W_{(1,1)}]]\oplus [[W_{(0,4)}]]\oplus 2\left( [[W_{(0,2)}]]\oplus W_{(0,0)}\right)\label{eq:spinorspacesp2prime}. 
\end{align}
The space of sections admits a splitting \begin{align*}L^2(S\otimes \Ad Q)\cong &\left( \bigoplus_{\gamma \in \widehat{\text{Sp}(2)}}\text{Hom}(V_\gamma, S)_{\text{Sp}(1)\times \text{U}(1)}\otimes V_\gamma\right) \oplus \\ &\left( \bigoplus_{\gamma \in \widehat{\text{Sp}(2)}}\text{Hom}(V_\gamma, S\otimes [[W_{(0,2)}]])_{\text{Sp}(1)\times \text{U}(1)}\otimes V_\gamma\right)
\end{align*}
and this is preserved by the operators $(D^t_{A_\text{can}})_\gamma.$
For each irreducible representation $V_\gamma$ of Sp$(2)$ we will therefore consider the operators defined by the restriction of $(D^t_{A_\text{can}})_\gamma$ to the spaces Hom$(V_\gamma, S)_{\text{Sp}(1)\times \text{U}(1)}$ and Hom$(V_\gamma, S\otimes [[W_{(0,2)}]])_{\text{Sp}(1)\times \text{U}(1)}.$\par
The first case to consider is when $V_\gamma=V_{(0,0)}=\C$ is the trivial representation. The space $\text{Hom}(\C , S)_{\text{Sp}(1)\times \text{U}(1)}$ has dimension $2$ and a basis is given by the maps $q^{(0,0)}\colon \C \to \Lambda^{(0,0)}\subset S$ and $q^{(3,3)}\colon \C \to \Lambda^{(3,3)}\subset S$  defined in the obvious way. On this space $(D^1_{A_\text{can}})_\gamma=\text{cl}(I_a)\rho_{V_\gamma^*}(I_a)\equiv 0$ since the action is trivial. Now 
$$(D^0_{A_\text{can}})_\gamma = (D^1_{A_\text{can}})_\gamma -\frac{3}{4}\text{Re}\Omega$$ 
and so \thref{eigenvals} tells us that $q^{(0,0)}$ and $q^{(3,3)}$ are eigenvectors of $(D^0_{A_\text{can}})_\gamma$ with eigenvalues $-3$ and $3$ respectively. Note that the sections these maps define correspond to the Killing spinor $s_6$ and $\text{Vol}\cdot s_6.$\par
Consider now the space $\text{Hom}(V_\gamma, S\otimes [[W_{(0,2)}]])_{\text{Sp}(1)\times \text{U}(1)}$ when $V_\gamma$ is the trivial representation so that once again $(D^1_{A_\text{can}})_\gamma$ acts trivially. This space has dimension 2 since $S\otimes [[W_{(0,2)}]]$ contains two trivial components coming from the subspaces $W_{(0,2)}\otimes W_{(0,-2)}$ and $W_{(0,-2)}\otimes W_{(0,2)}.$ Since the spaces in question are actually subspaces of $\mf{m}^*_\C \otimes [[W_{(0,2)}]]\subset S\otimes [[W_{(0,2)}]]$ we find from \thref{eigenvals} that $\text{Re}\Omega$ acts trivially on Hom$(\C,S\otimes [[W_{(0,2)}]])_{\text{Sp}(1)\times \text{U}(1)}.$ Therefore $(D^t_{A_\text{can}})_\gamma =0$ on this space for all $t.$
We conclude  that the eigenvalues are as stated in \thref{sp2evals}.

\subsection{Eigenvalues from the Vector Representation of \texorpdfstring{$\text{Spin}(5)$}{Spin(5)}}
Throughout this section the representation of Sp$(2)$ under consideration is $V_\gamma=V_{(1,0)}$ the vector representation of Spin$(5).$ Consider first the restriction of $(D^t_{A_\text{can}})_\gamma$ to the space $\text{Hom}(V_{(1,0)}, S\otimes [[W_{(0,2)}]])_{\text{Sp}(1)\times \text{U}(1)}.$
The branching of $V_{(1,0)}$ as a representation of Sp$(1)\times \text{U}(1)$ is \cite{charbonneau2016deformations}
$V_{(1,0)}=[[W_{(1,1)}]]\oplus W_{(0,0)}$ and one therefore finds that $\text{Hom}(V_{(1,0)},S\otimes [[W_{(0,2)}]])_{\text{Sp}(1)\times \text{U}(1)}$ has dimension four and a basis is given by the Sp$(1)\times\text{U}(1)$ equivariant maps which factor
$$q^{(i,j)}_{(k,l)(m,n)}\colon V_{(1,0)}\to W_{(i,j)}\to W_{(k,l)}\otimes W_{(m,n)}\hookrightarrow S\otimes [[W_{(0,2)}]] $$
where the first map is a projection and the second map is an embedding of $W_{(i,j)}$ in $W_{(k,l)}\otimes W_{(m,n)}\subset S\otimes [[W_{(0,2)}]].$ The basis one obtains is $$\left\{q^{(1,1)}_{(1,-1)(0,2)},q^{(1,-1)}_{(1,1)(0,-2)},q^{(0,0)}_{(0,-2)(0,2)},q^{(0,0)}_{(0,2)(0,-2)}\right\},$$
observe that all of these maps factor through $\Lambda^1\subset S$ so that $\text{Re}\Omega $ acts trivially on this space. Using \thref{repeigenvalsofsquare} and \eqref{eq:sp2evals} we find that $(D^\frac{1}{3}_{A_\text{can}})^2_\gamma$ acts as  multiplication by $8$ on this space. The eigenvalues of $(D^\frac{1}{3}_{A_\text{can}})_\gamma$ are therefore $\pm \sqrt{ 8},$ each with multiplicity 2. Since $\text{Re}\Omega$ acts trivially on this space the spectrum is identical for each $(D^t_{A_\text{can}})_\gamma$ for each $t\in \R.$  
\par
Let us now consider the action of the operators  
$(D^t_{A_\text{can}})_\gamma|_{\text{Hom}(V_{(1,0)}, S)_{\text{Sp}(1)\times \text{U}(1)}}$. The space Hom$(V_{(1,0)},S)_{\text{Sp}(1)\times \text{U}(1)}$ four dimensional. A basis is provided by linearly independent  maps which factor $$q^{(i,j)}\colon V_{(1,0)}\to W_{(i,j)}\hookrightarrow S$$
with the first map being projection and the second being inclusion. This yields the basis \begin{equation}\label{eq:sp2calcbasis}
\left\{ q^{(1,1)},q^{(1,-1)}q^{(0,0)},\tilde{q}^{(0,0)}\right\}
\end{equation}
where $q^{(0,0)}$ maps into $\Lambda^{(0,0)}\subset S$ and $\tilde{q}^{(0,0)}$ maps into $\Lambda^{(3,3)}\subset S.$ The situation here is more complicated than those previously considered since an Sp$(1)\times \text{U}(1)$-equivariant map takes values in the entire spinor space $S$, not just one of the subspaces $(\Lambda^{(0,0}\oplus \Lambda^{(3,3)})$ or $\Lambda^1_\C.$ We know from \thref{repeigenvalsofsquare} that $(D^\frac{1}{3}_{A_\text{can}})_\gamma=12$ on this space and  the eigenvalues of $(D^\frac{1}{3}_{A_\text{can}})_\gamma$ are thus $\pm \sqrt{12}.$ The complication arises because the 3-form $\text{Re}\Omega$ acts non-trivially on this space. To calculate the eigenvalues of $(D^0_{A_\text{can}})_\gamma$ on this space we work with formula \eqref{eq:D^thomspaceformula}. We will seek to calculate the action $\rho_{V_{(1,0)}}(I_a)q^{(i,j)}$ where $I_a$ is an orthonormal basis of $\mf{m}$ and $q^{(i,j)}$ is one of the basis vectors from \eqref{eq:sp2calcbasis}, as well as understanding the action of Clifford multiplication and the almost complex structure in this basis. \par
The first job is to find an orthonormal basis for $\mf{sp}(2)$ with respect to the metric \eqref{eq:killingform}. We choose to view the Lie algebra $\mf{sp}(2)$ as a subspace of $\text{Mat}_4(\C)$ by viewing the quaternionic algebra as an algebra of $2\times 2$ complex matrices in the standard way. From this point of view the Killing form is  \cite{fulton2013representation} 
$$\text{Tr}(\ad_X\circ\ad_Y)=6\text{Tr}XY$$
for $X,Y\in \mf{sp}(2)\subset \text{Mat}_4(\C).$ \par 
One finds that 
$$ \begin{array}{cc}
     H_1=\begin{pmatrix}\bf{i}&0 \\
     0&0 \end{pmatrix}= \begin{pmatrix} 
     0&i & 0 & 0  \\
    i&0 &0&0 \\
    0& 0& 0& 0\\
    0& 0& 0& 0
    \end{pmatrix} & 
    H_2=\begin{pmatrix}\bf{j}&0 \\
     0&0 \end{pmatrix}= \begin{pmatrix}
    0 & -1 & 0 & 0 \\
    1 & 0 & 0 & 0\\
    0 & 0 & 0 & 0\\
    0 & 0 & 0& 0
    \end{pmatrix} \\ 
    H_3 =\begin{pmatrix}\bf{k}&0 \\
     0&0 \end{pmatrix}=\begin{pmatrix}
    i & 0 & 0 & 0\\
    0 & -i & 0 & 0\\
    0 & 0 & 0 & 0
    \end{pmatrix}
    &
    H_4=\begin{pmatrix}0&0 \\
     0&\bf{i} \end{pmatrix}=\begin{pmatrix}
    0 & 0 & 0&0\\
    0&0&0&0\\
    0&0&0&i\\
    0&0&i&0
    \end{pmatrix}
    \\
    M_1=\begin{pmatrix}0&0 \\
     0&\bf{j} \end{pmatrix} =\begin{pmatrix}
    0&0&0&0\\
    0&0&0&0\\
    0&0&0&-1\\
    0&0&1&0 \end{pmatrix}
    &
    M_2=\begin{pmatrix}0&0 \\
     0&\bf{k} \end{pmatrix} =\begin{pmatrix}
    0&0&0&0\\
    0&0&0&0\\
    0&0&i&0\\
    0&0&0&-i \end{pmatrix} \\
    M_3=\frac{1}{\sqrt{2}}\begin{pmatrix}
    0 & 1\\
    -1 & 0 \end{pmatrix}=\frac{1}{\sqrt{2}}\begin{pmatrix}
    0&0&1&0\\
    0&0&0&1\\
    -1&0&0&0\\
    0&-1&0&0\end{pmatrix}
    &
    M_4=\frac{1}{\sqrt{2}}\begin{pmatrix}
    0 & \bf{i}\\
    \bf{i} & 0 \end{pmatrix}= \frac{1}{\sqrt{2}}\begin{pmatrix}
    0&0&0&i\\
    0&0&i&0\\
    0&i&0&0\\
    i&0&0&0\end{pmatrix}
    \\
    M_5=\frac{1}{\sqrt{2}}\begin{pmatrix}
    0 & \bf{j}\\
    \bf{j} & 0 \end{pmatrix}=\frac{1}{\sqrt{2}}\begin{pmatrix}
    0&0&0&-1\\
    0&0&1&0\\
    0&-1&0&0\\
    1&0&0&0 \end{pmatrix} 
    &
    M_6=\frac{1}{\sqrt{2}}\begin{pmatrix}
    0 & \bf{k}\\
    \bf{k} & 0 \end{pmatrix}=\frac{1}{\sqrt{2}}\begin{pmatrix}
    0&0&i&0\\
    0&0&0&-i\\
    i&0&0&0\\
    0&-i&0&0\end{pmatrix}.
\end{array}
$$
is an orthonormal set of generators with the matrices $H_i$ generating the subalgebra $\mf{sp}(1)\oplus \mf{t}$ in \eqref{eq:sp1u1subalgebra}. The almost complex structure $J$ acts as follows:
$$\begin{array}{cc}
    J(M_1)=-M_2 & J(M_2)=M_1  \\
     J(M_3)=M_4& J(M_4)=-M_3 \\
     J(M_5)=-M_6 & J(M_6)=M_5
\end{array}$$
and by calculating the structure constants of the algebra and appealing to the formula \eqref{eq:holformincoords} one finds that 
\begin{equation}\label{eq:holvolsp2}
    \text{Re}\Omega =e^{135}+ e^{146} + e^{236}+e^{254}
\end{equation}
in the local frame $\{e^a\}$ of $T^*\C \mathbb{P}^3$ determined by the basis $\{M_a\}$ of $\mf{m}\cong \mf{m}^*.$ \par 
Note we have isomorphisms
$$\text{Hom}(V_{(1,0)},S)_{\text{Sp}(1)\times \text{U}(1)}\cong \left( S\otimes V_{(1,0)}^*\right)_{\text{Sp}(1)\times \text{U}(1)}\cong \left( S\otimes V_{(1,0)}\right)_{\text{Sp}(1)\times \text{U}(1)}$$
and this last space is the subspace of $S\otimes V_{(1,0)}$  consisting of vectors fixed by  the action of the $\text{Sp}(1)\times \text{U}(1)$ subgroup of Sp$(2).$ For the representation $V_{(1,0)}$ we choose basis vectors 
$$\begin{array}{ccc} 
I\coloneqq \begin{pmatrix}
1&0&0&0\\
0&1&0&0\\
0&0&-1&0\\
0&0&0&-1\end{pmatrix},
&X\coloneqq \begin{pmatrix}
0&0&1&0\\
0&0&0&1\\
1&0&0&0\\
0&1&0&0 \end{pmatrix},
&Y\coloneqq \begin{pmatrix}
0&0&0&i\\
0&0&i&0\\
0&-i&0&0\\
-i&0&0&0\end{pmatrix}\\
Z\coloneqq \begin{pmatrix}
0&0&0&1\\
0&0&-1&0\\
0&-1&0&0\\
1&0&0&0\end{pmatrix}, 
&W\coloneqq \begin{pmatrix}
0&0&i&0\\
0&0&0&-i\\
-i&0&0&0\\
0&i&0&0\end{pmatrix}
\end{array}$$
and the action is matrix commutation.\par
To determine what the maps $q^{(i,j)}$ from \eqref{eq:sp2calcbasis} look like in the space $\left( V_{(1,0)}\otimes S\right)_{\text{Sp}(1)\times\text{U}(1)}$ we note first that $W_{(i,j)}\otimes W_{(k,l)}$ contains a copy of the trivial representation of Sp$(1)\times \text{U}(1)$ if and only if $i=k$ and $j=-l.$ We will therefore determine explicit decompositions of the representation $V_{(1,0)}$ and the spinor space $S$ into irreducible representations of Sp$(1)\times \text{U}(1).$\par
By calculating the action $\rho_{V_{(1,0)}}(H_i)$ on the basis matrices of $V_{(1,0)}$ we can determine  explicitly the decomposition $V_{(1,0)}=[[W_{(1,1)}]]\oplus W_{(0,0)}.$ Let us set 
\begin{align*}
    \theta_1\coloneqq X+iY, \qquad \theta_2\coloneqq W-iZ\\
    \overline{\theta}_1\coloneqq X-iY, \qquad \overline{\theta}_2\coloneqq W+iZ
\end{align*}
then one finds that $W_{(0,0)}=\langle I \rangle_\C$ whilst $W_{(1,1)}=\langle \theta_1,\theta_2\rangle_\C$ and $W_{(1,-1)}=\langle \overline{\theta }_1, \overline{\theta}_2\rangle_\C.$ \\
A similar analysis works for the representation $\mf{m}.$ If we define 
\begin{align*}
    \phi_1\coloneqq M_3+iM_4, \qquad \phi_2\coloneqq M_6+iM_5 \\
    \overline{\phi}_1\coloneqq M_3-iM_4, \qquad \overline{\phi}_2\coloneqq M_6-iM_5
\end{align*}
then one finds that $W_{(1,1)}=\langle \phi_1,\phi_2\rangle_\C$ and $W_{(1,-1)}=\langle \overline{\phi}_1, \overline{\phi}_2\rangle_\C.$
With this in hand we can look for invariant vectors in $W_{(i,j)}\otimes W_{(i,-j)}\subset S\otimes V_{(1,0)}.$ One finds that
\begin{align*}
    q^{(0,0)}&= {\bf{1}}\otimes I\\
    \tilde{q}^{(0,0)}&=\text{Vol}\otimes I\\
    q^{(1,1)}&=\phi_1\otimes \overline{\theta_1} + \phi_2\otimes \overline{\theta}_2 \\
    q^{(1,-1)}&=\overline{\phi}_1\otimes \theta_1 + \overline{\phi}_2\otimes \theta_2
\end{align*}
form a basis of $\text{Hom}(V_{(0,1)},S\otimes \mf{su}(2)_\C)_{\text{Sp}(1)\times\text{U}(1)}$.
We can now proceed to calculate the matrix of the Dirac operator in this basis using the formula 
\begin{equation}\label{eq:reptheorydiracsp2}
(D^1_{A_\text{can}})_\gamma = \text{cl}(M_a)\rho_{V_{(1,0)}}(M_a).
\end{equation}
Inspection of \eqref{eq:nkclifford} reveals that $\text{cl}(M_a)(1\otimes v)=M_a\otimes v$ for $v\in V_{(1,0)}$ and we find that $\rho_{V_{(1,0)}}(M_1)I=\rho_{V_{(1,0)}}(M_2)I=0$ whilst 
\begin{align*}
    \rho_{V_{(1,0)}}(M_3)I=-\sqrt{2}X, \qquad     \rho_{V_{(1,0)}}(M_4)I=-\sqrt{2}Y\\
        \rho_{V_{(1,0)}}(M_5)I=\sqrt{2}Z, \qquad     \rho_{V_{(1,0)}}(M_6)I=-\sqrt{2}W.
\end{align*}
Applying these facts to \eqref{eq:reptheorydiracsp2} one finds 
\begin{align*}
    (D^1_{A_\text{can}})_\gamma q^{(0,0)}&= \sqrt{2}\left( -M_3\otimes X -M_4 \otimes Y + M_5\otimes Z -M_6\otimes W\right)\\
    &=-\frac{1}{\sqrt{2}}\left( q^{(1,1)}+q^{(1,-1)}\right). 
\end{align*}

One calculates the action of $(D^1_{A_\text{can}})_\gamma$ on the other basis vectors similarly, using the formula for Clifford multiplication \eqref{eq:nkclifford} and knowledge of the 3-form $\text{Re}\Omega $ \eqref{eq:holvolsp2} and the almost complex structure $J$. This yields the matrix $(D^1_{A_\text{can}})_\gamma $ in the basis $q^{(0,0)},\tilde{q}^{(0,0)},q^{(1,1)},q^{(1,-1)}$. Given that \thref{eigenvals} informs of the action of $\text{Re}\Omega$ in this basis we can appeal to \eqref{eq:diracholformaction} and find the matrix of $(D^t_{A_\text{can}})_\gamma$ to be 
\begin{equation}
(D^t_{A_\text{can}})_\gamma=
    \begin{pmatrix}
    3(t-1)&0&-4\sqrt{2}&-4\sqrt{2}\\
    0&-3(t-1)&4i\sqrt{2}&-4i\sqrt{2}\\
    -\frac{1}{\sqrt{2}}&-\frac{i}{\sqrt{2}}&0&2\\
    -\frac{1}{\sqrt{2}}&\frac{i}{\sqrt{2}}&2&0
    \end{pmatrix}
\end{equation}
We perform a consistency check by noting that $(D^\frac{1}{3}_{A_\text{can}})^2_\gamma=\text{diag}(12,12,12,12)$ as predicted by \eqref{eq:dthirdsquare}. Calculating the eigenvalues of $(D^0_{A_\text{can}})_\gamma$ one finds they are as stated in \thref{sp2evals}.

\section{Calculation of Eigenvalues: \texorpdfstring{$S^6$}{S6} Case}\label{sec:appendixs6}
In this section we present the details required to prove \thref{evalsstdinst}. We perform the calculation of eigenvalues for  the operators  $(D^0_{A_\text{can}})_\gamma$ in the cases when $V_\gamma$ is the trivial representation (Appendix~\ref{sec:g2trivrep}), the standard representation (Appendix~\ref{sec:stdinststdrep}) and the adjoint representation of $G_2$ (Appendix~\ref{sec:stdinstadrep}). \par

Let us first recall the setup. The twisted spinor bundle on which $D^0_{A_\text{can}}$ acts is $(E,\rho_E)=(S\otimes (\mf{g}_2)_\C,\rho_S\otimes \Ad).$ 
As a representation of $G_2$ this twisted spinor space is 
\begin{align*} 
E&=(V_{(1,0)}\oplus V_{(0,0)})\otimes V_{(0,1)}\\
&= V_{(1,1)}\oplus V_{(2,0)}\oplus V_{(0,1)}\oplus V_{(1,0)}.
\end{align*}
The adjoint representation splits $(\mf{g}_2)_\C=W_{(1,1)}\oplus[[W_{(1,0)}]]$ as an $ \text{SU}(3) $-module, so as a representation of $ \text{SU}(3) $ the twisted spinor space is 
\begin{align*}
   E&=([[W_{(1,0)}]]\oplus 2W_{(0,0)})\otimes (W_{(1,1)}\oplus [[W_{(1,0)}]])\\
    &=[[W_{(2,1)}]]\oplus 2[[W_{(2,0)}]] \oplus 4W_{(1,1)}\oplus 4[[W_{(1,0)}]]\oplus 2W_{(0,0)}.
\end{align*}

Since $S=V_{(1,0)}\oplus V_{(0,0)}$ as a representation of $G_2$ it will prove convenient to decompose 
$$ \text{Hom}(S\otimes V_\gamma,(\mf{g}_2)_\C)_{ \text{SU}(3) }\cong \text{Hom}(V_{(1,0)}\otimes V_\gamma, (\mf{g}_2)_\C)_{ \text{SU}(3) }\oplus \text{Hom}(V_{(0,0)}\otimes V_\gamma, (\mf{g}_2)_\C )_{ \text{SU}(3) }$$
as $G_2$-modules.
We have the following  models for irreducible representations of $G_2:$ 
\begin{itemize}
\item$ V_{(1,0)}=\C^7$
\item $V_{(0,1)}=(\mf{g}_2)_\C=\{ \alpha \in \C \otimes \Lambda^2(\R^7)^* \, ; \, \alpha \wedge \psi_0=0\}$
\item $V_{(2,0)}=\{ \eta \in \C \otimes \Lambda^3 (\R^7)^* \, ; \, \eta \wedge \varphi_0=\eta \wedge \psi_0=0\}.$
\end{itemize}
For irreducible representations of $ \text{SU}(3) $ we model 
\begin{itemize}
\item $W_{(1,0)}=\Lambda^{(1,0)}(\R^6)^*$ 
\item $W_{(0,1)}=\Lambda^{(0,1)}(\R^6)^*$
\item $W_{(1,1)}=\mf{su}(3)=\{\alpha \in \Lambda^2(\C^6)^* \, ; \, *(\alpha \wedge \omega)=0\}$
\end{itemize}
where $\omega$ is the standard K\"ahler form on $\R^6.$ Note that $W_{(1,1)}$ embeds into $V_{(0,1)}$ by inclusion and the embedding of $W_{(1,0)}$ into $V_{(0,1)}$ is given by extending the embedding $F$ of $(\R^6)^*$ into $\mf{g}_2$ given by
$$F(v)=v\wedge dt +\frac{1}{2}v\, \lrcorner \, \text{Re}\Omega$$
to a complex linear map and restricting to $W_{(1,0)}\subset (\R^6)^*\otimes \C.$

\subsection{Eigenvalues from the Trivial Representation}\label{sec:g2trivrep}
Let $V_\gamma=V_{(0,0)}$ be the trivial representation of $G_2.$ Then Schur's lemma tells us that Hom$(\C, S\otimes (\mf{g_2})_\C)_{ \text{SU}(3) }\cong\text{Hom}(S,(\mf{g_2})_\C)_{ \text{SU}(3) }$ is two-dimensional. A basis for this space is given by the maps that factor through projections
\begin{align*}
    q^{(1,0)}_{(1,0)}&\colon S \to W_{(1,0)} \to   (\mf{g}_2)_\C \\
    q^{(0,1)}_{(0,1)}&\colon S \to W_{(0,1)} \to  (\mf{g}_2)_\C.
\end{align*}

When $V_\gamma$ is the trivial representation the operators $\left( \rho_{S\otimes V_\gamma^*}(\text{Cas}_\mf{m}) -\rho_S(\text{Cas}_\mf{m}) \right)$ and  $\rho_{V_\gamma^*}(\text{Cas}_\mf{m}) $
both vanish. From  \eqref{eq:diracascasimirs} we know that
$(D^{\rho}_{A_\text{can}})_\gamma$ is built from precisely these operators and hence vanishes. Note also that \thref{eigenvals} ensures that $\text{Re}\Omega$ acts as $0$ on this space since  $q^{(1,0)}_{(1,0)}$ and $q^{(0,1)}_{(0,1)}$ factor through $\Lambda^1\subset S$. Overall then 
$(D^0_{A_\text{can}})_\gamma$  vanishes identically on this space.

\subsection{Calculation of Eigenvalues from the Standard Representation}\label{sec:stdinststdrep}

We now consider the operator $(D^0_{A_\text{can}})_\gamma$ in the case when $V_\gamma=V_{(1,0)}$ be the standard representation of $G_2.$ This acts on the space Hom$(S\otimes V_{(1,0)}, (\mf{g}_2)_\C)_{ \text{SU}(3) }$ which is ten dimensional. It is convenient to split this space as follows:
$$\text{Hom}(S\otimes V_{(1,0)}, (\mf{g}_2)_\C)_{ \text{SU}(3) }\cong \text{Hom}(V_{(1,0)}\otimes V_{(1,0)},(\mf{g}_2)_\C)_{ \text{SU}(3) }\oplus \text{Hom}(\C \otimes V_{(1,0)}, (\mf{g}_2)_\C)_{ \text{SU}(3) }.$$
 Recall the operator
 $(D^\rho_{A_\text{can}})_\gamma$ from \eqref{eq:diracascasimirs} is built from Casimir operators. The matrix of this operator is block diagonal with respect to this splitting, and acts trivially  on $\text{Hom}(\C\otimes V_{(1,0)}, (\mf{g}_2)_\C )_{ \text{SU}(3) }$ by the previous calculation. \par
Our task is thus to calculate $(D^\rho_{A_\text{can}})_\gamma$ on $ \text{Hom}(V_{(1,0)}^{(1)}\otimes V^{(2)}_{(1,0)}, (\mf{g}_2)_\C)_{ \text{SU}(3) }$ where $V_{(1,0)}^{(i)}$ are distinct copies of $V_{(1,0)}.$ Notice  $\rho_{V_{(1,0)}^{(1)}\otimes V_{(1,0)}^{(2)}}(\text{Cas}_{\mf{su}(3)})=\rho_{(\mf{g}_2)_\C}(\text{Cas}_{\mf{su}(3)})$ on this space, so we can write the operator $(D^{\rho}_{A_\text{can}})_\gamma$ as \begin{equation}\label{eq:diracterms}
 (D^{\rho}_{A_\text{can}})_\gamma=\frac{1}{2} \left( \rho_{V^{(1)}_{(1,0)}\otimes V^{(2)}_{(1,0)}}(\text{Cas}_{\mf{g}_2})-\rho_{(\mf{g}_2)_\C}(\text{Cas}_{\mf{su}(3)})- \rho_{V^{(1)}_{(1,0)}}(\text{Cas}_\mf{m})-\rho_{V^{(2)}_{(1,0)}}(\text{Cas}_\mf{m}) \right).
\end{equation}
We shall pick two bases of $\text{Hom}(V_{(1,0)},S\otimes (\mf{g}_2)_\C)_{\text{SU}(3)},$ one which diagonalises the first term in \eqref{eq:diracterms} and one which diagonalises the other terms. We will then relate the two bases to find the matrix of $(D^\rho_{A_\text{can}})_\gamma$  in one of the bases. Finally, by calculating the action of $\vol$ and $\text{Re}\Omega,$ we will find the matrix of $(D^0_{A_\text{can}})_\gamma.$\par
 
To first basis of $\text{Hom}(V^{(1)}_{(1,0)}\otimes V^{(2)}_{(1,0)}, (\mf{g}_2)_\C)_{ \text{SU}(3) }$ is obtained by finding non-zero $ \text{SU}(3) $-equivariant maps
$$q^{(i,j)(k,l)}_{(m,n)}\colon V^{(1)}_{(1,0)}\otimes V^{(2)}_{(1,0)} \to W_{(i,j)}\otimes W_{(k,l)} \to W_{(m,n)}\to   (\mf{g_2})_\C. $$ Observe that this basis diagonalises each term on the right hand side of \eqref{eq:diracterms}, except for the first term.
These maps are constructed from the following projection maps:

\begin{itemize}
    \item $V_{(1,0)} \to W_{(0,0)}, (u+adt) \mapsto adt$
    \item $V_{(1,0)} \to W_{(1,0)}, (u+adt) \mapsto \frac{1}{2}(1+iJ)u$
    \item $ V_{(1,0)} \to W_{(0,1)}, (u+adt) \mapsto \frac{1}{2}(1-iJ)u$
    \item $\Lambda^{1,1}(\R^6)\to W_{(1,1)}, \alpha \mapsto \alpha -\frac{1}{3} \langle \alpha, \omega \rangle \omega.$
\end{itemize}
The basis  we get is
\begin{center}
\begin{tabular}{|p{2.5cm}|p{14cm}|}
\hline 
Map & Factorisation and Formula \\
\hline 
$q_1=q^{(0,0)(1,0)}_{(1,0)}$ &$ V^{(1)}_{(1,0)}\otimes V^{(2)}_{(1,0)} \to W_{(0,0)} \otimes W_{(1,0)} \to W_{(1,0)} \to(\mf{g}_2)_\C$ \\
&$(u+adt) \otimes (v+bdt) \mapsto adt \otimes \frac{1}{2}(1+iJ)v \mapsto \frac{1}{2}a(1+iJ)v \mapsto F(\frac{1}{2}a(1+iJ)v)$\\
\hline
$q_2=q^{(0,0)(0,1)}_{(0,1)}$ &$ V^{(1)}_{(1,0)}\otimes V^{(2)}_{(1,0)} \to W_{(0,0)} \otimes W_{(0,1)} \to W_{(0,1)} \to (\mf{g}_2)_\C$ \\
&$(u+adt) \otimes (v+bdt) \mapsto adt \otimes \frac{1}{2}(1-iJ)v \mapsto \frac{1}{2}a(1-iJ)v \mapsto F(\frac{1}{2}a(1-iJ)v)$\\
\hline
$q_3=q^{(1,0)(0,0)}_{(1,0)}$& $ V^{(1)}_{(1,0)}\otimes V^{(2)}_{(1,0)} \to W_{(1,0)}\otimes W_{(0,0)} \to W_{(1,0)} \to (\mf{g}_2)_\C$ \\
&$(u+adt) \otimes (v+bdt) \mapsto \frac{1}{2}(1+iJ)u\otimes bdt \mapsto \frac{1}{2}b(1+iJ)u \mapsto F(\frac{1}{2}b(1+iJ)u)$\\
\hline 
$q_4=q^{(0,1)(0,0)}_{(0,1)}$& $V^{(1)}_{(1,0)}\otimes V^{(2)}_{(1,0)} \to W_{(0,1)}\otimes W_{(0,0)} \to W_{(0,1)} \to (\mf{g}_2)_\C$ \\
&$(u+adt) \otimes (v+bdt) \mapsto \frac{1}{2}(1-iJ)u\otimes bdt \mapsto \frac{1}{2}b(1-iJ)u \mapsto F(\frac{1}{2}b(1-iJ)u)$ \\
\hline 
$q_5=q^{(1,0)(1,0)}_{(0,1)}$& $ V^{(1)}_{(1,0)}\otimes V^{(2)}_{(1,0)} \to W_{(1,0)}\otimes W_{(1,0)} \to W_{(0,1)} \to (\mf{g}_2)_\C$\\
& $(u+adt) \otimes (v+bdt)\mapsto \frac{1}{2}(1+iJ)u\otimes \frac{1}{2}(1+iJ)v \mapsto \frac{1}{4} [((1+iJ)u)\wedge ((1+iJ)v)]\lrcorner \overline{\Omega} \mapsto F(\frac{1}{4} [((1+iJ)u)\wedge ((1+iJ)v)]\lrcorner \overline{\Omega})$\\
\hline
$q_6=q^{(0,1)(0,1)}_{(1,0)}$ & $ V^{(1)}_{(1,0)}\otimes V^{(2)}_{(1,0)} \to W_{(0,1)}\otimes W_{(0,1)} \to W_{(1,0)} \to (\mf{g}_2)_\C $\\
& $(u+adt) \otimes (v+bdt)\mapsto \frac{1}{2}(1-iJ)u\otimes \frac{1}{2}(1-iJ)v \mapsto \frac{1}{4} [((1-iJ)u)\wedge ((1-iJ)v)]\lrcorner \Omega \mapsto F(\frac{1}{4} [((1-iJ)u)\wedge ((1-iJ)v)]\lrcorner \Omega)$\\
\hline
$q_7=q^{(1,0)(0,1)}_{(1,1)}$& $ V^{(1)}_{(1,0)}\otimes V^{(2)}_{(1,0)} \to W_{(1,0)}\otimes W_{(0,1)} \to W_{(1,1)} \to (\mf{g}_2)_\C $\\
& $(u+adt) \otimes (v+bdt)\mapsto\frac{1}{2}(1+iJ)u\otimes \frac{1}{2}(1-iJ)v \to \frac{1}{4}[((1+iJ)u)\wedge((1-iJ)v)-\frac{1}{3}\langle ((1+iJ)u)\wedge((1-iJ)v),\omega \rangle \omega  ] $\\
\hline
$q_8=q^{(0,1)(1,0)}_{(1,1)}$& $V^{(1)}_{(1,0)}\otimes V^{(2)}_{(1,0)} \to W_{(0,1)}\otimes W_{(1,0)} \to W_{(1,1)} \to (\mf{g}_2)_\C $\\
& $(u+adt) \otimes (v+bdt)\mapsto\frac{1}{2}(1-iJ)u\otimes \frac{1}{2}(1+iJ)v \to \frac{1}{4}[((1-iJ)u)\wedge((1+iJ)v)-\frac{1}{3}\langle ((1-iJ)u)\wedge((1+iJ)v),\omega \rangle \omega ]  $\\
\hline
\end{tabular}
\end{center}
Furthermore we can extend this to a basis of Hom$(S\otimes V_{(1,0)}, (\mf{g}_2)_\C)_{ \text{SU}(3) }$  by adding the maps $q_9=\text{Vol}\cdot q^{(0,0)(1,0)}_{(1,0)}$ and $q_{10}=\text{Vol}\cdot q^{(0,0)(0,1)}_{(0,1)}.$

Next we choose a basis that diagonalises $\rho_{S\otimes V_{(1,0)}^*}(\text{Cas}_{\mf{g}_2})$.
We considering projections through the splitting of $V^{(1)}_{(1,0)}\otimes V^{(2)}_{(1,0)}$ into irreducible representations of $G_2$:
$$p^{(i,j)}_{(k,l)} \colon V^{(1)}_{(1,0)}\otimes V^{(2)}_{(1,0)} \to V_{(i,j)}\to W_{(k,l)} \to \mf{g}_2.$$ 
These maps are eigenvectors of $\rho_{S\otimes V_{(1,0)}^*}(\text{Cas}_{\mf{g}_2})$ with eigenvalue $c^{\mf{g}_2}_{(i,j)}.$ 
To relate the two bases it is necessary to understand each of the projection maps involved in the above construction, then by composition we will be able to understand how they are on an element of $V\otimes V_{(1,0)}. $
 
The  maps we choose are as follows:\\
\begin{center}
\begin{tabular}{|p{2.5cm}|p{14cm}|}
\hline
Map & Factorisation and Formula \\
\hline 
$p_1= p^{(1,0)}_{(1,0)}$&$V^{(1)}_{(1,0)}\otimes V^{(2)}_{(1,0)} \to V_{(1,0)} \to W_{(1,0)}\to (\mf{g}_2)_\C$\\
& $(u+adt) \otimes (v+bdt)\mapsto F\left(\frac{1}{2}(1+iJ)\left[ (u\wedge v) \lrcorner \text{Im}\Omega + bJu-aJv\right]\right)$ \\
\hline 
$p_2= p^{(1,0)}_{(0,1)}$&$V^{(1)}_{(1,0)}\otimes V^{(2)}_{(1,0)} \to V_{(1,0)} \to W_{(0,1)}\to (\mf{g}_2)_\C $\\
& $(u+adt) \otimes (v+bdt)\mapsto F\left(\frac{1}{2}(1-iJ)\left[ (u\wedge v) \lrcorner \text{Im}\Omega + bJu-aJv\right]\right)$ \\
\hline 
$p_3=p^{(0,1)}_{(1,0)}$& $ V^{(1)}_{(1,0)}\otimes V^{(2)}_{(1,0)}\to V_{(0,1)} \to W_{(1,0)}\to (\mf{g}_2)_\C $\\
& $(u+adt)\otimes (v+bdt) \mapsto F\left(\frac{1}{2}(1+iJ)\left[\frac{2}{3}(bu-av)-\frac{1}{3}\left[ (u\wedge v) \lrcorner \text{Re}\Omega\right] \right] \right) $ \\
\hline
$p_4=p^{(0,1)}_{(0,1)}$& $V^{(1)}_{(1,0)}\otimes V^{(2)}_{(1,0)} \to V_{(0,1)} \to W_{(0,1)}\to (\mf{g}_2)_\C $\\
& $(u+adt)\otimes (v+bdt) \mapsto F\left(\frac{1}{2}(1-iJ)\left[\frac{2}{3}(bu-av)-\frac{1}{3}\left[ (u\wedge v) \lrcorner \text{Re}\Omega\right] \right] \right) $ \\
\hline 
$p_5=p^{(0,1)}_{(1,1)}$& $V^{(1)}_{(1,0)}\otimes V^{(2)}_{(1,0)} \to V_{(0,1)} \to W_{(1,1)}\to (\mf{g}_2)_\C$ \\
& $(u+adt)\otimes (v+bdt) \mapsto  \frac{1}{3} u\wedge v -\frac{1}{2}*_6(\omega \wedge u \wedge v ) +\frac{1}{6}*_6(\omega \wedge *_6(\omega \wedge u \wedge v))$\\
\hline 
$p_6=p^{(2,0)}_{(1,0)}$& $V^{(1)}_{(1,0)}\otimes V^{(2)}_{(1,0)}\to V_{(2,0)} \to W_{(1,0)}\to (\mf{g}_2)_\C $\\
& $(u+adt)\otimes (v+bdt) \mapsto F\left( \frac{1}{2}(1+iJ)\left[ Ju \wedge b\omega+ a\omega\wedge Jv\right]\llcorner \omega \right) $ \\
\hline
$p_7=p^{(2,0)}_{(0,1)}$& $V^{(1)}_{(1,0)}\otimes V^{(2)}_{(1,0)} \to V_{(2,0)} \to W_{(0,1)}\to (\mf{g}_2)_\C $\\
& $(u+adt)\otimes (v+bdt) \mapsto F\left( \frac{1}{2}(1-iJ)\left[ Ju \wedge b\omega + a\omega  \wedge Jv\right]\llcorner \omega \right)$ \\
\hline
$p_8=p^{(2,0)}_{(1,1)}$& $ V^{(1)}_{(1,0)}\otimes V^{(2)}_{(1,0)} \to V_{(2,0)} \to W_{(1,1)}\to (\mf{g}_2)_\C $\\
& $(u+adt)\otimes (v+bdt) \mapsto *_6((u\lrcorner \text{Im}\Omega) \wedge (v\lrcorner \text{Im}\Omega))-\frac{1}{3}\langle *_6 (u\lrcorner \text{Im}\Omega) \wedge (v\lrcorner \text{Im}\Omega),\omega \rangle \omega $\\
\hline

\end{tabular}
\end{center}
Each map $p^{(i,j)}_{(k,l)}$ is an eigenvector of $\rho_{S\otimes V_{(1,0)}^*}(\text{Cas}_\mf{g_2})$ with eigenvalue $c_{(i,j)}^\mf{g_2}.$

One finds that the bases $p^{(i,j)}_{(k,l)}$ and $q^{(i,j)(k,l)}_{(m,n)}$ are related as follows:
\begin{align*}
p_1&=iq_1-iq_3-\frac{i}{2}q_6 \\
p_2&=-iq_2 + i q_4 + \frac{i}{2}q_5 \\
p_3&=-\frac{2}{3}q_1+\frac{2}{3}q_3-\frac{1}{6}q_6 \\
p_4&= -\frac{2}{3}q_2+\frac{2}{3}q_4-\frac{1}{6}q_5 \\
p_5&=\frac{1}{2}q_7 + \frac{1}{2}q_8 \\
p_6&= 2iq_1+2iq_3 \\
p_7&= -2iq_2 -2iq_4 \\
p_8&= \frac{1}{2}iq_7 -\frac{1}{2}iq_8.
\end{align*}

Recall the maps  $p^{(i,j)}_{(k,l)}$ are eigenvectors of  $\rho_{S\otimes V_{(0,1)}^*}(\text{Cas}_{\mf{g}_2})$ with eigenvalue $c^{\mf{g}_2}_{(i,j)}$ so that in the basis $p_1,\cdots, p_8$ we have
$$\rho_{S\otimes V_{(0,1)}^*}(\text{Cas}_{\mf{g}_2})=\text{diag}(-6,-6,-12,-12,-12,-14,-14,-14).$$
We have seen that the maps $q^{(i,j)(k,l)}_{(m,n)}$ are eigenvectors of $-\rho_{(\mf{g}_2)_\C}(\text{Cas}_{\mf{su}(3)})- \rho_{S}(\text{Cas}_\mf{m})-\rho_{V_\gamma^*}(\text{Cas}_\mf{m})$ with eigenvalue 
$ c^{\mf{su}(3)}_{(i,j)}+c^{\mf{su}(3)}_{(k,l)}-c^{\mf{su}(3)}_{(m,n)}-2c^{\mf{g}_2}_{(1,0)}$. In the basis $q_1,\ldots q_8$ we find 
$$-\rho_{\mf{g}_2}(\text{Cas}_{\mf{su}(3)})- \rho_{S}(\text{Cas}_\mf{m})-\rho_{V_\gamma^*}(\text{Cas}_\mf{m})=\text{diag}(12,12,12,12,8,8,13,13)
$$ 
and \eqref{eq:diracterms} says that $(D^\rho_{A_\text{can}})_\gamma$ is the sum of these two operators.

Furthermore $\text{Re}\Omega$ acts in the basis $q_1, \ldots ,q_{10}$ as the matrix $\text{diag}(4,4,0,0,0,0,0,0,-4,-4)$ and hence we obtain via \eqref{eq:D^thomspaceformula} the matrix of the twisted Levi-Civita Dirac operator in this basis:
\begin{equation}\label{eq:diracmatrixstdrep}
    (D^t_{A_\text{can}})_\gamma=
\begin{pmatrix}
3(t-1)&0&-1&0&0&-4&0&0&0&0\\
0&3(t-1)&0&-1&-4&0&0&0&0&0\\
-1&0&0&0&0&8&0&0&-i&0\\
0&-1&0&0&8&0&0&0&0&i\\
0&-\frac{1}{4}&0&\frac{1}{2}&0&0&0&0&0&-\frac{i}{4}\\
-\frac{1}{4}&0&\frac{1}{2}&0&0&0&0&0&\frac{i}{4}&0\\
0&0&0&0&0&0&0&1&0&0\\
0&0&0&0&0&0&1&0&0&0\\
0&0&i&0&0&-4i&0&0&-3(t-1)&0\\
0&0&0&-i&4i&0&0&0&0&-3(t-1)
\end{pmatrix}.
\end{equation}

A consistency check is obtained by observing that \eqref{eq:dthirdsquare} ensures the basis $q_1,\cdots q_{10}$ diagonalises $(D^{\frac{1}{3}}_{A_\text{can}})^2_\gamma$ and  
$$(D^{\frac{1}{3}}_{A_\text{can}})^2_\gamma q^{(i,j)(k,l)}_{(m,n)} =\left(-c^{\mf{g}_2}_{(1,0)}+c^{\mf{su}(3)}_{(m,n)}+4 \right)q^{(i,j)(k,l)}_{(m,n)}.$$
We find

$$(D^{\frac{1}{3}}_{A_\text{can}})^2_\gamma=\text{diag}(6,6,6,6,6,6,1,1,6,6)$$
as expected. By calculating the eigenvalues of \eqref{eq:diracmatrixstdrep} one finds they are as stated in \thref{evalsstdinst}.

\subsection{Calculation of Eigenvalues from the Adjoint Representation}\label{sec:stdinstadrep}

We now consider the case when $V_\gamma=V_{(0,1)}$ is the adjoint representation. The operator $(D^0_{A_\text{can}})_\gamma$  acts on the space $\text{Hom}(S\otimes V_{(0,1)}, (\mf{g}_2)_\C)_{ \text{SU}(3) }.$ A similar method to the previous case, of the standard representation, is applicable. 

 We have $V_{(1,0)}\otimes V_{(0,1)}=V_{(1,1)}\oplus V_{(2,0)}\oplus V_{(1,0)}$ and the summands split as reps of $ \text{SU}(3) $ as follows:
\begin{align*}
     V_{(1,1)}&=[[W_{(2,1)}]]\oplus [[W_{(2,0)}]]\oplus 2W_{(1,1)}\oplus [[W_{(1,0)}]] \\
V_{(2,0)}&=[[W_{(2,0)}]]\oplus W_{(1,1)}\oplus [[W_{(1,0)}]]\oplus W_{(0,0)} \\
V_{(1,0)}&=[[W_{(1,0)}]]\oplus W_{(0,0)}.
\end{align*}
We therefore see that $\text{Hom}(S\otimes V_{(0,1)}, (\mf{g}_2)_\C)_{ \text{SU}(3) }$ has dimension 12.  
As before we split $\text{Hom}(S\otimes V_{(0,1)}, (\mf{g}_2)_\C)_{ \text{SU}(3) }\cong \text{Hom}(V_{(1,0)}\otimes V_{(0,1)}, (\mf{g}_2)_\C)_{ \text{SU}(3) }\oplus \text{Hom}(\C \otimes V_{(0,1)},\mf{g}_2)_{ \text{SU}(3) }$ so  that the operator $(D^\rho_{A_\text{can}})_\gamma$ is block diagonal and its action on $\text{Hom}(\C \otimes V_{(0,1)},(\mf{g}_2)_\C)_{ \text{SU}(3) }$ is 0.\par
On $\text{Hom}(V_{(1,0)}\otimes V_{(0,1)}, (\mf{g}_2)_\C)_{ \text{SU}(3) }$ the operator $(D^{\rho}_{A_\text{can}})_\gamma$ takes the form
\begin{equation}\label{eq:diractermsadj}
(D^{\rho}_{A_\text{can}})_\gamma=\frac{1}{2}(\rho_{V_{(1,0)}\otimes V_{(0,1)}}(\text{Cas}_{\mf{g}_2})-\rho_{(\mf{g}_2)_\C}(\text{Cas}_{\mf{su}(3)})- \rho_{V_{(1,0)}}(\text{Cas}_\mf{m})-\rho_{V_{(0,1)}}(\text{Cas}_\mf{m})).
\end{equation}
To calculate the action of the operator on $\text{Hom}(V_{(1,0)}\otimes V_{(0,1)},(\mf{g}_2)_\C)_{ \text{SU}(3) }$ we  again  pick two bases for this space which diagonalise the various  Casimir operators from which \eqref{eq:diracterms} tells us $(D^\rho_{A_\text{can}})_\gamma$ is constructed. As before we first pick a basis consisting of maps that factor as follows:

$$   q^{(i,j)(k,l)}_{(m,n)}\colon V\otimes  V_{(0,1)}\to W_{(i,j)}\otimes W_{(k,l)}\to W_{(m,n)}\to (\mf{g}_2)_\C.$$
These maps diagonalise each term on the right hand side of \eqref{eq:diractermsadj} except for the first term, they are as follows:
\begin{center}
\begin{tabular}{|p{2.5cm}|p{14cm}|}
\hline
Map & Factorisation and Formula \\
\hline 
$q_1=q^{(0,0)(1,0)}_{(1,0)}$&$ V_{(1,0)}\otimes V_{(0,1)} \to W_{(0,0)}\otimes W_{(1,0)} \to W_{(1,0)}\to (\mf{g}_2)_\C $\\
& $(u+adt) \otimes (\alpha + v\wedge dt + \frac{1}{2}v\lrcorner \text{Re}\Omega)\mapsto F\left(a\frac{1}{2}(1+iJ)v\right)$ \\
\hline 
$q_2=q^{(0,0)(0,1)}_{(0,1)}$&$ V_{(1,0)}\otimes V_{(0,1)} \to W_{(0,0)}\otimes W_{(0,1)} \to W_{(0,1)}\to (\mf{g}_2)_\C $\\
& $(u+adt) \otimes (\alpha + v\wedge dt + \frac{1}{2}v\lrcorner \text{Re}\Omega)\mapsto F\left(a\frac{1}{2}(1-iJ)v\right)$ \\
\hline 
$q_3=q^{(0,0)(1,1)}_{(1,1)}$&$ V_{(1,0)}\otimes V_{(0,1)}\to W_{(0,0)}\otimes W_{(1,1)} \to W_{(1,1)}\to (\mf{g}_2)_\C $\\
& $(u+adt) \otimes (\alpha + v\wedge dt + \frac{1}{2}v\lrcorner \text{Re}\Omega)\mapsto a\alpha $ \\
\hline 
$q_4=q^{(1,0)(1,1)}_{(1,0)}$&$ V_{(1,0)}\otimes V_{(0,1)}\to W_{(1,0)}\otimes W_{(1,1)} \to W_{(1,0)}\to (\mf{g}_2)_\C $\\
& $(u+adt) \otimes (\alpha + v\wedge dt + \frac{1}{2}v\lrcorner \text{Re}\Omega)\mapsto F\left( (\frac{1}{2}(1+iJ)u)\lrcorner\alpha\right)$ \\
\hline
$q_5=q^{(0,1)(1,1)}_{(0,1)}$&$ V_{(1,0)}\otimes V_{(0,1)}\to W_{(0,1)}\otimes W_{(1,1)} \to W_{(0,1)}\to (\mf{g}_2)_\C $\\
& $(u+adt) \otimes (\alpha + v\wedge dt + \frac{1}{2}v\lrcorner \text{Re}\Omega)\mapsto F\left( (\frac{1}{2}(1-iJ)u)\lrcorner\alpha\right)$ \\
\hline
$q_6=q^{(1,0)(0,1)}_{(1,1)}$&$ V_{(1,0)}\otimes V_{(0,1)}\to W_{(1,0)}\otimes W_{(0,1)} \to W_{(1,1)}\to (\mf{g}_2)_\C $\\
& $(u+adt) \otimes (\alpha + v\wedge dt + \frac{1}{2}v\lrcorner \text{Re}\Omega)\mapsto \frac{1}{4}[((1+iJ)u)\wedge((1-iJ)v)-\frac{1}{3}\langle ((1+iJ)u)\wedge((1-iJ)v),\omega \rangle \omega  ] $ \\
\hline
$q_7=q^{(0,1)(0,1)}_{(1,0)}$&$ V_{(1,0)}\otimes V_{(0,1)}\to W_{(0,1)}\otimes W_{(0,1)} \to W_{(1,0)}\to (\mf{g}_2)_\C $\\
& $(u+adt) \otimes (\alpha + v\wedge dt + \frac{1}{2}v\lrcorner \text{Re}\Omega)\mapsto F\left(  \frac{1}{4} [((1-iJ)u)\wedge ((1-iJ)v)]\lrcorner \Omega \right)$ \\
\hline
$q_8=q^{(1,0)(1,0)}_{(0,1)}$&$ V_{(1,0)}\otimes V_{(0,1)} \to W_{(1,0)}\otimes W_{(1,0)} \to W_{(0,1)}\to (\mf{g}_2)_\C $\\
& $(u+adt) \otimes( \alpha + v\wedge dt + \frac{1}{2}v\lrcorner \text{Re}\Omega)\mapsto F\left(  \frac{1}{4} [((1+iJ)u)\wedge ((1+iJ)v)]\lrcorner \overline{\Omega} \right)$ \\
\hline
$q_9=q^{(0,1)(1,0)}_{(1,1)}$&$V_{(1,0)}\otimes V_{(0,1)}\to W_{(0,1)}\otimes W_{(1,0)} \to W_{(1,1)}\to (\mf{g}_2)_\C $\\
& $(u+adt) \otimes (\alpha + v\wedge dt + \frac{1}{2}v\lrcorner \text{Re}\Omega)\mapsto \frac{1}{4}[((1-iJ)u)\wedge((1+iJ)v)-\frac{1}{3}\langle ((1-iJ)u)\wedge((1+iJ)v),\omega \rangle \omega  ]. $ \\
\hline
\end{tabular}
\end{center}
We extend this set to a basis of $\text{Hom}(S\otimes V_{(0,1)}, (\mf{g}_2)_\C)_{ \text{SU}(3) }$ by adding the three maps 
\begin{align*}
q_{10}&=\text{Vol}\cdot q^{(0,0)(1,0)}_{(1,0)}\\
q_{11}&=\text{Vol}\cdot q^{(0,0)(0,1)}_{(0,1)}\\
q_{12}&=\text{Vol}\cdot q^{(0,0)(1,1)}_{(1,1)}.
\end{align*}

To diagonalise the operator $\rho_{S\otimes V_\gamma^*} (\text{Cas}_{\mf{g}_2})$ we choose maps that factor as follows:
$$ p^{(i,j)}_{(k,l)}: V_{(1,0)} \otimes V_{(0,1)} \to V_{(i,j)}\to W_{(k,l)} \to (\mf{g}_2)_\C.$$
These maps are eigenvectors of $\rho_{S\otimes V_\gamma^*} (\text{Cas}_{\mf{g}_2})$ with eigenvalue $c^{\mf{g}_2}_{(i,j)}.$

Let $w\in V$ and $\beta \in V_{(0,1)},$  applying Schur's lemma where necessary gives the required projection maps to be:
\begin{itemize}
    \item $V_{(1,0)}\otimes V_{(0,1)}\to V_{(2,0)}$, $w \otimes \alpha \mapsto \wedge \beta -\frac{1}{4}(w\lrcorner \beta)\lrcorner \psi$
    \item $V_{(1,0)}\otimes V_{(0,1)} \to V_{(1,0)}$, $w\otimes \beta \mapsto w\lrcorner \beta  $
    \item $V_{(2,0)}\to W_{(0,1)}$, $dt\wedge \kappa + \eta \mapsto \frac{1}{2}(1+iJ)(\kappa \lrcorner \text{Re}\Omega) $
    \item $V_{(2,0)}\to W_{(1,1)}$, $dt \wedge \kappa + \eta \mapsto  \frac{1}{3}\kappa -\frac{1}{2}*_6(\omega \wedge \kappa) +\frac{1}{6}*_6(\omega\wedge *_6(\omega \wedge \kappa))$
    \item $V_{(1,0)} \to W_{(1,0)}$, $u+a dt \mapsto \frac{1}{2}(1+iJ)u$
\end{itemize}

The difficulty in working with this basis is that the space $V_{(1,1)}$ cannot be modelled as a subspace of $\Lambda^*(\R^7)^*.$ Instead we work with maps that factor through $V_{(1,1)}$ by noticing that they must be orthogonal with respect to the natural inner product on the space given by $\langle X, Y\rangle=\text{Trace}(X^*Y)$, since they factor through orthogonal subspaces. We can therefore find expressions for the maps $p^{(1,1)}_{(i,j)}$ by ensuring the basis vectors are mutually orthogonal. The basis vectors can be described as follows:
\begin{center}
\begin{tabular}{|p{2.5cm}|p{14cm}|}
\hline
Map & Factorisation and Formula \\
\hline 
$p_1= p^{(1,0)}_{(1,0)}$&$ V_{(1,0)}\otimes V_{(0,1)} \to V_{(1,0)} \to W_{(1,0)}\to (\mf{g}_2)_\C $\\
& $(u+adt) \otimes( \alpha + v\wedge dt + \frac{1}{2}v\lrcorner \text{Re}\Omega)\mapsto \frac{1}{2}(1+iJ)[u\lrcorner \alpha -\frac{1}{2}(u\wedge v)\lrcorner \text{Re}\Omega -av]$ \\
\hline 

$p_2= p^{(2,0)}_{(1,0)}$&$ V_{(1,0)}\otimes V_{(0,1)}  \to V_{(2,0)} \to W_{(1,0)}\to (\mf{g}_2)_\C $\\
& $(u+adt) \otimes (\alpha + v\wedge dt + \frac{1}{2}v\lrcorner \text{Re}\Omega)\mapsto \frac{1}{2}(1+iJ)[u\wedge v + a\alpha -\frac{1}{4}(-(u\lrcorner\alpha)\lrcorner \text{Re} \Omega + av\lrcorner \text{Re} \Omega)]$ \\
\hline 
$p_3=p^{(1,1)}_{(1,0)}$&$ V_{(1,0)}\otimes V_{(0,1)}  \to V_{(1,1)} \to W_{(1,0)}\to (\mf{g}_2)_\C $\\
& orthogonal to $p_1$ and $p_2$\\
\hline 
$p_4= p^{(1,0)}_{(0,1)}$&$ V_{(1,0)}\otimes V_{(0,1)}  \to V_{(1,0)} \to W_{(0,1)}\to(\mf{g}_2)_\C $\\
& $(u+adt) \otimes (\alpha + v\wedge dt + \frac{1}{2}v\lrcorner \text{Re}\Omega)\mapsto \frac{1}{2}(1-iJ)[u\lrcorner \alpha -\frac{1}{2}(u\wedge v)\lrcorner \text{Re}\Omega  -av]$ \\
\hline 
$p_5= p^{(2,0)}_{(0,1)}$&$V_{(1,0)}\otimes V_{(0,1)} \to V_{(2,0)} \to W_{(0,1)}\to (\mf{g}_2)_\C $\\
& $(u+adt) \otimes (\alpha + v\wedge dt + \frac{1}{2}v\lrcorner \text{Re}\Omega)\mapsto \frac{1}{2}(1-iJ)[u\wedge v + a\alpha -\frac{1}{4}(-(u\lrcorner\alpha)\lrcorner \text{Re} \Omega + av\lrcorner \text{Re} \Omega)]$ \\
\hline 
$p_6=p^{(1,1)}_{(0,1)}$&$ V_{(1,0)}\otimes V_{(0,1)} \to V_{(1,1)} \to W_{(0,1)}\to (\mf{g}_2)_\C $\\
& orthogonal to $p_4$ and $p_5$\\
\hline 
$p_7=p^{(2,0)}_{(1,1)}$&$V_{(1,0)}\otimes V_{(0,1)} \to V_{(2,0)} \to W_{(1,1)} \to (\mf{g}_2)_\C$ \\
&$(u+adt) \otimes (\alpha + v\wedge dt + \frac{1}{2}v\lrcorner \text{Re}\Omega)\mapsto a\alpha +  \frac{1}{3}(u\wedge v) -\frac{1}{2}*_6(\omega \wedge u\wedge v) +\frac{1}{6}*_6(\omega\wedge *_6(\omega \wedge u\wedge v))$ \\
\hline 
$p_8=p^{(1,1)}_{(1,1)}$&$V_{(1,0)}\otimes V_{(0,1)} \to V_{(1,1)} \to W_{(1,1)} \to (\mf{g}_2)_\C$ \\
&orthogonal to $p_7$ and $p_9$ \\
\hline 
$p_9=\tilde{p}^{(1,1)}_{(1,1)}$&$V_{(1,0)}\otimes V_{(0,1)} \to V_{(1,1)} \to W_{(1,1)} \to (\mf{g}_2)_\C$ \\
&orthogonal to $p_7$ and $p_8$ \\
\hline 

\end{tabular}
\end{center}

Here $p^{(1,1)}_{(1,1)}$ and $\tilde{p}^{(1,1)}_{(1,1)}$ factor through two different copies of $W_{(1,1)}$ contained in $V_{(1,1)}.$\\
For the maps $p^{(i,j)}_{(k,l)}$ with $(i,j)\neq (1,1)$ we can determine their expression in terms of $q_i$ as before. The result is the following:
\begin{itemize}
    \item $p_1=-q_1+q_4-\frac{1}{4}q_7$
    \item $p_2=\frac{1}{2}q_1+\frac{1}{2}q_4+\frac{3}{}q_7$
    \item $p_4=-q_2+q_5-\frac{1}{4}q_8$
    \item $p_5=\frac{1}{2}q_2+\frac{1}{2}q_5+\frac{3}{8}q_8 $
    \item $p_7=q_3+q_6+q_9.$
\end{itemize}
By explicitly computing the matrices of the maps $q^{(i,j)(k,l)}_{(m,n)}$ we are able to calculate the norm of each map. They are as follows:
\begin{lemma}
Let $\norm{q}^2\coloneqq \text{Trace}(q^\dagger q),$ then the basis $q^{(ij)(kl)}_{(mn)}$ is orthogonal and 
\begin{itemize}
    \item $\norm{q^{(0,0)(1,0)}_{(1,0)}}^2=3$
    \item $\norm{q^{(0,1)(0,1)}_{(1,0)}}^2=\norm{q^{(1,0)(1,0)}_{(1,0)}}^2=48$
    \item $\norm{q^{(1,0)(1,1)}_{(1,0)}}^2=\norm{q^{(0,1)(1,1)}_{(0,1)}}^2=12$
    \item $\norm{q^{(0,0)(1,1)}_{(1,1)}}^2=8$
    \item $\norm{q^{(1,0)(0,1)}_{(1,1)}}^2=\norm{q^{(0,1)(1,0)}_{(1,1)}}=\frac{16}{3}.$
\end{itemize}
\end{lemma}

By ensuring the maps $p^{(1,1)}_{(k,l)}$ are orthogonal to the other basis maps, we obtain the following :
\begin{itemize}
    \item $p_3=-8q_1-q_4+q_7$
    \item $p_6=-8q_2-q_5+q_8$
    \item $p_8=q_6-q_9$
    \item $ p_9=-\frac{4}{3}q_3+q_6+q_9.$
\end{itemize}
We calculate the matrix of $(D^t_{A_\text{can}})_\gamma$ in the basis $q_1,\cdots ,q_{12}$ in the same way as we have done for the standard representation and find
$$
(D^t_{A_\text{can}})_\gamma = \left(\begin{array}{*{12}c}
3(t-1)&0&0&-4&0&0&8&0&0&0&0&0 \\
0&3(t-1)&0&0&-4&0&0&8&0&0&0&0 \\
0&0&3(t-1)&0&0&1&0&0&1&0&0&0 \\
-1&0&0&0&0&0&-4&0&0&-i&0&0 \\
0&-1&0&0&0&0&0&-4&0&0&i&0 \\
0&0&\frac{3}{2}&0&0&0&0&0&2&0&0&\frac{3i}{2} \\
\frac{1}{2}&0&0&-1&0&0&0&0&0&-\frac{i}{2}&0&0 \\
0&\frac{1}{2}&0&0&-1&0&0&0&0&0&\frac{i}{2}&0 \\
0&0&\frac{3}{2}&0&0&2&0&0&0&0&0&-\frac{3i}{2} \\
0&0&0&4i&0&0&8i&0&0&-3(t-1)&0&0 \\
0&0&0&0&-4i&0&0&-8i&0&0&-3(t-1)&0 \\
0&0&0&0&0&-i&0&0&i&0&0&-3(t-1)
\end{array}\right).
$$

Again a consistency check is carried out by calculating
$$ (D^{\frac{1}{3}}_{A_\text{can}})^2_\gamma=\text{diag}(12,12,7,12,12,7,12,12,7,12,12,7)$$
as expected. By calculating the eigenvalues of the above matrix for $(D^0_{A_\text{can}})_\gamma$ one finds they are as stated in \thref{evalsstdinst}.

\newpage
\bibliographystyle{plain}
\bibliography{research}

\begin{thebibliography}{10}

\bibitem{agricola2015spinorial}
S.G. Agricola I.~Chiossi, T.~Friedrich, and J.~H{\"o}ll.
\newblock {Spinorial description of $SU(3)$ and $G_2$-manifolds}.
\newblock {\em Journal of Geometry and Physics}, 1(98):535--555, 2015.

\bibitem{bar1992dirac}
C.~B{\"a}r.
\newblock {The Dirac operator on homogeneous spaces and its spectrum on
  3-dimensional lens spaces}.
\newblock {\em Archiv der Mathematik}, 59(1):65--79, 1992.

\bibitem{bar1996dirac}
C.~B{\"a}r.
\newblock {The Dirac operator on space forms of positive curvature}.
\newblock {\em Journal of the Mathematical Society of Japan}, 48(1):69--83,
  1996.

\bibitem{bartnik1986mass}
R.~Bartnik.
\newblock The mass of an asymptotically flat manifold.
\newblock {\em Communications on pure and applied mathematics}, 39(5):661--693,
  1986.

\bibitem{baum1990twistor}
H.~Baum, T.~Friedrich, R.~Grunewald, and I.~Kath.
\newblock {Twistor and Killing spinors on Riemannian manifolds}.
\newblock {\em Seminarberichte, Humboldt Universit{\"a}t Sektion Mathematik,
  Berlin}, 108, 1990.

\bibitem{bryant1985metrics}
R.L. Bryant.
\newblock {Metrics with holonomy $G_2$ or Spin (7)}.
\newblock In {\em Arbeitstagung Bonn 1984}, pages 269--277. Springer, 1985.

\bibitem{bryant2003some}
R.L. Bryant.
\newblock {Some Remarks on $G_2$-Structures}.
\newblock {\em arXiv preprint math/0305124}, 2003.

\bibitem{bryant1989construction}
R.L. Bryant and S.M. Salamon.
\newblock On the construction of some complete metrics with exceptional
  holonomy.
\newblock {\em Duke Mathematical Journal}, 58(3):829--850, 1989.

\bibitem{charbonneau2016deformations}
B.~Charbonneau and D.~Harland.
\newblock {Deformations of nearly K{\"a}hler instantons}.
\newblock {\em Communications in Mathematical Physics}, 348(3):959--990, 2016.

\bibitem{choquet1982analysis}
Y.~Choquet-Bruhat, C.~DeWitt-Morette, and M.~Dillard-Bleick.
\newblock {\em {Analysis, manifolds, and physics}}.
\newblock Gulf Professional Publishing, 1982.

\bibitem{clarke2014instantons}
A.~Clarke.
\newblock {Instantons on the exceptional holonomy manifolds of Bryant and
  Salamon}.
\newblock {\em Journal of Geometry and Physics}, 82:84--97, 2014.

\bibitem{donaldson2006mathematical}
S.K Donaldson.
\newblock Gauge theory: Mathematical applications.
\newblock In {\em Encyclopedia of Mathematical Physics}, pages 468 -- 481.
  Academic Press, Oxford, 2006.

\bibitem{donaldson1990geometry}
S.K. Donaldson and P.B. Kronheimer.
\newblock {\em The geometry of four-manifolds}.
\newblock Oxford University Press, 1990.

\bibitem{donaldson2009gauge}
S.K. Donaldson and E.~Segal.
\newblock {Gauge theory in higher dimensions, II}.
\newblock {\em arXiv preprint arXiv:0902.3239}, 2009.

\bibitem{donaldson1998gauge}
S.K. Donaldson and R.P. Thomas.
\newblock {Gauge theory in higher dimensions, in "The Geometric Universe;
  Science, Geometry, And The Work Of Roger Penrose”}, 1998.

\bibitem{mythesis}
J.~Driscoll.
\newblock {\em {Deformations of Asymptotically Conical $G_2$-Instantons}}.
\newblock PhD thesis, University of Leeds, 2020.

\bibitem{eells1985twistorial}
J.~Eells and S.~Salamon.
\newblock Twistorial construction of harmonic maps of surfaces into
  four-manifolds.
\newblock {\em Annali della Scuola Normale Superiore di Pisa-Classe di
  Scienze}, 12(4):589--640, 1985.

\bibitem{fernandez1982riemannian}
M.~Fern{\'a}ndez and A.~Gray.
\newblock {Riemannian manifolds with structure group $G_2$}.
\newblock {\em Annali di matematica pura ed applicata}, 132(1):19--45, 1982.

\bibitem{foscolo2017exoticnk}
L.~Foscolo and M.~Haskins.
\newblock {New $G_2$-holonomy cones and exotic nearly K\"ahler structures on
  $S^6$ and $S^3\times S^3$}.
\newblock {\em Annals of Mathematics}, 185(1):59--130, 2017.

\bibitem{foscolo2018infinitely}
L.~Foscolo, M.~Haskins, and J.~Nordstr{\"o}m.
\newblock {Infinitely many new families of complete cohomogeneity one
  $G_2$-manifolds: $G_2$ analogues of the Taub-NUT and Eguchi-Hanson spaces}.
\newblock {\em Journal of the European Mathematical Society}, to appear.

\bibitem{freed2012instantons}
D.S. Freed and K.~Uhlenbeck.
\newblock {\em Instantons and four-manifolds}, volume~1.
\newblock Springer, 2012.

\bibitem{friedrich1985first}
T.~Friedrich and R.~Grunewald.
\newblock {On the first eigenvalue of the Dirac operator on 6-dimensional
  manifolds}.
\newblock {\em Annals of Global Analysis and Geometry}, 3(3):265--273, 1985.

\bibitem{fulton2013representation}
W.~Fulton and J.~Harris.
\newblock {\em Representation theory: a first course}, volume 129.
\newblock Springer, 2013.

\bibitem{gunaydin1995seven}
M.~G{\"u}naydin and H.~Nicolai.
\newblock {Seven-dimensional octonionic Yang-Mills instanton and its extension
  to an heterotic string soliton}.
\newblock {\em Physics Letters B}, 351(1-3):169--172, 1995.

\bibitem{harland2010yang}
D.~Harland, T.A. Ivanova, O.~Lechtenfeld, and A.D. Popov.
\newblock {Yang-Mills flows on nearly K{\"a}hler manifolds and
  $G_2$-instantons}.
\newblock {\em Communications in Mathematical Physics}, 300(1):185--204, 2010.

\bibitem{karigiannis2012deformation}
S.~Karigiannis and J.D. Lotay.
\newblock {Deformation theory of $G_2$ conifolds}.
\newblock {\em Communications in Analysis and Geometry}, 28:to appear, 2020.

\bibitem{lawson2016spin}
H.B. Lawson and M-L. Michelsohn.
\newblock {\em Spin Geometry}, volume~38.
\newblock Princeton university press, 2016.

\bibitem{lockhart1985elliptic}
R.B. Lockhart and R.C. McOwen.
\newblock Elliptic differential operators on noncompact manifolds.
\newblock {\em Annali della Scuola Normale Superiore di Pisa-Classe di
  Scienze}, 12(3):409--447, 1985.

\bibitem{lotay2009asymptotically}
J.D. Lotay.
\newblock Asymptotically conical associative 3-folds.
\newblock {\em The Quarterly Journal of Mathematics}, 62(1):131--156, 2009.

\bibitem{lotay2009deformation}
J.D. Lotay.
\newblock Deformation theory of asymptotically conical coassociative 4-folds.
\newblock {\em Proceedings of the London Mathematical Society}, 99(2):386--424,
  2009.

\bibitem{lotay2018su2insantons}
J.D. Lotay and G.~Oliveria.
\newblock {$SU(2)^2$-invariant $G_2$-instantons}.
\newblock {\em Mathematische Annalen}, 371:961--1011, 2018.

\bibitem{marshal2002deformations}
S.P. Marshall.
\newblock {\em Deformations of special Lagrangian submanifolds}.
\newblock PhD thesis, University of Oxford, 2002.

\bibitem{moroianu2010hermitian}
A.~Moroianu and U.~Semmelmann.
\newblock {The Hermitian Laplace operator on nearly K{\"a}hler manifolds}.
\newblock {\em Communications in Mathematical Physics}, 294(1):251--272, 2010.

\bibitem{nakajima1990moduli}
H.~Nakajima.
\newblock {Moduli spaces of anti-self-dual connections on ALE gravitational
  instantons}.
\newblock {\em Inventiones Mathematicae}, 102(1):267--303, 1990.

\bibitem{oliveira2014monopoles}
G.~Oliveira.
\newblock {Monopoles on the Bryant--Salamon $G_2$-manifolds}.
\newblock {\em Journal of Geometry and Physics}, 86:599--632, 2014.

\bibitem{taubes1987gauge}
C.H. Taubes.
\newblock Gauge theory on asymptotically periodic 4-manifolds.
\newblock {\em Journal of Differential Geometry}, 25(3):363--430, 1987.

\bibitem{wang1958invariant}
H-C. Wang.
\newblock On invariant connections over a principal fibre bundle.
\newblock {\em Nagoya Mathematical Journal}, 13:1--19, 1958.

\bibitem{wolf1968homogeneous1}
J.A. Wolf and A.~Gray.
\newblock {Homogeneous spaces defined by Lie group automorphisms. I}.
\newblock {\em Journal of Differential Geometry}, 2:77--114, 1968.

\bibitem{xu2009instantons}
F.~Xu.
\newblock {On instantons on nearly K\"ahler 6-manifolds}.
\newblock {\em Asian Journal of Mathematics}, 13(4):535--568, 2009.

\end{thebibliography}

\end{document}